\documentclass[12pt]{article}
\usepackage{amsfonts}
\usepackage{amsthm}
\usepackage{amsmath}    % need for subequations
\usepackage{bm}         % defines commands to access bold math symbols
\usepackage{graphicx}   % need for figures
\usepackage{verbatim}   % useful for program listings
\usepackage{color}      % use if color is used in text
\usepackage{subfigure}  % use for side-by-side figures
\usepackage{epsfig}
\usepackage{tikz}       % need to draw pictures
\usepackage{stmaryrd}
\usepackage{hyperref}   % use for hypertext links, including those to external documents and URLs
\numberwithin{equation}{section} % need to be defined here otherwise the command "`subequations"' leads to errors
\usepackage{cleveref} % useful for referring to theorems, lemmas, ...
\usepackage{booktabs}
\usepackage{caption}

\newtheorem{theorem}{Theorem}[section]
\newtheorem{lemma}[theorem]{Lemma}
\newtheorem{definition}[theorem]{Definition}
\newtheorem{prop}[theorem]{Proposition}
\newtheorem{corollary}[theorem]{Corollary}
\theoremstyle{definition}
\newtheorem{remark}{Remark}%[theorem]

\newtheorem{problem}{Problem}

\allowdisplaybreaks

\newcommand{\norm}[1]{\left\Vert#1\right\Vert}
\newcommand{\abs}[1]{\left\vert#1\right\vert}

\newcommand{\average}[1]{\ensuremath{\{\!\!\{#1\}\!\!\}} }
\newcommand{\jump}[1]{\ensuremath{[\![#1]\!]} }

\newcommand{\refE}{K_{\text{ref}}}

\newcommand{\n}{\mathbf{n}}

\newcommand {\R}{\mathbb{R}}

\begin{document}

\title{Arbitrary Lagrangian-Eulerian discontinuous Galerkin method for  conservation laws on moving simplex meshes
}
\author{Pei Fu\footnote{ School of Mathematical Sciences, University of Science and Technology of China, Hefei, Anhui 230026, P.R. China.
  Email: sxfp2013@mail.ustc.edu.cn. },
Gero Schn\"ucke\footnote{ University of Cologne, Weyertal 86-90, 50931 K\"oln, Email: gschnuec@math.uni-koeln.de},
 Yinhua Xia\footnote{Corresponding author.
School of Mathematical Sciences, University of Science and Technology of China,
  Hefei, Anhui 230026, P.R. China.
  Email: yhxia@ustc.edu.cn. Research supported by NSFC grants No. 11871449 and No. 11471306, and a grant from the Science \& Technology on Reliability \& Environmental Engineering Laboratory (No. 6142A0502020817).
	} }
\date{ }
\maketitle

\begin{abstract}
In Klingenberg, Schn\"ucke and Xia (Math. Comp. 86 (2017), 1203-1232) an arbitrary Lagrangian-Eulerian discontinuous Galerkin (ALE-DG) method to solve conservation laws has been developed and analyzed. In this paper, the ALE-DG method will be extended to several dimensions. The method will be designed for simplex meshes. This will ensure that the method satisfies the geometric conservation law, if the accuracy of the time integrator is not less than the value of the spatial dimension. For the semi-discrete method the $\mathrm{L}^2$-stability will be proven. Furthermore, an error estimate which provides the suboptimal ($k+\frac{1}{2}$) convergence with respect to the $\mathrm{L}^{\infty}\left(0,T;\mathrm{L}^{2}\left(\Omega\right)\right)$-norm will be presented, when an arbitrary monotone flux is used and for each cell the approximating functions are given by polynomials of degree $k$. The two dimensional fully-discrete explicit method will be combined with the bound preserving limiter developed by Zhang, Xia and Shu in (J. Sci. Comput. 50 (2012), 29-62). This limiter does not affect the high order accuracy of a numerical method. Then, for the ALE-DG method revised by the limiter the validity of a discrete maximum principle will be proven. The numerical stability, robustness and accuracy of the method will be shown by a variety of two dimensional computational experiments on moving triangular meshes.
\vfill
\end{abstract}

% Introduction -----------------------------------------------------------------------------------------------------------------
\section{Introduction}
The present paper investigates the development and analysis of an arbitrary Lagrangian-Eulerian discontinuous Galerkin (ALE-DG) method for scalar conservation laws in several space dimensions
\begin{align}\label{eqn:CL}
\partial_{t}u+\nabla\cdot\boldsymbol{f}\left(u\right)=0, \ \text{in }\Omega \times \left(0,T\right),
\end{align}
with initial condition $u_0({\bf x})$ and suitable boundary conditions. The domain $\Omega\subseteq\R^{d}$ is an open convex polyhedron and the flux $\boldsymbol{f}\left(u\right):=\left(f_{1}\left(u\right),...,f_{d}\left(u\right)\right)^{T}$ is a suitable vector field. In general the problem \eqref{eqn:CL} has no classical solutions. Discontinuities like shock waves could appear in the solution, regardless of the smoothness of the initial data. Hence, the problem needs to be investigated with a class of generalized solutions. The existence of an unique physical relevant solution for the problem was proven by Kru\v{z}kov in \cite{Kruzkov1970}. This solution is called entropy solution. In particular, Kru\v{z}kov proved that the unique entropy solution satisfies the maximum principle. This means the entropy solution is bounded by the interval $\left[m,M\right]$, where
\begin{equation}\label{bounds}
M:=\underset{{\bf x}\in\Omega}{\max}\,u_{0}\left({\bf x}\right)\quad\text{and}\quad m:=\underset{{\bf x}\in\Omega}{\min}\,u_{0}\left({\bf x}\right).
\end{equation}

The discontinuous Galerkin (DG) method, introduced by Reed and Hill \cite{Reed1973} in the context of a neutron transport equation, is a finite element method with discontinuous basis functions. The choice of discontinuous basis functions gives the method a local structure (elements only communicate with immediate neighbors) and the property to handle complex mesh geometries. These features of the DG method are attractive for parallel and high performance computing. Therefore, in particular, the explicit Runge-Kutta DG (RK-DG) method for convection-dominated problems developed and analyzed by Cockburn, Shu and several co-authors in a series of publications (cf. the review article \cite{Cockburn2001} for a summation of their pioneering works) became very popular in the last decades.

The RK-DG method of Cockburn and Shu was developed for a static computational mesh, but in engineering applications like aeroelastic computations of wings (cf. for instance Robinson et al. \cite{Robinson1991}) numerical methods with a deformable moving mesh are desirable. Nevertheless, a deformable computational domain can lead to strong distortions in the mesh geometry which can be the source of numerical artifacts and instabilities. In the arbitrary Lagrangian-Eulerian (ALE) approach the mesh can move with the fluid like in the Lagrangian specification or the mesh can be static as in the Eulerian specification. This flexibility has a stabilizing effect on an ALE method, since it is possible to switch to the Eulerian specification whenever distortions appear in the mesh geometry. The ALE kinematics were rigorously described by Donea et al. \cite{Donea2004}. Moreover, in the literature there are different strategies to combine the ALE approach with the RK-DG method. Among others Lomtev, Kirby, Karniadakis \cite{Lomtev1999JCP}, Nguyen \cite{Nguyen2010JFS}, Persson et al. \cite{Persson2009, Persson2015}, Kopriva et al. \cite{Kopriva2016, Minoli2011} and Boscheri, Dumbser \cite{Boscheri2017} developed and analyzed ALE-DG methods for convection-dominated problems on a moving domain.

In this paper, an ALE-DG method for solving the problem \eqref{eqn:CL} on moving simplex meshes is introduced. This method is an extension of the ALE-DG method developed by Klingenberg et al. \cite{Klingenberg2016, KlingenbergHJ2016HJ}. In order to describe the ALE kinematics, we assume that the distribution of the grid points is explicitly given for an upcoming time level by a suitable moving grid methodology. On the basis of this assumption, we can define local affine linear ALE mappings which connect the time-dependent simplex cells with a time-independent reference simplex cell. This simple construction of the ALE mappings ensures that our ALE-DG method has a local structure like the RK-DG method and the discrete geometric conservation law (D-GCL) is satisfied, when a suitable high order accurate Runge-Kutta (RK) method is used. The geometric conservation law (GCL) describes the time evolution of the metric terms in a grid deformation method and has an important influence on the stability and accuracy of a method. The significance of the GCL was first analyzed by Lombard and Thomas \cite{Lombard1979}, thenceforth the GCL was investigated in the context of moving mesh finite volume and finite element methods by Farhat et al. \cite{Farhat2001,Farhat2000,Lesoinne1996}, Mavriplis, Yang \cite{Mavriplis2006} and \'{E}tienne, Garon, Pelletier \cite{Etienne2009}.

Besides the D-GCL, a discrete maximum principle is discussed for our ALE-DG method. In general, even on a static mesh, it is not easy to design a high order method which satisfies a discrete maximum principle for the problem \eqref{eqn:CL} without affecting the high order accuracy of the method. Two approaches are commonly used in the literature. The first approach is the flux correction approach. Based on this approach Xu \cite{Xu2014} developed a technique to ensure that a high order method satisfies the maximum principle and maintains the high order accuracy. Moreover, an algebraic flux correction approach for finite element methods was introduced by Kuzmin in \cite{Kuzmin2001} and \cite[Chapter 4]{Kuzmin2010}. Another approach was developed by Zhang and Shu \cite{Zhang2010}. This approach based on a bound preserving limiter which does not affect the high order accuracy of a high order method. In particular, the bound preserving limiter was developed for rectangular meshes by Zhang, Shu \cite{Zhang2010} and for triangular meshes by Zhang, Xia, Shu \cite{Zhang2012}. This approach allows the development of high order accurate maximum principle satisfying schemes by a simple investigation of the forward Euler step, since the common convexity argument (cf. Gottlieb and Shu \cite{Gottlieb1998}) can be used to extend the result for the forward Euler step to the high order total-variation-diminishing RK (TVD-RK) methods. However, Farhat, Geuzaine and Grandmont \cite{Farhat2001} proved that for ALE finite volume methods a discrete maximum principle is satisfied, if and only if the D-GCL is satisfied. Unfortunately, for our ALE-DG method the D-GCL is only fulfillment, if the accuracy of the RK method corresponds with the spatial dimension. Hence, we cannot expect that our forward Euler ALE-DG method satisfies a discrete maximum principle, when the bound preserving limiter is applied. Nevertheless, it turns out that the GCL is an ODE in our ALE-DG method. This ODE can be solved exactly by a RK method with an order not less than the value of the spatial dimension. In two dimensions, we use the second and the third order TVD-RK methods developed by Shu in \cite{Shu} to solve the GCL and the actual ALE-DG method. Then the RK stage solutions for the GCL  are used to update the metric terms in the RK stages of the actual ALE-DG method. This time integration strategy allows to develop second and third order accurate fully-discrete ALE-DG methods. We prove that these methods satisfy a discrete maximum principle when the bound preserving limiter is applied. Furthermore, we present numerical experiments which support the expectation that the ALE-DG method also satisfies a discrete maximum principle when the five stage fourth order TVD-RK method developed by Spiteri and Ruuth in \cite{Spiteri} and the bound preserving limiter are used.

In addition, we present an a priori error estimate for our ALE-DG method. A priori error analysis to smooth solutions of the second and third order RK-DG method on a static mesh for scalar and symmetrizable systems of conservation laws were mainly done by Zhang, Shu et al. More precisely, Zhang and Shu proved in \cite{QZhang2004} and \cite{QZhang2006} that under a slightly more restrictive Courant-Friedrichs-Lewy (CFL) constraint than the commonly used constraint the a priori error of the second order RK-DG method behaves as $\mathcal{O}\left(\triangle t^{2}+h^{k+\frac{1}{2}}\right)$ in the $\mathrm{L}^{2}$-norm, when a local polynomial basis of degree $k\geq 1$ and an arbitrary monotone flux are applied. In this context the quantity $\triangle t$ denotes the time step and $h$ denotes the maximum cell length. Likewise, Zhang and Shu proved in \cite{QZhang2010} that under the usual CFL constraint the a priori error of the third order RK-DG method for scalar conservation laws behaves as $\mathcal{O}\left(\triangle t^{3}+h^{k+\frac{1}{2}}\right)$ in the $\mathrm{L}^{2}$-norm. This result was extended to symmetrizable systems of conservation laws by Luo, Shu, and Zhang \cite{Luo2015}. Furthermore, a priori error estimates for the third order RK-DG method in the context of linear scalar conservation Laws with discontinuous initial data were proven in \cite{QZhang2014}. Error estimates for fully discrete ALE-DG methods to solve linear conservation laws were proven by Zhou, Xia and Shu in \cite{Zhou2018}. In this work, we merely prove that for smooth solutions of the problem \eqref{eqn:CL} the a priori error of the semi-discrete ALE-DG method behaves as $\mathcal{O}\left(h^{k+\frac{1}{2}}\right)$, when polynomials of degree $k\geq \max\left\{ 1,\frac{d}{2}\right\}$ are used on the reference cell and an arbitrary monotone flux is applied.

The rest of the paper is organized as follows. In Section 2, we introduce the local affine linear ALE mappings, a time-dependent test function space and our semi-discrete ALE-DG method. In Section 3, we present some theoretical results for the semi-discrete ALE-DG method. In particular, the $\mathrm{L}^{2}$-stability is proven. Afterward, in Section 4, the fully-discrete ALE-DG method is investigated. We prove that the D-GCL is satisfied under certain conditions which are related to the spatial dimension. Furthermore, in two dimensions second and third order accurate fully-discrete ALE-DG methods on moving triangular meshes are presented. We prove that these methods satisfy a discrete maximum principle, when the bound preserving limiter for triangular meshes developed by Zhang, Xia, Shu in \cite{Zhang2012} is applied. In Section 5, we validate the theoretical results by some computational examples and show that the ALE-DG method is numerically stable and high order accurate. Finally, we give some concluding remarks in Section 6.

\subsubsection*{Constants and notation}
In the present paper, vectors, vector valued functions and matrices are denoted by bold letters. Scalar quantities are denoted by regular letters. The set $K\left(t\right)$ denotes a time-dependent open simplex cell in a $d$ dimensional domain with the edges $F_{K\left(t\right)}^{\nu}$, $\nu=1,\dots,d+1$. Volume integrals with respect to the open set $K\left(t\right)$ and surface integrals with respect to the edges $F_{K\left(t\right)}^{\nu}$, $\nu=1,\dots,d+1$, are denoted by the bracket notation. Hence, for all $v,w\in\mathrm{L}^{2}\left(K\left(t\right)\right)\cup \mathrm{L}^{2}\left(\partial K\left(t\right)\right)$ and $\nu=1,\dots,d+1$ the notations $\left(v,w\right)_{K\left(t\right)}:=\int_{K\left(t\right)}vw \, d\textbf{x}$, $\left\langle v,w\right\rangle _{F_{K\left(t\right)}^{\nu}}:=\int_{F_{K\left(t\right)}^{\nu}}vw\,d\boldsymbol{\Gamma}$ and $\left\langle v,w\right\rangle _{\partial K\left(t\right)}:=\sum_{\nu=1}^{d+1}\left\langle v,w\right\rangle _{F_{K\left(t\right)}^{\nu}}$ are applied. Furthermore, to avoid confusion with different constants, we denote by $C$ a positive constant, which is independent of the mesh size and the numerical solutions for the conservation law \eqref{eqn:CL}, but it may depend on the solution of the PDE and may have a different value in each occurrence.

% The ALE-DG discretization ----------------------------------------------------------------------------------------------------
\section{The ALE-DG discretization}
In this section, we present the semi-discrete ALE-DG discretization of the problem \eqref{eqn:CL}. At first, the ALE framework to derive the ALE-DG method in several dimensions is briefly listed. Afterward, the ALE framework is used to derive the semi-discrete ALE-DG method for solving the problem \eqref{eqn:CL}.

\subsection{The ALE-DG setting}\label{Sec:TheALE-DGSetting}
In this section, we present the time dependent cells and introduce some identities for the metric quantities to transform derivatives on a reference cell.

\subsubsection{The time-dependent simplex mesh}
We assume that there exists a regular mesh $\mathcal{T}_{(t_n)}$ of simplices at any time level $t_{n}$, $n =0,\dots,\mathcal{N}$, which covers exactly the convex polyhedron domain $\Omega$ such that
\[
\overline{\Omega}=\bigcup\left\{ \overline{K(t_{n})}\mid\ K(t_{n})\in\mathcal{T}_{(t_{n})}\right\}.
\]
The mesh topology of $\mathcal{T}_{(t_n)}$ and $\mathcal{T}_{(t_{n+1})}$ is assumed to be the same. This means:
\begin{itemize}
\item[a)] $\mathcal{T}_{(t_n)}$ and $\mathcal{T}_{(t_{n+1})}$ are simplex meshes of the domain $\Omega$.
\item[b)] $\mathcal{T}_{(t_n)}$ and $\mathcal{T}_{(t_{n+1})}$ have the same number of cells.
\item[c)] The cells of both simplex meshes are positively oriented with respect to the reference simplex
\begin{equation}\label{ReferenceCell}
\refE:=\left\{ \boldsymbol{\xi}=\left(\xi_{1},...,\xi_{d}\right)^{T}\in\R^{d}:\ \xi_{\nu}\geq0,\ \forall\nu,\ \text{and}\ \sum_{\nu=1}^{d}\xi_{\nu}\leq1\right\}.
\end{equation}
\end{itemize}		
The $d+1$ vertices of each simplex $K(t_n)\in\mathcal{T}_{(t_n)}$ are denoted by ${\bf v}_{1}^{n},\dots,{\bf v}_{d+1}^{n}$. We define for all $t\in \left[t_{n},t_{n+1}\right]$ and $\ell =1,...,d+1$ time-dependent straight lines
\begin{equation}\label{vertices}
{\bf v}_{\ell}\left(t\right):={\bf v}_{\ell}^{n}+\boldsymbol{\omega}_{K^{n},\ell}\left(t-t_{n}\right),\qquad\boldsymbol{\omega}_{K^{n},\ell}:=\frac{1}{\triangle t}\left({\bf v}_{\ell}^{n+1}-{\bf v}_{\ell}^{n}\right).
\end{equation}
These straight lines are for any $t\in \left[t_{n},t_{n+1}\right]$ the vertices of a time-dependent simplex cell given by
\begin{align}\label{cells}
\begin{split}
K\left(t\right):=&\text{int}\left(\text{conv}\left\{ {\bf v}_{1}\left(t\right),\dots,{\bf v}_{d+1}\left(t\right)\right\} \right),
\quad \partial K\left(t\right)=\bigcup_{\nu=1}^{d+1}F_{K\left(t\right)}^{\nu},  \\
& F_{K\left(t\right)}^{\nu}:=\text{conv}\left(\left\{ {\bf v}_{1}\left(t\right),\dots,{\bf v}_{d+1}\left(t\right)\right\} \setminus\left\{ {\bf v}_{\nu}\left(t\right)\right\} \right),
\end{split}
\end{align}
where $\text{int}\left(\cdot\right)$ and $\text{conv}\left(\cdot\right)$ denote the interior and the convex hull of a set. In the following, the set of all time-dependent cells $K(t)$ is denoted by $\mathcal{T}_{\left(t\right)}$.
Furthermore, for any cell $K\left(t\right)\in\mathcal{T}_{(t)}$ the diameter of the cell and the radius of the largest ball, contained in $K\left(t\right)$, are denoted by $h_{K\left(t\right)}$ as well as $\rho_{K\left(t\right)}$. Additionally, we define the following global length
\begin{equation}\label{hmax}
h:=\underset{t\in\left[0,T\right]}{\max}\underset{K\left(t\right)\in\mathcal{T}_{\left(t\right)}}{\max}h_{K\left(t\right)}.
\end{equation}
Henceforth, we assume:
\begin{enumerate}
\item[(A1)] The domain $\overline{\Omega}$ is for all  $t\in \left[0,T\right]$ exactly covered by the time-dependent cells \eqref{cells} such that $\overline{\Omega}=\bigcup_{K\left(t\right)\in\mathcal{T}_{\left(t\right)}}\overline{K\left(t\right)}$.
\item[(A2)] For all  $t\in \left[0,T\right]$ and all cells $K\left(t\right)\in\mathcal{T}_{\left(t\right)}$ $J_{K\left(t\right)}=\text{det}\left(\mathrm{{\bf A}}_{K\left(t\right)}\right)>0$.
\item[(A3)] It exists constants $\kappa>0$ and $\tau>0$, independent of $h$, such that for all  $t\in \left[0,T\right]$
 \[
h_{K\left(t\right)}\leq\kappa\rho_{K\left(t\right)}\quad\text{and}\quad h \leq\tau h_{K\left(t\right)},\quad\forall K\left(t\right)\in\mathcal{T}_{\left(t\right)}.
 \]
\end{enumerate}

\subsubsection{The ALE mapping and the grid velocity field}\label{Sec:TheALEMappingGridVelocityField}
The time-dependent simplex cells \eqref{cells} can be mapped to the time-independent reference simplex element \eqref{ReferenceCell} by the affine linear time-dependent mapping
\begin{equation}\label{mapping}
\boldsymbol{\chi}_{K\left(t\right)}:\refE\to\overline{K\left(t\right)},\quad\boldsymbol{\xi}\mapsto\boldsymbol{\chi}_{K\left(t\right)}\left(\boldsymbol{\xi},t\right):=\mathrm{{\bf A}}_{K\left(t\right)}\boldsymbol{\xi}+{\bf v}_{1}\left(t\right),
\end{equation}
where the matrix $\mathrm{{\bf A}}_{K\left(t\right)}$ is given by
\begin{equation}\label{matrix2}
\mathrm{{\bf A}}_{K\left(t\right)}:=\left({\bf v}_{2}\left(t\right)-{\bf v}_{1}\left(t\right),\dots,{\bf v}_{d+1}\left(t\right)-{\bf v}_{1}\left(t\right)\right).
\end{equation}
We note that the matrix $\mathrm{{\bf A}}_{K\left(t\right)}$ is the Jacobian matrix of the mapping $\boldsymbol{\chi}_{K\left(t\right)}$ and the corresponding determinant is
\begin{equation}\label{Jacobian}
J_{K\left(t\right)}=\text{det}\left(\mathrm{{\bf A}}_{K\left(t\right)}\right)=d!\left|K\left(t\right)\right|,
\end{equation}
where $\left|K\left(t\right)\right|$ denotes the volume of the cell $K\left(t\right)$. In particular, $J_{K\left(t\right)}$ is independent of the spatial variables and belongs to $P^{d}\left(\left[t_{n-1},t_{n}\right]\right)$. It is worth to mention that in general for non-simplicial moving meshes $J_{K\left(t\right)}$ depends on spatial and temporal variables, since the shape of the elements can change when the corners move with different speed. In Figure \ref{motion} the two dimensional situation for a triangular element and a rectangular element is illustrated. The implementation of metric quantities  which depend on spatial variables is not easy and requires caution (e.g. cf. Kopriva  \cite{Kopriva2006}).
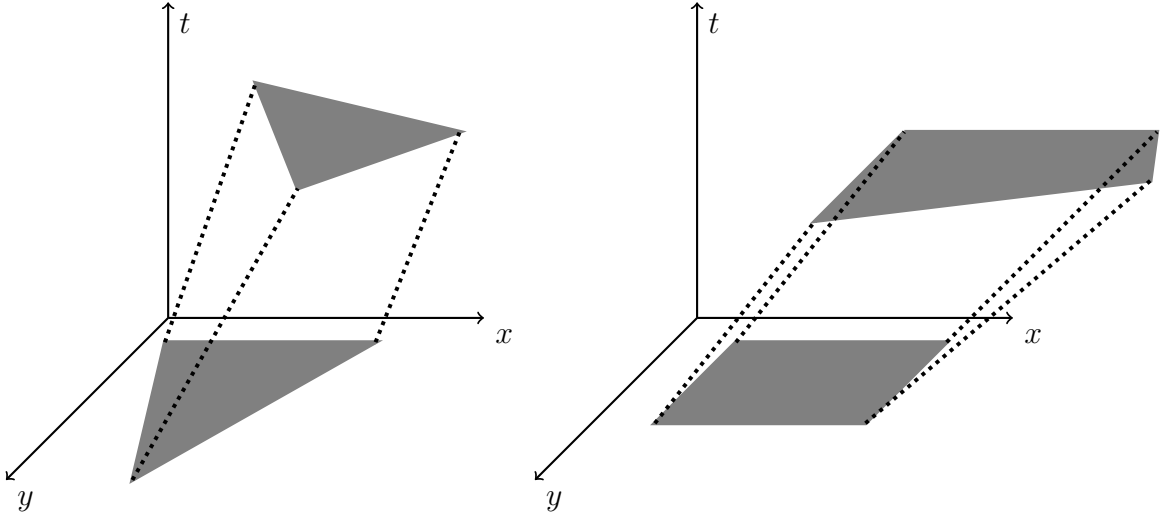
\begin{figure}
\begin{center}
\begin{tikzpicture}[scale=2.8]
\draw[fill=black,opacity=0.5,line width=1.5pt]
(0.1, 0 , 0.3) -- (0.6, 0 , 2) -- (1.1, 0, 0.3)  -- cycle;
\draw[line width=1.5pt, style=dotted]
(0.1, 0 , 0.3) -- (0.3, 1 , -0.3);
\draw[line width=1.5pt, style=dotted]
(0.6, 0 , 2) -- (1, 1 , 1);
\draw[line width=1.5pt, style=dotted]
(1.1, 0, 0.3) -- (1.5, 1, 0.3);
\draw[fill=black,opacity=0.5,line width=1.5pt]
(0.3, 1 , -0.3) -- (1, 1 , 1) -- (1.5, 1, 0.3)  -- cycle;
 \draw[thick,->] (0,0,0) -- (1.5,0,0) node[anchor = north west]{$x$};
 \draw[thick,->] (0,0,0) -- (0,1.5,0) node[anchor = north west]{$t$};
 \draw[thick,->] (0,0,0) -- (0,0,2)   node[anchor = north west]{$y$};
\end{tikzpicture}
\begin{tikzpicture}[scale=2.8]
\draw[line width=1.5pt,fill=black,opacity=0.5]
(0.3, 0 , 0.3) -- (1.3, 0 , 0.3) -- (1.3, 0, 1.3) -- (0.3, 0 , 1.3) -- (0.3, 0 , 0.3);
\draw[line width=1.5pt, style=dotted]
(0.3, 0 , 0.3) -- (1.6, 1.5 , 1.6);
\draw[line width=1.5pt, style=dotted]
(1.3, 0 , 0.3) -- (2.8, 1.5 , 1.6);
\draw[line width=1.5pt, style=dotted]
(1.3, 0, 1.3) -- (3.0, 1.5, 2.2);
\draw[line width=1.5pt, style=dotted]
(0.3, 0 , 1.3) -- (1.6, 1.5 , 2.7);
\draw[fill=black,opacity=0.5,line width=1.5pt]
(1.6, 1.5 , 1.6) -- (2.8, 1.5 , 1.6) -- (3.0, 1.5, 2.2) -- (1.6, 1.5 , 2.7) -- (1.6, 1.5 , 1.6);
 \draw[thick,->] (0,0,0) -- (1.5,0,0) node[anchor = north west]{$x$};
 \draw[thick,->] (0,0,0) -- (0,1.5,0) node[anchor = north west]{$t$};
 \draw[thick,->] (0,0,0) -- (0,0,2)   node[anchor = north west]{$y$};
\end{tikzpicture}
\caption{\label{motion} \textbf{Left:} The vertices of a triangle element at the current time level move to the vertices of a triangle element at the next time level.
\textbf{Right:} The corners of a rectangular element at the current time level move to the corners of a trapezoid element at the next time level.}
\end{center}
\end{figure}

Since the matrix ${\bf A}_{K\left(t\right)}$ is the Jacobian matrix of the mapping \eqref{mapping}, %and the corresponding determinant $J_{K\left(t\right)}:=\text{det}\left(\mathrm{{\bf A}}_{K\left(t\right)}\right)$ is merely time-dependent,
we have the following metric transformations
\begin{equation}\label{Piola}
\nabla\cdot\boldsymbol{f}=\nabla_{\boldsymbol{\xi}}\cdot\left[\left({\bf A}_{K\left(t\right)}^{-1}\right)\boldsymbol{f}^{*}\right],  \qquad
\nabla u={\bf A}_{K\left(t\right)}^{-T}\nabla_{\boldsymbol{\xi}}u^{*},
\qquad \n_{K(t)}
%d\boldsymbol{\Gamma}
=J_{K\left(t\right)}{\bf A}_{K\left(t\right)}^{-T}\n_{\refE},
%d\boldsymbol{\Gamma}_{\xi}
\end{equation}
where $\boldsymbol{f}:\R\to\R^{d}$ is an arbitrary vector field  with $\boldsymbol{f}^{*}=\boldsymbol{f}\circ\boldsymbol{\chi}_{K\left(t\right)}$, $u$ is a scalar function with $u^{*}=u\circ\boldsymbol{\chi}_{K\left(t\right)}$, $\n_{K(t)}$ is the normal of the cell $K\left(t\right)$ and $\n_{\refE}$ is the reference normal. A proof of these metric transformations is given in Ciarlet \cite[p. 461]{CiarletNF}. Moreover, the mapping \eqref{mapping} provides the grid velocity field in the point $\textbf{x}=\boldsymbol{\chi}_{K\left(t\right)}\left(\boldsymbol{\xi},t\right)$
\begin{equation}\label{gridvelocity}
\boldsymbol{\omega}_{K\left(t\right)}\left({\bf x},t\right):=\frac{d}{dt}\Big(\boldsymbol{\chi}_{K\left(t\right)}\left(\boldsymbol{\xi},t\right)\Big).
\end{equation}
The definition \eqref{gridvelocity} provides the following relation
\begin{equation}\label{gridvelocity1}
\partial_{\xi_{i}}\omega_{j}\left(\boldsymbol{\chi}_{K\left(t\right)}\left(\boldsymbol{\xi},t\right)\Big),t\right)=\left[\frac{d}{dt}\left(\mathrm{{\bf A}}_{K\left(t\right)}\right)\right]_{ji},\qquad i,j=1,\dots,d,
\end{equation}
where $\omega_{j}$ are the coefficients of the grid velocity and $\left[\frac{d}{dt}\left(\mathrm{{\bf A}}_{K\left(t\right)}\right)\right]_{ji}$ are the coefficients of the matrix $\frac{d}{dt}\left(\mathrm{{\bf A}}_{K\left(t\right)}\right)$. We are also interested to find an identity for the time derivative of the determinant $J_{K\left(t\right)}$. Hence, we apply Jacobi's formula (cf. Bellman \cite{Bellman}), using the equation relating the adjugate of $\mathrm{{\bf A}}_{K\left(t\right)}$ to the inverse $\mathrm{{\bf A}}_{K\left(t\right)}^{-1}$ and apply the identity \eqref{gridvelocity1}. This results in the identity
\begin{align}\label{TimeDerivativeJacobian1}
\begin{split}
\frac{d}{dt}\left(J_{K\left(t\right)}\right)
=& \text{tr}\left[\text{adj}\left(\mathrm{{\bf A}}_{K\left(t\right)}\right)\frac{d}{dt}\left(\mathrm{{\bf A}}_{K\left(t\right)}\right)\right] \\
=&\sum_{i=1}^{d}\sum_{j=1}^{d}\left(-1\right)^{i+j}M_{K\left(t\right)}^{ij}\left[\frac{d}{dt}\left(\mathrm{{\bf A}}_{K\left(t\right)}\right)\right]_{ji}  \\
=&\left(\sum_{i=1}^{d}\sum_{j=1}^{d}\left[\mathrm{{\bf A}}_{K\left(t\right)}^{-1}\right]_{ij}\left(\partial_{\xi_{i}}\omega_{j}\right)\right)J_{K\left(t\right)} \\
=&\left(\sum_{i=1}^{d}\left[\mathrm{{\bf A}}_{K\left(t\right)}^{-1}\left(\partial_{\xi_{i}}\boldsymbol{\omega}\right)\right]_{i}\right)J_{K\left(t\right)},
\end{split}
\end{align}
where $\text{tr}\left(\cdot\right)$ denotes the trace of a matrix, $\text{adj}\left(\cdot\right)$ denotes the adjoint of a matrix and $M_{K\left(t\right)}^{ij}$  is the $(i,j)$ minor of $\mathrm{{\bf A}}_{K\left(t\right)}$. Moreover, since the matrix $\mathrm{{\bf A}}_{K(t)}$ does not depend on spatial variables, we obtain
\begin{equation}\label{TimeDerivativeJacobian}
\frac{d}{dt}\left(J_{K\left(t\right)}\right)=\left(\sum_{i=1}^{d}\left[\mathrm{{\bf A}}_{K\left(t\right)}^{-1}\left(\partial_{\xi_{i}}\boldsymbol{\omega}\right)\right]_{i}\right)J_{K\left(t\right)}=\left(\nabla_{\boldsymbol{\xi}}\cdot\left[\mathrm{{\bf A}}_{K\left(t\right)}^{-1}\boldsymbol{\omega}\right]\right)J_{K\left(t\right)}.
\end{equation}
Finally, we summarize some properties of the grid velocity. These properties will be used in the next sections.
\begin{lemma}\label{PropertiesGridVelocity} The grid velocity $\boldsymbol{\omega}$ defined by \eqref{gridvelocity} has the properties:
\begin{itemize}
\item[(i)] For all $t\in\left[t_{n},t_{n+1}\right]$ the grid velocity belongs to the space  $P^{1}\left(\refE,\R^{d}\right)$.
\item[(ii)] The grid velocity is time-independent for all points contained in the set $\partial{\refE}$.
\item[(iii)] The divergence of the grid velocity satisfies
\begin{equation}\label{DivergenceGridvelocity}
\left(\nabla\cdot\boldsymbol{\omega}\right)J_{K\left(t\right)}=\left(\nabla_{\boldsymbol{\xi}}\cdot\left[\mathrm{{\bf A}}_{K\left(t\right)}^{-1}\boldsymbol{\omega}\right]\right)J_{K\left(t\right)}\in P^{d-1}\left(\left[t_{n-1},t_{n}\right]\right).
\end{equation}
\end{itemize}
\end{lemma}
\begin{proof}
Since the matrix $\frac{d}{dt}\left(\mathrm{{\bf A}}_{K\left(t\right)}\right)$ is time independent, the property \textit{(i)} follows  from the definition of the grid velocity \eqref{gridvelocity}.

The property \textit{(ii)} follows directly from the definitions of the time-dependent straight lines \eqref{vertices}, the cells \eqref{cells} and the grid velocity \eqref{gridvelocity}.

Finally, the property \textit{(iii)} follows from the metric transformations \eqref{Piola} and the identity \eqref{TimeDerivativeJacobian}, since $\left[\frac{d}{dt}\left(\mathrm{{\bf A}}_{K\left(t\right)}\right)\right]_{ij}$ is time independent and the minors $M_{K\left(t\right)}^{ij}$, $i,j=1,\dots,d$, belong to the space $P^{d-1}\left(\left[t_{n-1},t_{n}\right]\right)$.
\end{proof}

% The test function space ----------------------------------------------------------------------------------------------
\subsection{The approximation space}
We define the approximation space
\begin{equation}\label{testfunctions}
\mathcal{V}_{h}(t):=\left\{ v\in\mathrm{L}^{2}\left(\Omega\right):\ v\circ\boldsymbol{\chi}_{K\left(t\right)}\in P^{k}\left(\refE\right),\ \forall K\left(t\right)\in\mathcal{T}_{\left(t\right)}\right\} ,
\end{equation}
where $P^{k}\left(\refE\right)$ denotes the space of polynomials in $\refE$ of degree at most $k$. The functions from the space $\mathcal{V}_{h}(t)$ are discontinuous along the interface of two adjacent cells. Thus, we define for a function $v\in\mathcal{V}_{h}(t)$, an arbitrary cell $K\left(t\right)\in\mathcal{T}_{\left(t\right)}$ and all $\nu =1,...,d+1$ the following limits
\begin{equation*}\label{trace}
v^{\text{int}_{K\left(t\right)}}\left({\bf x}\right):=\underset{\varepsilon\to0^{+}}{\lim}\,v\left({\bf x}-\varepsilon\mathbf{n}_{K\left(t\right)}^{\nu}\right),
\quad
v^{\text{ext}_{K\left(t\right)}}\left({\bf x}\right):=\underset{\varepsilon\to0^{+}}{\lim}\,v\left({\bf x}+\varepsilon\mathbf{n}_{K\left(t\right)}^{\nu}\right),
\quad\forall{\bf x}\in F_{K\left(t\right)}^{\nu},
\end{equation*}
where the vector $\n_{K\left(t\right)}^{\nu}$, $\nu=1,...,d+1$, is the outward normal of the cell $K\left(t\right)$ with respect to the simplex face $F_{K\left(t\right)}^{\nu}$. Then, the cell average and jump of the function $v$ along the simplex face $F_{K\left(t\right)}^{\nu}$ are defined by
\begin{equation*}\label{trace1}
\average{v}:=\frac{1}{2}\left(v^{\text{int}_{K\left(t\right)}}+v^{\text{ext}_{K\left(t\right)}}\right),
\quad
\jump{v}:=v^{\text{ext}_{K\left(t\right)}}-v^{\text{int}_{K\left(t\right)}}.
\end{equation*}

% The semi-discrete ALE-DG method ----------------------------------------------------------------------------------------------
\subsection{The semi-discrete ALE-DG method}
For each cell $K(t) \in \mathcal{T}_{\left(t\right)}$, we approximate the solution $u$ of the problem \eqref{eqn:CL} by the function
\begin{equation}\label{movinggridaproxsolution}
u_{h}\left({\bf x},t\right)=\sum_{j=1}^{r}u_{j}^{K\left(t\right)}\left(t\right)\phi_{j}^{K\left(t\right)}\left({\bf x},t\right),\quad\text{for all }t\in\left[t_{n},t_{n+1}\right)\text{ and }{\bf x}\in K\left(t\right),
\end{equation}
where $r:=\frac{\left(k+d\right)!}{d!k!}$ and $\left\{ \phi_{1}^{K(t)}\left({\bf x},t\right),...,\phi_{r}^{K(t)}\left({\bf x},t\right)\right\} $ is a basis of the space $\mathcal{V}_{h}\left(t\right)$ in the cell $K(t)$.
The coefficients $u_{1}^{K(t)}\!\left(t\right)$,...,$u_{r}^{K(t)}\!\left(t\right)$ in \eqref{movinggridaproxsolution} are the unknowns of the method. In order to determine these coefficients, we plug the function \eqref{movinggridaproxsolution} in \eqref{eqn:CL}, multiply the equation by a test function $v \in \mathcal{V}_{h}(t)$ and use the change of variables theorem for integrals. This results in the equation
 \begin{equation}\label{eqn:weakCLReferenceCell}
\left(J_{K\left(t\right)}\left(\partial_{t}u_{h}\right),v^{*}\right)_{\refE}+\left(J_{K\left(t\right)}\left(\nabla_{x}\cdot\boldsymbol{f}\left(u_{h}\right)\right),v^{*}\right)_{\refE}=0,
\end{equation}
where $v^{*}=v\circ\chi_{K\left(t\right)}$. The chain rule formula and the metric transformations \eqref{Piola} provide
\begin{equation}\label{TimeDerivative}
\frac{d}{dt}u_{h}^{*}=\partial_{t}u_{h}+\boldsymbol{\omega}\cdot\nabla u_{h}=
\partial_{t}u_{h}+\boldsymbol{\omega}\cdot{\bf A}_{K\left(t\right)}^{-T}\nabla_{\boldsymbol{\xi}}u_{h}^{*},
\end{equation}
where $u_{h}^{*}=u_{h}\circ\chi_{K\left(t\right)}$. Next the identities \eqref{TimeDerivativeJacobian} and \eqref{TimeDerivative} provide
\begin{equation}\label{TimeDerivative1}
J_{K\left(t\right)}\left(\partial_{t}u_{h}\right)=\frac{d}{dt}\left(J_{K(t)}u_{h}^{*}\right)-J_{K(t)}\nabla_{\boldsymbol{\xi}}\cdot\left[{\bf A}_{K\left(t\right)}^{-1}\Big(\boldsymbol{\omega}u_{h}^{*}\Big)\right].
\end{equation}
Then the equation \eqref{eqn:weakCLReferenceCell} becomes
\begin{equation}\label{eqn:weakCLReferenceCell1}
\left(\frac{d}{dt}\left(J_{K\left(t\right)}u_{h}^{*}\right),v^{*}\right)_{\refE}+\left(J_{K\left(t\right)}\left(\nabla_{\boldsymbol{\xi}}\cdot\tilde{\boldsymbol{g}}\left(\boldsymbol{\omega},u_{h}^{*}\right)\right),v^{*}\right)_{\refE}=0,
\end{equation}
where $\tilde{\boldsymbol{g}}\left(\boldsymbol{\omega},u_{h}^{*}\right):=\mathrm{{\bf A}}_{K\left(t\right)}^{-1}\Big(\boldsymbol{g}\left(\boldsymbol{\omega},u_{h}^{*}\right)\Big)$ with
\begin{equation}\label{flux:g}
\boldsymbol{g}\left(\boldsymbol{\omega},u\right):=\boldsymbol{f}\left(u\right)-\boldsymbol{\omega}\left({\bf x},t\right)u.
\end{equation}
At this point, we proceed similar as in the derivation of the standard DG method on a static mesh. First, we apply the integration by parts formula in the second integral in \eqref{eqn:weakCLReferenceCell1}. Then, we replace the flux function $\tilde{\boldsymbol{g}}\left(\boldsymbol{\omega},u_{h}^{*,\text{int}_{\refE}}\right)\cdot\n_{\refE}$ in the surface integrals by a numerical flux function $\widehat{g}\left(\boldsymbol{\omega},u_{h}^{*,\text{int}_{\refE}},u_{h}^{*,\text{ext}_{\refE}},J_{K(t)}\tilde{\n}\left(t\right)\right)$ with $\tilde{\n}\left(t\right)=A_{K\left(t\right)}^{-T}\n_{\refE}$. The numerical flux function needs to satisfy certain properties. These properties are discussed in the Section \ref{sec:NumericalFlux}. Finally, on the reference cell, the semi-discrete ALE-DG method appears as the following problem:
% Formulation 1: The semi-discrete ALE-DG method on the reference cell -------------------------------------------------------------------------------
\begin{problem}[The semi-discrete ALE-DG method on the reference cell]\label{ALE-DG2Method}
Find a function $u_{h}\in\mathcal{V}_{h}(t)$, such that for all $v\in \mathcal{V}_{h}(t)$ and all cells $K\left(t\right)\in\mathcal{T}_{\left(t\right)}$ holds
\begin{align}\label{eqn:DJ-CL2}
\begin{split}
 \left(\frac{d}{dt}\left(J_{K\left(t\right)}u_{h}^{*}\right),v^{*}\right)_{\refE}
 =&\left(J_{K\left(t\right)}\tilde{\boldsymbol{g}}\left(\boldsymbol{\omega},u_{h}^{*}\right),\nabla_{\boldsymbol{\xi}}v^{*}\right)_{\refE} \\
& -\left\langle \widehat{g}\left(\boldsymbol{\omega},u_{h}^{*,\text{int}_{\refE}},u_{h}^{*,\text{ext}_{\refE}},J_{K(t)}\tilde{\n}\left(t\right)\right),v^{*,\text{int}_{\refE}}\right\rangle _{\partial\refE},
\end{split}
\end{align}
where $\tilde{\n}\left(t\right)=A_{K\left(t\right)}^{-T}\n_{\refE}$.
\end{problem}
\noindent
Since the test functions $v^{*}=v\circ\chi_{K\left(t\right)}$ are time independent on the reference cell $\refE$, we obtain
\begin{equation}
\left(\frac{d}{dt}\left(J_{K\left(t\right)}u_{h}^{*}\right),v^{*}\right)_{\refE}=\frac{d}{dt}\left(J_{K\left(t\right)}u_{h}^{*},v^{*}\right)_{\refE}=\frac{d}{dt}\left(u_{h},v\right)_{K\left(t\right)}.
\end{equation}
Therefore, on the physical domain,  \textbf{\Cref{ALE-DG2Method}} is equivalent to:
% Formulation 1: The semi-discrete ALE-DG method -------------------------------------------------------------------------------
\begin{problem}[The semi-discrete ALE-DG method]\label{ALE-DG1Method}
Find a function $u_{h}\in\mathcal{V}_{h}(t)$, such that for all $v\in \mathcal{V}_{h}(t)$ and all cells $K\left(t\right)\in\mathcal{T}_{\left(t\right)}$ holds
\begin{align}\label{eqn:DJ-CL1}
\begin{split}
\frac{d}{dt}\left(u_{h},v\right)_{K\left(t\right)}=& \left(\boldsymbol{g}\left(\boldsymbol{\omega},u_{h}\right),\nabla v\right)_{K\left(t\right)} \\
&-\left\langle \widehat{g}\left(\boldsymbol{\omega},u_{h}^{\text{int}_{K(t)}},u_{h}^{\text{ext}_{K(t)}},\n_{K(t)}\right),v^{\text{int}_{K(t)}}\right\rangle _{\partial K(t)}.
\end{split}
\end{align}
\end{problem}

% The numerical flux function ----------------------------------------------------------------------------------------------
\subsection{The numerical flux function}\label{sec:NumericalFlux}
The numerical flux in the ALE-DG method should satisfy:
\begin{enumerate}
\item[(P1)] The flux is consistent with $\boldsymbol{g}\left(\boldsymbol{\omega},u\right)\cdot\n_{K(t)}$.
\item[(P2)] The function $u\mapsto\widehat{g}\left(\boldsymbol{\omega},u,\cdot,\n_{K(t)}\right)$ is increasing and Lipschitz continuous.
\item[(P3)] The function $u\mapsto\widehat{g}\left(\boldsymbol{\omega},\cdot,u,\n_{K(t)}\right)$ is decreasing and Lipschitz continuous.
\item[(P4)] The flux is conservative such that
\[
\widehat{g}\left(\boldsymbol{\omega},u_{h}^{\text{int}_{K(t)}},u_{h}^{\text{ext}_{K(t)}},\n_{K(t)}\right)=-\widehat{g}\left(\boldsymbol{\omega},u_{h}^{\text{ext}_{K(t)}},u_{h}^{\text{int}_{K(t)}},-\n_{K(t)}\right).
\]
\end{enumerate}
\begin{remark}
The properties ($P1$) - ($P4$) of the numerical flux supply for any cell $K(t)\in\mathcal{T}_{\left(t\right)}$ and all $v\in\left[\min\left(u_{1},u_{2}\right),\max\left(u_{1},u_{2}\right)\right]$
\begin{equation}\label{Eflux}
\left(\boldsymbol{g}\left(\boldsymbol{\omega},v\right)\cdot\n_{K(t)}-\widehat{g}\left(\boldsymbol{\omega},u_{1},u_{2},\n_{K(t)}\right)\right)\left(u_{2}-u_{1}\right)\geq0.
\end{equation}
The inequality \eqref{Eflux} is the e-flux condition, which was introduced by Osher \cite{Osher1985}.
\end{remark}
In general every numerical flux with these properties can be used in  the ALE-DG method. A common example is the Lax-Friedrichs flux. This flux is given by
\begin{subequations}\label{Lax-Friedrichs}
\begin{equation}
\widehat{g}\left(\boldsymbol{\omega},u_{1},u_{2},\n_{K(t)}\right):=\widehat{g}_{+}\left(\boldsymbol{\omega},u_{1},\n_{K(t)}\right)-\widehat{g}_{-}\left(\boldsymbol{\omega},u_{2},\n_{K(t)}\right),
\end{equation}
\begin{equation}\label{g+-}
\widehat{g}_{\pm}\left(\boldsymbol{\omega},u,\n_{K(t)}\right):=\frac{1}{2}\left[\lambda^{n}u\pm\boldsymbol{g}\left(\boldsymbol{\omega},u\right)\cdot\n_{K(t)}\right],
\end{equation}
\begin{equation}\label{lambda}
\lambda^{n}:=\max\left\{ \left|\left(\partial_{u}\boldsymbol{g}\left(\boldsymbol{\omega},u\right)\right)\cdot\n_{K(t)}\right|:\ u\in\left[m,M\right],\ t\in\left[t_{n},t_{n+1}\right]\right\}.
%\lambda^{n}:=\max\left\{ \left|\partial_{u}\boldsymbol{g}\left(\boldsymbol{\omega}\left({\bf x},t\right),u\right)\cdot\n_{K(t)}\right|:\ u\in\left[m,M\right],\ t\in\left[t_{n},t_{n+1}\right],\ {\bf x}\in K(t)\right\}.
\end{equation}
\end{subequations}
We note that the functions $\widehat{g}_{\pm}\left(\boldsymbol{\omega},u,\n_{K(t)}\right)$ are increasing in the second argument.

% Theoretical results for the semi-discrete ALE-DG method ----------------------------------------------------------------------
\section{Theoretical results for the semi-discrete method}
In this section, we present some theoretical results for the semi-discrete ALE-DG method. We start with a proof for the $\mathrm{L}^{2}$-stability of the semi-discrete ALE-DG method. This proof requires techniques which were introduced by Jiang and Shu in \cite{Jiang1994MC} to proof a cell entropy inequality for the DG method. We note that for scalar conservation laws the function $\eta\left(u\right)=\frac{1}{2}u^{2}$ is an entropy. Thus, the $\mathrm{L}^{2}$-stability provides also entropy stability in the sense that the total entropy at a certain time point is bounded by the total entropy at initial time.
% L^2-stability ---------------------------------------------------------------------------------------------------------------
%\subsection{\texorpdfstring{$\mathrm{\textbf{L}}^{\textbf{2}}$}{L2}-stability}
\begin{prop}\label{L2stability}
The solution $u_{h}$ of the semi-discrete ALE-DG method satisfies for any $t\in [0,T]$
\begin{equation}\label{stabilityL2}
\norm{u_{h}(t)}_{\mathrm{L}^{2}\left(\Omega\right)}\leq\norm{u_{h}(0)}_{\mathrm{L}^{2}\left(\Omega\right)}
\end{equation}
when the problem \eqref{eqn:CL} is considered with periodic boundary conditions.
\end{prop}
\begin{proof}
Let $K(t)\in\mathcal{T}_{\left(t\right)}$ be an arbitrary cell. We use the ALE-DG solution $u_{h}^{*}=u_{h}\circ\chi_{K\left(t\right)}$ as test function in the equation \eqref{eqn:DJ-CL2} and obtain
\begin{align}\label{stabilityL2:eq1}
\begin{split}
\left(\frac{d}{dt}\left(J_{K\left(t\right)}u_{h}^{*}\right),u_{h}^{*}\right)_{\refE}=& \left(J_{K\left(t\right)}\tilde{\boldsymbol{g}}\left(\boldsymbol{\omega},u_{h}^{*}\right),\nabla_{\boldsymbol{\xi}}u_{h}^{*}\right)_{\refE} \\
&-\left\langle \widehat{g}\left(\boldsymbol{\omega},u_{h}^{*,\text{int}_{\refE}},u_{h}^{*,\text{ext}_{\refE}},J_{K(t)}\tilde{\n}\left(t\right)\right),u_{h}^{*,\text{int}_{\refE}}\right\rangle _{\partial\refE}.
\end{split}
\end{align}
We obtain by the identity \eqref{TimeDerivativeJacobian} and the change of variables theorem for integrals
\begin{align}\label{stabilityL2:eq2}
\begin{split}
\left(\frac{d}{dt}\left(J_{K\left(t\right)}u_{h}^{*}\right),u_{h}^{*}\right)_{\refE}
=& \frac{1}{2}\frac{d}{dt}\left(\left(u_{h}^{*}\right)^{2},J_{K\left(t\right)}\right)_{\refE}+\frac{1}{2}\left(J_{K(t)}\nabla_{\boldsymbol{\xi}}\cdot\left[\left({\bf A}_{K\left(t\right)}^{-1}\right)\boldsymbol{\omega}\right],\left(u_{h}^{*}\right)^{2}\right)_{\refE} \\
=&\frac{1}{2}\frac{d}{dt}\left(u_{h},u_{h}\right)_{K\left(t\right)}+\frac{1}{2}\left(\nabla\cdot\boldsymbol{\omega},u_{h}^{2}\right)_{K\left(t\right)}.
\end{split}
\end{align}
Next, we define the vector valued functions
\begin{equation}\label{primitive1}
\boldsymbol{F}(u):=\left(\int^{u}f_{1}\left(u\right)\,du,...,\int^{u}f_{d}\left(u\right)\,du\right)^{T} \quad \text{and} \quad
\boldsymbol{G}\left(\boldsymbol{\omega},u\right):=\boldsymbol{F}\left(u\right)-\frac{1}{2}\boldsymbol{\omega}u^{2}.
\end{equation}
Then, we obtain by the metric transformations \eqref{Piola}, the functions \eqref{primitive1} and the change of variables theorem for integrals
%\begin{align}\label{stabilityL2:eq3}
%\begin{split}
%& \left(J_{K\left(t\right)}\tilde{\boldsymbol{g}}\left(\boldsymbol{\omega},u_{h}^{*}\right),\nabla_{\boldsymbol{\xi}}u_{h}^{*}\right)_{\refE} \\
%=&\left(J_{K\left(t\right)}\left({\bf A}_{K\left(t\right)}^{-1}\right)\boldsymbol{f}\left(u_{h}^{*}\right),\nabla_{\boldsymbol{\xi}}u_{h}^{*}\right)_{\refE}-\frac{1}{2}\left(J_{K\left(t\right)}\left({\bf A}_{K\left(t\right)}^{-1}\right)\boldsymbol{\omega},\nabla_{\boldsymbol{\xi}}\left(u_{h}^{*}\right)^{2}\right)_{\refE} \\
%=&\left(J_{K\left(t\right)}\left({\bf A}_{K\left(t\right)}^{-1}\right)\boldsymbol{g}\left(\boldsymbol{\omega},u_{h}^{*}\right),\nabla_{\boldsymbol{\xi}}u_{h}^{*}\right)_{\refE}+\frac{1}{2}\left(J_{K\left(t\right)}\nabla_{\boldsymbol{\xi}}\cdot\left[\left({\bf A}_{K\left(t\right)}^{-1}\right)\boldsymbol{\omega}\right],\left(u_{h}^{*}\right)^{2}\right)_{\refE} \\
%=& \left(\boldsymbol{g}\left(\boldsymbol{\omega},u_{h}\right),\nabla u_{h}\right)_{K\left(t\right)}+\frac{1}{2}\left(\nabla\cdot\boldsymbol{\omega},u_{h}^{2}\right)_{K\left(t\right)}.
%\end{split}
%\end{align}
%Furthermore, the functions \eqref{primitive1} and the divergence theorem supply
%\begin{equation}\label{stabilityL2c}
%\left(\boldsymbol{g}\left(\boldsymbol{\omega},u_{h}\right),\nabla u_{h}\right)_{K\left(t\right)}=\left(\nabla\cdot\boldsymbol{G}\left(\boldsymbol{\omega},u_{h}\right),1\right)_{K\left(t\right)}=\sum_{\nu=1}^{d+1}\left\langle \boldsymbol{G}\left(\boldsymbol{\omega},u_{h}^{\text{int}_{K(t)}}\right)\cdot\n_{K\left(t\right)}^{\nu},1\right\rangle _{F_{K(t)}^{\nu}},
%\end{equation}

\begin{align}\label{stabilityL2:eq3}
\begin{split}
& \left(J_{K\left(t\right)}\tilde{\boldsymbol{g}}\left(\boldsymbol{\omega},u_{h}^{*}\right),\nabla_{\boldsymbol{\xi}}u_{h}^{*}\right)_{\refE}
-\frac{1}{2}\left(J_{K(t)}\nabla_{\boldsymbol{\xi}}\cdot\left[\left({\bf A}_{K\left(t\right)}^{-1}\right)\boldsymbol{\omega}\right],\left(u_{h}^{*}\right)^{2}\right)_{\refE}\\
=&\left(J_{K\left(t\right)}\left({\bf A}_{K\left(t\right)}^{-1}\right)\boldsymbol{f}\left(u_{h}^{*}\right),\nabla_{\boldsymbol{\xi}}u_{h}^{*}\right)_{\refE}-\frac{1}{2}\left(J_{K\left(t\right)}\left({\bf A}_{K\left(t\right)}^{-1}\right)\boldsymbol{\omega},\nabla_{\boldsymbol{\xi}}\left(u_{h}^{*}\right)^{2}\right)_{\refE} \\
&-\frac{1}{2}\left(J_{K(t)}\nabla_{\boldsymbol{\xi}}\cdot\left[\left({\bf A}_{K\left(t\right)}^{-1}\right)\boldsymbol{\omega}\right],\left(u_{h}^{*}\right)^{2}\right)_{\refE}\\
=&\left(J_{K\left(t\right)}\left({\bf A}_{K\left(t\right)}^{-1}\right)\boldsymbol{f}\left(u_{h}^{*}\right),\nabla_{\boldsymbol{\xi}}u_{h}^{*}\right)_{\refE}-\frac{1}{2}\left(J_{K\left(t\right)}\nabla_{\boldsymbol{\xi}}\cdot\left[\left({\bf A}_{K\left(t\right)}^{-1}\right)\boldsymbol{\omega}\left(u_{h}^{*}\right)^{2}\right],1\right)_{\refE} \\
=& \left(J_{K\left(t\right)}\nabla_{\boldsymbol{\xi}}\cdot\left(\left({\bf A}_{K\left(t\right)}^{-1}\right)\boldsymbol{G}\left(\boldsymbol{\omega},u_{h}^*\right)\right),1\right)_{\refE}=\left(\nabla\cdot\boldsymbol{G}\left(\boldsymbol{\omega},u_{h}\right),1\right)_{K\left(t\right)}\\
%\nabla u_{h}\right)_{K\left(t\right)}+\frac{1}{2}\left(\nabla\cdot\boldsymbol{\omega},u_{h}^{2}\right)_{K\left(t\right)}.
\end{split}
\end{align}
Furthermore,  the divergence theorem gives
\begin{equation}\label{stabilityL2c}
%\begin{split}
%\left(\boldsymbol{g}\left(\boldsymbol{\omega},u_{h}\right),\nabla u_{h}\right)_{K\left(t\right)}-\frac{1}{2}\left(\nabla\cdot\boldsymbol{\omega},u_{h}^{2}\right)_{K\left(t\right)} &=
\left(\nabla\cdot\boldsymbol{G}\left(\boldsymbol{\omega},u_{h}\right),1\right)_{K\left(t\right)}
=\sum_{\nu=1}^{d+1}\left\langle \boldsymbol{G}\left(\boldsymbol{\omega},u_{h}^{\text{int}_{K(t)}}\right)\cdot\n_{K\left(t\right)}^{\nu},1\right\rangle _{F_{K(t)}^{\nu}},
%\end{split}
\end{equation}
where the vectors $\n_{K\left(t\right)}^{\nu}$, $\nu=1,\dots,d+1$,  are the outward normals of the cell $K\left(t\right)$ with respect to the simplex faces $F_{K(t)}^{\nu}$, $\nu=1,\dots,d+1$. Likewise, the metric transformations \eqref{Piola} provide
\begin{align}\label{stabilityL2:eq4}
\begin{split}
& \left\langle \widehat{g}\left(\boldsymbol{\omega},u_{h}^{*,\text{int}_{\refE}},u_{h}^{*,\text{ext}_{\refE}},J_{K(t)}\tilde{\n}\left(t\right)\right),u_{h}^{*,\text{int}_{\refE}}\right\rangle _{\partial\refE} \\
=& \left\langle \widehat{g}\left(\boldsymbol{\omega},u_{h}^{\text{int}_{K\left(t\right)}},u_{h}^{\text{ext}_{K\left(t\right)}},\n_{K\left(t\right)}\right),u_{h}^{\text{int}_{K\left(t\right)}}\right\rangle _{\partial K\left(t\right)} \\
=&\sum_{\nu=1}^{d+1}\left\langle \widehat{g}\left(\boldsymbol{\omega},u_{h}^{\text{int}_{K\left(t\right)}},u_{h}^{\text{ext}_{K\left(t\right)}},\n_{K\left(t\right)}^{\nu}\right),u_{h}^{\text{int}_{K\left(t\right)}}\right\rangle _{F_{K(t)}^{\nu}}.
\end{split}
\end{align}
Next, we rearrange the equation \eqref{stabilityL2:eq1} by applying the identities \eqref{stabilityL2:eq2}, \eqref{stabilityL2:eq3}, \eqref{stabilityL2c} and \eqref{stabilityL2:eq4}. This results in the equation
\begin{align}\label{stabilityL2d}
\begin{split}
\frac{1}{2}\frac{d}{dt}\left(u_{h},u_{h}\right)_{K\left(t\right)}=&\sum_{\nu=1}^{d+1}\left\langle \boldsymbol{G}\left(\boldsymbol{\omega},u_{h}^{\text{int}_{K(t)}}\right)\cdot\n_{K\left(t\right)}^{\nu},1\right\rangle _{F_{K(t)}^{\nu}}  \\
&-\sum_{\nu=1}^{d+1}\left\langle \widehat{g}\left(\boldsymbol{\omega},u_{h}^{\text{int}_{K(t)}},u_{h}^{\text{ext}_{K(t)}},\n_{K\left(t\right)}^{\nu}\right),u_{h}^{\text{int}_{K(t)}}\right\rangle _{F_{K(t)}^{\nu}}.
\end{split}
\end{align}
Then, we sum the equation \eqref{stabilityL2d} over all cells $K(t)\in\mathcal{T}_{(t)}$ and obtain
\begin{align}\label{stabilityL2e}
\begin{split}
\frac{1}{2}\frac{d}{dt}\left\Vert u_{h}\right\Vert _{\mathrm{L}^{2}\left(\Omega\right)}^{2} - \frac{1}{2}\sum_{K\left(t\right)\in\mathcal{T}_{\left(t\right)}}\sum_{\nu=1}^{d+1}\left\langle \jump{\boldsymbol{G}\left(\boldsymbol{\omega},u_{h}\right)\cdot\n_{K\left(t\right)}^{\nu}},1\right\rangle _{F_{K\left(t\right)}^{\nu}}
 \\
+ \frac{1}{2}\sum_{K\left(t\right)\in\mathcal{T}_{\left(t\right)}}\sum_{\nu=1}^{d+1}\left\langle \widehat{g}\left(\boldsymbol{\omega},u_{h}^{\text{int}_{K(t)}},u_{h}^{\text{ext}_{K(t)}},\n_{K\left(t\right)}^{\nu}\right),\jump{u_{h}}\right\rangle _{F_{K\left(t\right)}^{\nu}}&=0,
\end{split}
\end{align}
since we consider the problem \eqref{eqn:CL} with periodic boundary conditions. The definition of the function $\boldsymbol{G}(\boldsymbol{\omega},u)$ in \eqref{primitive1} yields $\partial_{u}\boldsymbol{G}(\boldsymbol{\omega},u)=\boldsymbol{g}(\boldsymbol{\omega},u)$. Thus, for any cell $K(t)\in\mathcal{T}_{(t)}$ the mean value theorem and the e-flux condition \eqref{Eflux} provide for all $\nu=1,\dots,d+1$
\begin{equation}
\left\langle \boldsymbol{g}\left(\boldsymbol{\omega},\theta_{K\left(t\right)}^{\nu}\right)\cdot\n_{K\left(t\right)}^{\nu}-\widehat{g}\left(\boldsymbol{\omega},u^{\text{int}_{K(t)}},u^{\text{ext}_{K(t)}},\n_{K\left(t\right)}^{\nu}\right),\jump{u_{h}}\right\rangle _{F_{K\left(t\right)}^{\nu}}\geq0,
\end{equation}
where $\theta_{K\left(t\right)}^{\nu}$ is a value between $\min\left(u_{h}^{\text{int}_{K(t)}},u_{h}^{\text{ext}_{K(t)}}\right)$ and $\max\left(u_{h}^{\text{int}_{K(t)}},u_{h}^{\text{ext}_{K(t)}}\right)$.
Hence, the inequality \eqref{stabilityL2} follows by integrating the equation \eqref{stabilityL2e} over the interval $[0,t]$.
\end{proof}
% The a priori error estimate  ------------------------------------------------------------------------------------------------
\noindent
Furthermore, for sufficiently smooth solutions of the initial value problem \eqref{eqn:CL}, we have a suboptimal a priori error estimate in the sense of the $\mathrm{L}^{\infty}\left(0,T;\mathrm{L}^{2}\left(\Omega\right)\right)$-norm for the semi-discrete ALE-DG method.
\begin{theorem}\label{errorestimate2D}
Consider the initial value problem \eqref{eqn:CL} with periodic boundary conditions and let $u\in\mathrm{W}^{1,\infty}\left(0,T;\mathrm{H}^{k+1}\left(\Omega\right)\right)$ be the exact solution, the flux function $\boldsymbol{f}\in\mathcal{C}^{3}\left(\R,\R^{d}\right)$ and the grid velocity field $\boldsymbol{\omega}$ be bounded in $\Omega\times\left[0,T\right]$ and have bounded derivatives and $u_{h}$ be the solution of the semi-discrete ALE-DG method. The test function space \eqref{testfunctions} is given by piecewise polynomials of degree $k\geq \max\left\{ 1,\frac{d}{2}\right\}$. Furthermore, ($A1$) - ($A3$) are satisfied and the global length $h$ is given by \eqref{hmax}.
Then, it exists a constant $C$, which depends on the final time $T$ and is independent of $h$ and $u_{h}$, such that %the error function $e_{h}=u-u_{h}$ satisfies
\begin{equation}\label{apriorierror}
\left\Vert u-u_{h}\right\Vert _{\mathrm{L}^{\infty}\left(0,T;\mathrm{L}^{2}\left(\Omega\right)\right)}\leq Ch^{k+\frac{1}{2}}.
\end{equation}
\end{theorem}
\noindent
\Cref{errorestimate2D} can be proven with standard techniques from approximation theory (cf. Ciarlet \cite{CiarletFE}). In particular, a one dimensional proof is given in \cite{Klingenberg2016} and in \cite{KlingenbergHJ2016HJ} two dimensional error analysis for a semi-discrete ALE-DG method to solve the Hamilton-Jacobi equations is given. In addition, error analysis for the fully-discrete ALE-DG methods is given in \cite{Zhou2018}. Since, there are already these publications on error analysis for the ALE-DG method in the literature and the proof  of \Cref{errorestimate2D} vary merely in technical details, we skip the proof in this paper.
\begin{remark}
The proof of \Cref{errorestimate2D} requires the a priori assumption
\begin{equation}\label{Apriori}
\left\Vert u-u_{h}\right\Vert _{\mathrm{L}^{\infty}\left(\Omega\times\left[0,T\right]\right)}\leq h^{\frac{1}{2}}.
\end{equation}
These a priori assumption is not necessary, if the problem \eqref{eqn:CL} is considered with a linear flux function $\boldsymbol{f}\left(u\right)={\bf c}u$, ${\bf c}\in \R^{d}$. In the one-dimensional case, the a priori assumption
\begin{equation}\label{Apriori3}
\left\Vert u-u_{h}\right\Vert _{\mathrm{L}^{\infty}\left(0,T;\mathrm{L}^{2}\left(\Omega\right)\right)}\leq h
\end{equation}
can be applied, since for $d=1$ the inequality \eqref{Apriori3} and the inverse inequality \cite[Theorem 3.2.6.]{CiarletFE} provide the a priori assumption \eqref{Apriori}. The assumption \eqref{Apriori3} was applied in \cite{QZhang2004, Xu2007}. Moreover, the statement of \cref{errorestimate2D} can be also proven, for flux functions with less smoothness. Nevertheless, this requires a more restrictive a priori assumption and more restrictive bounds for the parameter $k$. This assumption was applied by Klingenberg et al. in \cite{Klingenberg2016}.
\end{remark}

% The fully-discrete method ----------------------------------------------------------------------------------------------------
\section{The fully-discrete method}
In this section, we discuss the time discretization of the ALE-DG method. In the first part of this section, we prove that the fully-discrete ALE-DG method satisfies  the discrete geometric conservation law (D-GCL). Afterward, in two dimensions, we prove that the second and the third order accurate fully-discrete ALE-DG methods satisfy the maximum principle, when the bound preserving limiter developed by Zhang, Xia and Shu in \cite{Zhang2012} is applied.

% Total-variation-diminishing Runge-Kutta methods ---------------------------------------------------------------------------------------
\subsection{Total-variation-diminishing Runge-Kutta (TVD-RK) methods}
We apply the high order TVD-RK methods developed by Shu in \cite{Shu} for the time discretization, which is also known as strong stability preserving Runge-Kutta (SSP-RK) methods \cite{Gottlieb1998, Gottlieb2001}. For a given ODE ${u}_t=q\left(u,t\right)$, a $s$-stage TVD-RK method can be written as %\mathcal{G}
\begin{align}\label{DiscreteODE}
\begin{split}
u^{n,0}= & u^{n}, \\
u^{n,i}= & \sum_{j=0}^{i-1}\left(\alpha_{ij}u^{n,j}+\triangle t\beta_{ij}q\left(u^{n,j},t_{n+\gamma_{j}}\right)\right), \quad \text{for }i=1,\dots,s, \\
u^{n+1}=& u^{n,s},
\end{split}
\end{align}
where $t_{n+\gamma_{j}}:=t_{n}+\gamma_{j}\triangle t$. The coefficients of the $s$-stage TVD-RK method \eqref{DiscreteODE} need to satisfy
\begin{equation}
0\leq\gamma_{j}\leq1;\qquad\alpha_{ij},\beta_{ij}\geq0;\qquad\alpha_{ij}=0\quad\Leftrightarrow\quad\beta_{ij}=0;\qquad\sum_{j=0}^{i-1}\alpha_{ij}=1;
\end{equation}
for all $i=1,\dots,s$ and $j=0,\dots,s-1$. In the Section \ref{TheDiscreteMaximumPrinciple}, we will investigate the commonly used second and third order TVD-RK methods from Shu \cite{Shu}. The coefficients for these methods are given in the Tables \ref{tab:TVDRK2} and \ref{tab:TVDRK3}.
\begin{center}
\captionof{table}{The coefficients for the second order TVD-RK method from Shu \cite{Shu}.}\label{tab:TVDRK2}
\begin{tabular}{c}
\toprule
TVD-RK2\\ \midrule
\begin{tabular}{c|c|c}
$\gamma_{0}=0$ & $\alpha_{10}=1$ & $\beta_{10}=1$ \\
$\gamma_{1}=1$ & $\alpha_{20}=\frac{1}{2}$, $\alpha_{21}=\frac{1}{2}$ & $\beta_{20}=0$, $\beta_{21}=\frac{1}{2}$
\end{tabular} \\
\bottomrule\par
\end{tabular}
\end{center}

\begin{center}
\normalfont\captionof{table}{The coefficients for the third order TVD-RK method from Shu \cite{Shu}.}\label{tab:TVDRK3}
\begin{tabular}{c}
\toprule
TVD-RK3\\ \midrule
\begin{tabular}{c|c|c}
$\gamma_{0}=0$ & $\alpha_{10}=1$ & $\beta_{10}=1$ \\
$\gamma_{1}=1$ & $\alpha_{20}=\frac{3}{4}$, $\alpha_{21}=\frac{1}{4}$ & $\beta_{20}=0$, $\beta_{21}=\frac{1}{4}$  \\
$\gamma_{2}=\frac{1}{2}$ & $\alpha_{31}=\frac{1}{3}$, $\alpha_{31}=0$, $\alpha_{32}=\frac{2}{3}$  & $\beta_{30}=0$, $\beta_{31}=0$, $\beta_{32}=\frac{2}{3}$
\end{tabular} \\
\bottomrule\par
\end{tabular}
\end{center}

% The discrete geometric conservation law -----------------------------------------------------------------------------------------------
\subsection{The discrete geometric conservation law (D-GCL)}\label{Section:GCL}
For the time discretization of the semi-discrete ALE-DG \textbf{Problem \ref{ALE-DG2Method}}, it is convenient to introduce the notation
\begin{align}\label{FancyRHS}
\begin{split}
\mathcal{G}\left(u_{h}^{*},v^{*},J_{K\left(t\right)},t\right):=&\left(J_{K\left(t\right)}\tilde{\boldsymbol{g}}\left(\boldsymbol{\omega}\left(t\right),u_{h}^{*}\right),\nabla_{\boldsymbol{\xi}}v^{*}\right)_{\refE} \\
&-\left\langle \widehat{g}\left(\boldsymbol{\omega},u_{h}^{*,\text{int}_{\refE}},u_{h}^{*,\text{ext}_{\refE}},J_{K\left(t\right)}\tilde{\n}\left(t\right)\right),v^{*,\text{int}_{\refE}}\right\rangle _{\partial\refE},
\end{split}
\end{align}
where $\tilde{\n}\left(t\right)=A_{K\left(t\right)}^{-T}\n_{\refE}$. At this point, we note that the grid velocity satisfies the property \textit{(ii)} in \Cref{PropertiesGridVelocity}. Therefore, in order to avoid confusion with the time-dependency of the grid velocity, we will highlight the time variable if the grid velocity depends on time, otherwise the time variable will be omitted. For all $v\in\mathcal{V}_{h}(t)$ with $v^{*}=v\circ\chi_{K\left(t\right)}$ the equation \eqref{eqn:DJ-CL2} can be written as
\begin{equation}\label{eqn:DJ-CL2V2}
\left(\frac{d}{dt}\left(J_{K\left(t\right)}u_{h}^{*}\right),v^{*}\right)_{\refE}=\mathcal{G}\left(u_{h}^{*},v^{*},J_{K\left(t\right)},t\right).
\end{equation}
%The fully-discrete ALE-DG needs to satisfy the discrete geometric conservation law (D-GCL).
%\noindent
In accordance with Guillard and Farhat \cite{Farhat2000}, we introduce the following definition.
\begin{definition}\label{Def:D-GCL1}
A fully-discrete moving mesh method for the initial value problem \eqref{eqn:CL} satisfies a D-GCL, if for all $c\in\R$ and $n=0,\dots,\mathcal{N}-1$
\begin{equation*}\label{Def:D-GCL}
u_{h}\left(\mathbf{x},t_{n}\right)=c,\quad\text{for all }\mathbf{x}\in\Omega\qquad\Rightarrow\qquad u_{h}\left(\mathbf{x},t_{n+1}\right)=c,\quad\text{for all }\mathbf{x}\in\Omega.
\end{equation*}
\end{definition}
\begin{lemma}\label{RHS}
Let $c\in\R$ be a constant. Then for all $t\in \left[t_{n},t_{n+1}\right]$, $K\left(t\right)\in\mathcal{T}_{\left(t\right)}$ and $v\in\mathcal{V}_{h}(t)$ with $v^{*}=v\circ\chi_{K\left(t\right)}$, it holds
\begin{equation}\label{RHS1}
\mathcal{G}\left(c,v^{*},J_{K\left(t\right)},t\right)=\left(J_{K\left(t\right)}\nabla_{\boldsymbol{\xi}}\cdot\left[\mathrm{{\bf A}}_{K\left(t\right)}^{-1}\Big(\boldsymbol{\omega}\left(t\right)c\Big)\right],v^{*}\right)_{\refE}.
\end{equation}
\end{lemma}
\begin{proof}
Since $\boldsymbol{f}\left(c\right)$ contains merely constant coefficients and
\begin{equation*}
\tilde{\boldsymbol{g}}\left(\boldsymbol{\omega},c\right)=\mathrm{{\bf A}}_{K\left(t\right)}^{-1}\Big(\boldsymbol{g}\left(\boldsymbol{\omega},c\right)\Big)=\mathrm{{\bf A}}_{K\left(t\right)}^{-1}\Big(\boldsymbol{f}\left(c\right)\Big)-\mathrm{{\bf A}}_{K\left(t\right)}^{-1}\Big(\boldsymbol{\omega}\left(t\right)c\Big),
\end{equation*}
the integration by parts formula provides
\begin{align}\label{RHS2}
\begin{split}
\left(J_{K\left(t\right)}\tilde{\boldsymbol{g}}\left(\boldsymbol{\omega}\left(t\right),c\right),\nabla_{\boldsymbol{\xi}}v^{*}\right)_{\refE}
=& \quad \left(J_{K\left(t\right)}\nabla_{\boldsymbol{\xi}}\cdot\left[\mathrm{{\bf A}}_{K\left(t\right)}^{-1}\big(\boldsymbol{\omega}\left(t\right)c\big)\right],v^{*}\right)_{\refE} \\
& +\left\langle \tilde{\boldsymbol{g}}\left(\boldsymbol{\omega}\left(t\right),c\right)\cdot\left(J_{K\left(t\right)}\n_{\refE}\right),v^{*,\text{int}_{\refE}}\right\rangle _{\partial\refE}.
\end{split}
\end{align}
Furthermore, since $\tilde{\n}\left(t\right)=A_{K\left(t\right)}^{-T}\n_{K_{ref}}$, the property (P1) of the numerical flux provides
\begin{align}\label{RHS3}
\begin{split}
\widehat{g}\left(\boldsymbol{\omega},c,c,J_{K\left(t\right)}\tilde{\n}\left(t\right)\right)=\tilde{\boldsymbol{g}}\left(\boldsymbol{\omega}\left(t\right),c\right)\cdot\left(J_{K\left(t\right)}\n_{\refE}\right).
\end{split}
\end{align}
Thus, we obtain the identity \eqref{RHS1} by \eqref{RHS2} and \eqref{RHS3}.
\end{proof}
\noindent
Next, we assume that $u_{h}^*=c$ solves the semi-discrete ALE-DG method \textbf{Problem \ref{ALE-DG2Method}}. Then, we obtain by \eqref{eqn:DJ-CL2V2} and \eqref{RHS1}
\begin{equation}\label{eqn:ConstatState}
\left(\frac{d}{dt}\left(J_{K\left(t\right)}c\right),v^{*}\right)_{\refE}=\left(J_{K\left(t\right)}\nabla_{\boldsymbol{\xi}}\cdot\left[\mathrm{{\bf A}}_{K\left(t\right)}^{-1}\big(\boldsymbol{\omega}\left(t\right)c\big)\right],v^{*}\right)_{\refE}.
\end{equation}
The equation \eqref{eqn:ConstatState} and the ODE \eqref{TimeDerivativeJacobian} are equivalent, since $c$ is an arbitrary constant, $v\in\mathcal{V}_{h}(t)$ with $v^{*}=v\circ\chi_{K\left(t\right)}$ is an arbitrary test function and the quantities
\begin{equation}
J_{K\left(t\right)}\in P^{d}\left(\left[t_{n},t_{n+1}\right]\right), \qquad \left(\nabla_{\boldsymbol{\xi}}\cdot\left[\mathrm{{\bf A}}_{K\left(t\right)}^{-1}\big(\boldsymbol{\omega}\left(t\right)\big)\right]\right)J_{K\left(t\right)}\in P^{d-1}\left(\left[t_{n},t_{n+1}\right]\right)
\end{equation}
are merely time-dependent. We note that the time evolution of the  metric terms $J_{K\left(t\right)}$ needs to be respected in the time discretization of the semi-discrete ALE-DG method \textbf{Problem \ref{ALE-DG2Method}}. Therefore, we discretize the ODE \eqref{TimeDerivativeJacobian} and \eqref{eqn:DJ-CL2V2} simultaneously by the same TVD-RK method. The stage solutions of the TVD-RK  discretization for \eqref{TimeDerivativeJacobian} will be used to update the metric terms in the TVD-RK discretization for \eqref{eqn:DJ-CL2V2}.
\\
\noindent
{\bf\textit{The fully-discrete ALE-DG method:}}
First, the ODE \eqref{TimeDerivativeJacobian} is discretized by a s-stage TVD-RK method:
\begin{subequations}\label{DiscreteTimeDerivativeJacobian}
\begin{align}
%\begin{split}
 &J_{K^{n,0}}=J_{K^{n}}, \label{DiscreteTimeDerivativeJacobian1} \\
%& \text{for }i=1,\dots,s: \\
& J_{K^{n,i}}=\sum_{j=0}^{i-1}\left(\alpha_{ij}J_{K^{n,j}}+\beta_{ij}\triangle t\left(\nabla_{\boldsymbol{\xi}}\cdot\left[\mathrm{{\bf A}}_{K^{n+\gamma_{j}}}^{-1}\big(\boldsymbol{\omega}^{n+\gamma_{j}}\big)\right]\right)J_{K^{n,j}}\right), \quad \text{for }i=1,\dots,s, \label{DiscreteTimeDerivativeJacobian2} \\
& J_{K^{n+1}}=J_{K^{n,s}}, \label{DiscreteTimeDerivativeJacobian3}
%\end{split}
\end{align}
\end{subequations}
where $K^{n+\gamma_{j}}:=K\left(t_{n}+\gamma_{j}\triangle t\right)$ and $\boldsymbol{\omega}^{n+\gamma_{j}}:=\boldsymbol{\omega}\left(t_{n}+\gamma_{j}\triangle t\right)$.
The stage solutions $\left\{ J_{K^{n,i}}\right\} _{i=0}^{s}$ are used to update the metric terms in the TVD-RK discretization of \eqref{eqn:DJ-CL2V2}. The Runge-Kutta method needs to solve the ODE \eqref{TimeDerivativeJacobian} exact such that
\begin{equation}\label{AdditionalODECondition}
J_{K^{n+1}}=J_{K\left(t_{n+1}\right)}=d!\left|K\left(t_{n+1}\right)\right|,\qquad\forall K\left(t_{n+1}\right)\in\mathcal{T}_{\left(t_{n+1}\right)},
\end{equation}
where $\mathcal{T}_{(t_{n+1})}$ is the regular mesh of simplices which has been used in the Section \ref{Sec:TheALE-DGSetting} to construct the time-dependent cells \eqref{cells}. We note that in a $d$-dimensional space a TVD-RK method with order greater than or equal to $d$ is necessary to compute the metric term $J_{K^{n+1}}$ exactly, since the right hand side of the equation \eqref{TimeDerivativeJacobian} belongs to the space $P^{d-1}\left(\left[t_{n},t_{n+1}\right]\right)$.
\begin{problem}[The fully-discrete ALE-DG method on the reference cell]\label{FullDiscreteALE-DGMethod}
Find a function $u_{h}\in\mathcal{V}_{h}(t)$, such that for all $v\in \mathcal{V}_{h}(t)$ there holds
\begin{subequations}\label{FullDiscreteALE-DGMethod1}
\begin{align}
\left(J_{K^{n,0}}u_{h}^{n,0,*},v^{*}\right)_{\refE}=&\quad \left(J_{K^{n}}u_{h}^{n,*},v^{*}\right)_{\refE}, \\
\left(J_{K^{n,i}}u_{h}^{n,i,*},v^{*}\right)_{\refE}=&\quad
\sum_{j=0}^{i-1}\alpha_{ij}\left(J_{K^{n,j}}u_{h}^{n,j,*},v^{*}\right)_{\refE} \nonumber \\
& +\sum_{j=0}^{i-1}\beta_{ij}\triangle t\mathcal{G}\left(u_{h}^{n,j,*},v^{*},J_{K^{n,j}},t_{n+\gamma_{j}}\right),\quad\text{for }i=1,\dots,s, \label{FullDiscreteALE-DGMethod2} \\
\left(J_{K^{n+1}}u_{h}^{n+1,*},v^{*}\right)_{\refE}=&\quad \left(J_{K^{n,s}}u_{h}^{n,s,*},v^{*}\right)_{\refE},
\end{align}
\end{subequations}
where the metric terms $\left\{ J_{K^{n,i}}\right\} _{i=0}^{s}$ are computed by \eqref{DiscreteTimeDerivativeJacobian}.
\end{problem}
\noindent
We note that on a static mesh the cells $K(t)$ are time-independent. Thus, on a static mesh, the \textbf{Problem \ref{FullDiscreteALE-DGMethod}} corresponds to the TVD-RK DG method developed by Cockburn and Shu in \cite{Cockburn2001}. Next, we prove that the fully-discrete ALE-DG method satisfies the following statement.
\begin{theorem}\label{GCL}
Suppose a s-stage TVD-RK method with order greater than or equal to $d$ is used in \eqref{DiscreteTimeDerivativeJacobian} and \textbf{Problem \ref{FullDiscreteALE-DGMethod}}, and the solutions $\left\{ J_{K^{n,i}}\right\} _{i=0}^{s}$ are used to compute the metric terms in \textbf{Problem \ref{FullDiscreteALE-DGMethod}}. Moreover, the solution at time level $t_{n}$ satisfies $u_{h}^{n,*}=c\in\R$. Then it is $u_{h}^{n,i,*}=c$ for all $i=0,\dots,s$.
\end{theorem}
\begin{proof}
Let $i\in\left\{ 0,\dots,s\right\}$ be an arbitrary fixed index and $v\in \mathcal{V}_{h}(t)$ be an arbitrary test function. We are interested to investigate the $i$-th Runge-Kutta stage in the \textbf{Problem \ref{FullDiscreteALE-DGMethod}}. Hence, we can assume that $u_{h}^{n,j,*}=c$ for all $j=0,\dots,s-1$. Then, the equation \eqref{RHS1} in \Cref{RHS} provides
\begin{align}\label{GCL1}
\begin{split}
\left(J_{K^{n,i}}u_{h}^{n,i,*},v^{*}\right)_{\refE}
=&
\quad \sum_{j=0}^{i-1}\alpha_{ij}\left(J_{K^{n,j}}c,v^{*}\right)_{\refE} \\
&+\sum_{j=0}^{i-1}\beta_{ij}\triangle t\left(J_{K^{n,j}}\nabla_{\boldsymbol{\xi}}\cdot\left[\mathrm{{\bf A}}_{K^{n+\gamma_{j}}}^{-1}\Big(\boldsymbol{\omega}^{n+\gamma_{j}}c\Big)\right],v^{*}\right)_{\refE}.
\end{split}
\end{align}
Next, we multiply the equation \eqref{DiscreteTimeDerivativeJacobian2} by $cv^{*}$ and integrate the result over the reference element $\refE$. This provides the identity
\begin{align}\label{GCL2}
\begin{split}
\left(J_{K^{n,i}}c,v^{*}\right)_{\refE}
=&
\quad \sum_{j=0}^{i-1}\alpha_{ij}\left(J_{K^{n,j}}c,v^{*}\right)_{\refE} \\
&+\sum_{j=0}^{i-1}\beta_{ij}\triangle t\left(J_{K^{n,j}}\nabla_{\boldsymbol{\xi}}\cdot\left[\mathrm{{\bf A}}_{K^{n+\gamma_{j}}}^{-1}\Big(\boldsymbol{\omega}^{n+\gamma_{j}}c\Big)\right],v^{*}\right)_{\refE}.
\end{split}
\end{align}
Since the metric terms $\left\{ J_{K^{n,j}}\right\} _{j=0}^{i-1}$ in \eqref{FullDiscreteALE-DGMethod2} are computed by \eqref{DiscreteTimeDerivativeJacobian}, the equation \eqref{GCL2} can be plugged in the equation \eqref{GCL1} and it follows
\begin{equation*}
\left(J_{K^{n,i}}u_{h}^{n,i,*},v^{*}\right)_{\refE}=\left(J_{K^{n,i}}c,v^{*}\right)_{\refE}.
\end{equation*}
Thus, it follows $u_{h}^{n,i,*}=c$, since the metric term $J_{K^{n,i}}$ is merely time-dependent and the test function $v\in \mathcal{V}_{h}(t)$ was chosen arbitrary.
\end{proof}
\noindent
\Cref{GCL} provides that constant initial data is preserved in each Runge-Kutta stage. In particular, a direct consequence of the \Cref{GCL} is the following result.
\begin{corollary}
Suppose a s-stage TVD-RK method with order greater than or equal to $d$ is used in \eqref{DiscreteTimeDerivativeJacobian} and \textbf{Problem \ref{FullDiscreteALE-DGMethod}}. Furthermore, the solutions $\left\{ J_{K^{n,i}}\right\} _{i=0}^{s}$ given by \eqref{DiscreteTimeDerivativeJacobian} are used to compute the metric terms in \textbf{Problem \ref{FullDiscreteALE-DGMethod}}. Then the fully-discrete ALE-DG method satisfies the D-GCL in the sense of the \Cref{Def:D-GCL1}.
\end{corollary}

% The discrete maximum principle --------------------------------------------------------------------------------------------------
\subsection{The discrete maximum principle}\label{TheDiscreteMaximumPrinciple}
In this section, we investigate the fully-discrete ALE-DG method \textbf{Problem \ref{FullDiscreteALE-DGMethod}} in two dimensions. First, we review the bound preserving limiter which was developed by Zhang, Xia and Shu in \cite{Zhang2012}.
%Then, we discuss the forward Euler method as time integrator and show that this method is not an appropriate choice for the time discretization of the semi-discrete ALE-DG method \textbf{Problem \ref{ALE-DG2Method}}.
%Then, we prove that the ALE-DG method with the bound preserving limiter does not satisfy a discrete maximum principle, when the explicit Euler method is used as time integrator.
Then, we consider the second and third order accurate TVD-RK methods given in the Tables \ref{tab:TVDRK2} and \ref{tab:TVDRK3}. We prove that for these methods the fully-discrete ALE-DG method satisfies the maximum principle when the bound preserving limiter is applied.

% The bound preserving limiter -------------------------------------------------------------------------------------------------
\subsubsection{The bound preserving limiter}\label{pplimiter}
In the following, we briefly review the bound preserving limiter methodology for an arbitrary $s$-stage TVD-RK method \eqref{DiscreteODE}. Let $u_{h}^{n}$ be the DG solution at time level $t=t_{n}$, $K\subseteq\Omega$ an arbitrary cell and $Q_{K}\subseteq K$ be a set of quadrature points. The corresponding quadrature formula needs to be exact for the integration of polynomials with degree $k$ on the cell $K$ and the set $Q_{K}$ needs to contain all quadrature points which are necessary to evaluate the surface integrals along the edges of $K$ in the DG method. A quadrature rule with these properties has been developed in \cite{Zhang2012} and will be presented in the Section \ref{QuadratureRule}. Then the cell average
\begin{equation*}
\overline{u}_{K}^{n}:=\frac{1}{\left|K\right|}\left(u_{h},1\right)_{K}
\end{equation*}
of the numerical solution $u_{h}$ can be computed exact by the quadrature formula. The limiter methodology is applied in two steps:
\begin{enumerate}
\item[(L1)] It needs to be ensured that the $s$-stage TVD-RK method satisfies
\begin{equation}\label{almostMP}
\begin{cases}
 & u_{h}^{n,0}|_{Q_{K}}=u_{h}^{n}|_{Q_{K}},\\
 & \text{for }\ell=1,\dots,s: \\
  & u_{h}^{n,j}|_{Q_{K}}\in\left[m,M\right], \quad\forall j=0,...,\ell-1,
 \quad\Rightarrow\quad\overline{u}_{K}^{n,\ell}\in\left[m,M\right],\\
 & \overline{u}_{K}^{n+1}=\overline{u}_{K}^{n,s},
\end{cases}
\end{equation}
where $m$ as well as $M$ are given by \eqref{bounds}, $u_{h}^{n,\ell}$ is the solution given by the $\ell$-th stage of the RK-DG method and $\overline{u}_{K}^{n,\ell}$ denotes the corresponding cell average with respect to the cell $K$.
\item[(L2)] If (L1) is satisfied, the solution $u_{h}$ will be revised for all $\ell=1,...,s$ by
\begin{equation}\label{Limiter}
\widetilde{u}_{h}^{n,\ell}|_{K}:=\Theta\left(u_{h}^{n,\ell}|_{K}-\overline{u}_{K}^{n,\ell}\right)+\overline{u}_{K}^{n,\ell},
\end{equation}
\begin{equation*}
\Theta:=\min\left\{ \left|\frac{M-\overline{u}_{K}^{n,\ell}}{M_{K}-\overline{u}_{K}^{n,\ell}}\right|,\left|\frac{m-\overline{u}_{K}^{n,\ell}}{m_{K}-\overline{u}_{K}^{n,\ell}}\right|,1\right\},
\end{equation*}
where $M_{K}:=\underset{{\bf x}\in Q_{K}}{\max}\,u_{h}^{n,\ell}\left({\bf x}\right)$, $m_{K}:=\underset{{\bf x}\in Q_{K}}{\min}\,u_{h}^{n,\ell}\left({\bf x}\right)$. Then it is ensured that $\widetilde{u}_{h}^{n,\ell}|_{K}\in\left[m,M\right]$, for all $\ell=1,...,s$.
\end{enumerate}
The stability property (L1) is the maximum principle property for the cell-averages of the numerical solution $u_{h}$. We note that by an adjustment of the CFL constraint the stability property (L1) is satisfied for any high order TVD-RK method, if the forward Euler step satisfies (L1), since the TVD-RK methods are convex combinations of the forward Euler step (cf. Gottlieb and Shu \cite{Gottlieb1998}). However, the forward Euler method is merely first order accurate and thus this method is not an appropriate choice for the discretization of the ODE \eqref{TimeDerivativeJacobian} and the semi-discrete ALE-DG method \textbf{Problem \ref{ALE-DG2Method}} in two dimensions, since it does not provide that the equation \eqref{AdditionalODECondition} holds. Therefore, the convexity argument in Gottlieb and Shu \cite{Gottlieb1998} cannot be used to analyze a fully-discrete ALE-DG method which has a high order TVD-RK method as time integrator and uses the bound preserving limiter. For this reason, each high order TVD-RK method needs to be investigated separately.

% High order time discretization methods ---------------------------------------------------------------------------------------
\subsubsection{High order time discretization methods}\label{HighOrderMethods}
%Henceforth, we consider the \textbf{Problem \ref{ALE-DG2Method}} in two space dimensions.
The step (L1) in the bound preserving limiter methodology is related to the  cell averages of the numerical solution. Thus, a scheme satisfied by the cell averages of the ALE-DG solution $u_{h}$ needs to be investigated. In this section, we will present these schemes for the TVD-RK2 and the TVD-RK3 ALE-DG methods. Later, in the Section \ref{MPsecondorder} and \ref{MPthirdorder}, these schemes will be used to prove that the TVD-RK2 and the TVD-RK3 ALE-DG methods satisfy indeed the maximum principle, when the bound preserving limiter \eqref{Limiter} is applied.  \\
% The second order method ------------------------------------------------------------------------------------------------------
\noindent
{\bf\textit{The TVD-RK2 ALE-DG method:}}
We note that by the identity \eqref{Jacobian} the metric term $J_{K(t)}$ can be also written as $2\left|K\left(t\right)\right|$. Thus, for the TVD-RK2 method in Table \ref{tab:TVDRK2} the scheme \eqref{DiscreteTimeDerivativeJacobian} becomes
\begin{subequations}\label{2O:discreteODE1}
\begin{align}
\left|K^{n,1}\right|=&\left|K^{n}\right|+\triangle t\left(\nabla_{\boldsymbol{\xi}}\cdot\left[\mathrm{{\bf A}}_{K^{n}}^{-1}\big(\boldsymbol{\omega}^{n}\big)\right]\right)\left|K^{n}\right|, \label{2O:discreteODE2} \\
\left|K^{n+1}\right|=&\frac{1}{2}\left|K^{n}\right|+\frac{1}{2}\left|K^{n,1}\right|+\frac{\triangle t}{2}\left(\nabla_{\boldsymbol{\xi}}\cdot\left[\mathrm{{\bf A}}_{K^{n+1}}^{-1}\big(\boldsymbol{\omega}^{n+1}\big)\right]\right)\left|K^{n,1}\right|. \label{2O:discreteODE3}
\end{align}
\end{subequations}
Next, we apply the TVD-RK2 method and the test function $v^*=1$ in \eqref{FullDiscreteALE-DGMethod1} and obtain
\begin{subequations}\label{TVDRK2}
\begin{align}
\left|K^{n,1}\right|\overline{u}_{K^{n,1}}^{n,1}
=&\left|K^{n}\right|\overline{u}_{K^{n}}^{n} +\triangle t\mathcal{G}\left(u_{h}^{n,*},1,J_{K^{n}},t_{n}\right), \label{TVDRK2A} \\
\left|K^{n+1}\right|\overline{u}_{K^{n+1}}^{n+1}
=&\frac{1}{2}\left|K^{n}\right|\overline{u}_{K^{n}}^{n}+\frac{1}{2}\left|K^{n,1}\right|\overline{u}_{K^{n,1}}^{n,1}
+\frac{\triangle t}{2}\mathcal{G}\left(u_{h}^{n,1,*},1,J_{K^{n,1}},t_{n}\right),  \label{TVDRK2B}
\end{align}
\end{subequations}
where $\overline{u}_{K^{n}}^{n}$, $\overline{u}_{K^{n,1}}^{n,1}$ and $\overline{u}_{K^{n+1}}^{n+1}$ are the cell average values of the ALE-DG solution $u_{h}$ and the quantities $\left|K^{n}\right|$, $\left|K^{n,1}\right|$, $\left|K^{n+1}\right|$ are computed by \eqref{2O:discreteODE1}. We note that $\left|K\left(t_{n+1}\right)\right|=\left|K^{n+1}\right|$, since in two space dimensions, the TVD-RK2 method in Table \ref{tab:TVDRK2} solves the ODE \eqref{TimeDerivativeJacobian} exactly. \\
% The third order method -------------------------------------------------------------------------------------------------------
\noindent
{\bf\textit{The TVD-RK3 ALE-DG method:}}
For the TVD-RK3 method in Table \ref{tab:TVDRK3} the method \eqref{DiscreteTimeDerivativeJacobian} becomes
\begin{subequations}\label{3O:discreteODE1}
\begin{align}
\left|K^{n,1}\right|=&\left|K^{n}\right|+\triangle t\left(\nabla_{\boldsymbol{\xi}}\cdot\left[\mathrm{{\bf A}}_{K^{n}}^{-1}\left(\boldsymbol{\omega}^{n}\right)\right]\right)\left|K^{n}\right|, \label{3O:discreteODE2} \\
\left|K^{n,2}\right|=&\frac{3}{4}\left|K^{n}\right|+\frac{1}{4}\left|K^{n,1}\right|+\frac{\triangle t}{4}\left(\nabla_{\boldsymbol{\xi}}\cdot\left[\mathrm{{\bf A}}_{K^{n+1}}^{-1}\left(\boldsymbol{\omega}^{n+1}\right)\right]\right)\left|K^{n,1}\right|, \label{3O:discreteODE3} \\
\left|K^{n+1}\right|=&\frac{1}{3}\left|K^{n}\right|+\frac{2}{3}\left|K^{n,2}\right|+\frac{2\triangle t}{3}\left(\nabla_{\boldsymbol{\xi}}\cdot\left[\mathrm{{\bf A}}_{K^{n+\frac{1}{2}}}^{-1}\left(\boldsymbol{\omega}^{n+\frac{1}{2}}\right)\right]\right)\left|K^{n,2}\right|,   \label{3O:discreteODE4}
\end{align}
\end{subequations}
where $t_{n+\frac{1}{2}}:=\frac{1}{2}\left(t_{n}+t_{n+1}\right)$ and $K^{n+\frac{1}{2}}=K\left(t_{n+\frac{1}{2}}\right)$. In two and three space dimensions, the third order TVD-RK3 method in Table \ref{tab:TVDRK3} solves the ODE \eqref{TimeDerivativeJacobian} exactly and thus $\left|K\left(t_{n+1}\right)\right|=\left|K^{n+1}\right|$.
Next, we apply the test function $v^*=1$ in \eqref{FullDiscreteALE-DGMethod1} and obtain the scheme
\begin{subequations}\label{TVDRK3}
\begin{align}
\left|K^{n,1}\right|\overline{u}_{K^{n,1}}^{n,1}=&\left|K^{n}\right|\overline{u}_{K^{n}}^{n} +\triangle t\mathcal{G}\left(u_{h}^{n,*},1,J_{K^{n}},t_{n}\right), \label{TVDRK3A} \\
\left|K^{n,2}\right|\overline{u}_{K^{n,2}}^{n,2}=& \frac{3}{4}\left|K^{n}\right|\overline{u}_{K^{n}}^{n}+\frac{1}{4}\left|K^{n,1}\right|\overline{u}_{K^{n,1}}^{n,1}+\frac{\triangle t}{4}\mathcal{G}\left(u_{h}^{n,1,*},1,J_{K^{n,1}},t_{n+1}\right), \label{TVDRK3B} \\
\left|K^{n+1}\right|\overline{u}_{K^{n+1}}^{n+1}=&\frac{1}{3}\left|K^{n}\right|\overline{u}_{K^{n}}^{n}+\frac{2}{3}\left|K^{n,2}\right|\overline{u}_{K^{n,2}}^{n,2}+\frac{2\triangle t}{3}\mathcal{G}\left(u_{h}^{n,2,*},1,J_{K^{n,2}},t_{n+\frac{1}{2}}\right), \label{TVDRK3C}
\end{align}
\end{subequations}
where the values $\overline{u}_{K^{n}}^{n}$, $\overline{u}_{K^{n,1}}^{n,1}$, $\overline{u}_{K^{n,2}}^{n,2}$ and $\overline{u}_{K^{n+1}}^{n+1}$ are the cell average values of the ALE-DG solution $u_{h}$ and the quantities $\left|K^{n}\right|$, $\left|K^{n,1}\right|$, $\left|K^{n,2}\right|$, $\left|K^{n+1}\right|$ are computed by \eqref{3O:discreteODE1}.

% A quadrature rule to decompose the cell average value ------------------------------------------------------------------------
\subsubsection{A quadrature rule to decompose the cell average value}\label{QuadratureRule}
In order to prove that the second and the third order fully-discrete ALE-DG methods satisfy the maximum principle, when the ALE-DG solution is revised by the limiter \eqref{Limiter}, we proceed similar to the derivation by Zhang, Xia and Shu in \cite{Zhang2012}. For this task, we need to apply a special quadrature formula developed by Zhang, Xia and Shu to decompose the cell average values of the ALE-DG solution. In the following this quadrature formula is briefly reviewed.

First of all, it should be noted that in the implementation of the ALE-DG method the edge integrals are approximated  by a $k+1$-point $2k+1$ accurate Gauss quadrature formula. Hence, we obtain
\begin{align}\label{Aprox:EdgeIntegrals}
\begin{split}
& \left\langle \widehat{g}\left(\boldsymbol{\omega},u_{h}^{*,\text{int}_{\refE}},u_{h}^{*,\text{ext}_{\refE}},J_{K\left(t\right)}\tilde{\n}\left(t\right)\right),1\right\rangle _{\partial\refE} \\
\approx &
\sum_{\beta=1}^{k+1}\sum_{\nu=1}^{3}\sigma_{\beta}\widehat{g}\left(\boldsymbol{\omega}_{\nu,\beta},u_{\nu,\beta}^{*,\text{int}_{\refE}},u_{\nu,\beta}^{*,\text{ext}_{\refE}},J_{K\left(t\right)}\tilde{\n}_{F_{\refE}^{\nu}}\left(t\right)\right)\ell_{F_{\refE}^{\nu}},
\end{split}
\end{align}
where $\tilde{\n}\left(t\right)={\bf A}_{K\left(t\right)}^{-T}\n_{\refE}$, $\tilde{\n}_{F_{\refE}^{\nu}}\left(t\right)={\bf A}_{K\left(t\right)}^{-T}\n_{F_{\refE}^{\nu}}$, $F_{\refE}^{\nu}$, $\nu =1,2,3$, are the edges of the reference cell $\refE$, the corresponding normals and lengths of the edges are $\n_{F_{\refE}^{\nu}}$  as well as  $\ell_{F_{\refE}^{\nu}}$. Moreover, we denote by $u_{\nu,\beta}^{*,\text{int}_{\refE}}$ as well as $u_{\nu,\beta}^{*,\text{ext}_{\refE}}$ the values of the mapped ALE-DG solution $u_{h}^{*}:=u_{h}\circ\boldsymbol{\chi}_{K\left(t\right)}$ evaluated in the $\beta$-th Gauss quadrature point on the edge $F_{\refE}^{\nu}$. The corresponding quadrature weights for the interval $\left[-\frac{1}{2},\frac{1}{2}\right]$ are $\sigma_{\beta}$. Likewise, $\boldsymbol{\omega}_{\nu,\beta}$ are the values of the grid velocity $\boldsymbol{\omega}$ evaluated in the $\beta$-th Gauss quadrature point on the edge $F_{\refE}^{\nu}$. We obtain for a constant $c\in\R$
\begin{equation}\label{Aprox:EdgeIntegralsGridvelocity}
\begin{split}
\left\langle \left(\boldsymbol{\omega}c\right)\cdot J_{K\left(t\right)}\tilde{\n}\left(t\right),1\right\rangle _{\partial\refE}=\sum_{\beta=1}^{k+1}\sum_{\nu=1}^{3}\sigma_{\beta}\left(\boldsymbol{\omega}_{\nu,\beta}c\right)\cdot\Big(J_{K(t)}\tilde{\n}_{F_{\refE}^{\nu}}\left(t\right)\Big)\ell_{F_{\refE}^{\nu}},
\end{split}
\end{equation}
since the edge integrals are approximated by a $2k+1$ accurate Gauss quadrature formula and the grid velocity belongs to the space  $P^{1}\left(\refE,\R^{d}\right)$ for all $t\in\left[t_{n},t_{n+1}\right]$ according to \Cref{PropertiesGridVelocity}. Thus, it follows for a constant $c\in\R$
\begin{align}\label{Aprox:EdgeIntegralsConstantState}
\begin{split}
& \sum_{\beta=1}^{k+1}\sum_{\nu=1}^{3}\sigma_{\beta}\widehat{g}\left(\boldsymbol{\omega}_{\nu,\beta},c,c,J_{K\left(t\right)}\tilde{\n}_{F_{\refE}^{\nu}}\left(t\right)\right)\ell_{F_{\refE}^{\nu}}\\
= &
\left\langle \left(\boldsymbol{f}\left(c\right)-\boldsymbol{\omega}c\right)\cdot(J_{K(t)}\tilde{\n}\left(t\right)),1\right\rangle _{\partial\refE}=\left\langle \boldsymbol{g}\left(\boldsymbol{\omega},c\right)\cdot(J_{K(t)}\tilde{\n}\left(t\right)),1\right\rangle _{\partial\refE},
\end{split}
\end{align}
since the numerical flux has the property (P1). Therefore, the statement of \Cref{GCL} stays true when we approximate the edge integrals in \eqref{FancyRHS} by a $2k+1$ accurate Gauss quadrature formula.

The cell average values need to be decomposed by a quadrature formula, which includes the Gauss quadrature points for the edges $F_{\refE}^{\nu}$, $\nu =1,2,3$. A quadrature formula with this property has been developed by Zhang et al. in \cite{Zhang2012}. Let us assume that $N$ is the smallest integer with $2N-3\geq k$. The $3(N-1)(k+1)$-point quadrature formula for triangular elements in \cite{Zhang2012}, has the properties:
\begin{itemize}
\item The quadrature formula is exact for the integration of polynomials with degree $k$ on a triangular element.
\item The quadrature points include the Gauss quadrature points for the edges $F_{\refE}^{\nu}$, $\nu =1,2,3$.
\item All the quadrature weights are positive and the weights for the Gauss quadrature points are given by
\begin{equation}\label{qua:weights}
\sigma_{\beta}\hat{\sigma}:=\frac{2}{3}\sigma_{\beta}\tilde{\sigma},\quad\tilde{\sigma}=\frac{1}{N(N-1)},
\quad\beta=1,...,k+1,
\end{equation}
where $\tilde{\sigma}$ corresponds to the 1-st and N-th Gauss-Lobatto quadrature weights for the interval $\left[-\frac{1}{2},\frac{1}{2}\right]$.
\item The quadrature formula has $L:=3(N-2)(k+1)$ points in the interior of a triangular element.
\end{itemize}
This quadrature formula ensures that the cell average values can be written as follows
\begin{equation}\label{cellaverage}
\overline{u}_{K\left(t\right)}=\sum_{\nu=1}^{3}\sum_{\beta=1}^{k+1}\sigma_{\beta}\hat{\sigma}u_{\nu,\beta}^{*,\text{int}_{\refE}}+\sum_{\gamma=1}^{L}\widetilde{\sigma}_{\gamma}u_{\gamma}^{*,\text{int}_{\refE}}.
\end{equation}
We denote by $u_{\gamma}^{*,\text{int}_{\refE}}$, $\gamma=1,...,L$, the values of the mapped ALE-DG solution $u_{h}^{*}:=u_{h}\circ\boldsymbol{\chi}_{K\left(t\right)}$ evaluated in the quadrature points which are lying in the interior of the reference cell $\refE$. The corresponding quadrature weights are denoted by $\widetilde{\sigma}_{\gamma}$.

% The maximum principle for the second order RK ALE-DG method ------------------------------------------------------------------
\subsubsection{The maximum principle for the TVD-RK2 ALE-DG method}\label{MPsecondorder}
In this section, we prove that the TVD-RK2 ALE-DG method satisfies the maximum principle, when the ALE-DG solution is revised by the bound preserving limiter \eqref{Limiter}.

First of all, we show the property (L1). Therefore, similar to the process in \cite{Zhang2012}, we decompose the stages of the scheme \eqref{TVDRK2} in a sum of monotone increasing functions, which preserve constant states.
Therefore, it is convenient to use vector notations. In particular, we define the set
\[
\mathbb{M}:=\left\{ {\bf v}=\left({\bf v}_{1},{\bf v}_{2},{\bf v}_{3}\right):\ {\bf v}_{\nu}=\left(v_{\nu,1},...,v_{\nu,k+1}\right)\in\left[m,M\right]^{k+1},\ \forall\nu=1,2,3\right\}
\]
with $m$ and $M$ given by \eqref{bounds} and apply for any cell $K\left(t\right)\in\mathcal{T}_{\left(t\right)}$ and all $\nu=1,2,3$ the vector notations
\begin{subequations}\label{vectornotation}
\begin{equation}
{\bf u}_{\nu}^{*,\text{int}_{\refE}}:=\left(u_{\nu,1}^{*,\text{int}_{\refE}},\dots u_{\nu,k+1}^{*,\text{int}_{\refE}}\right), \quad
{\bf u}_{\nu}^{*,\text{ext}_{\refE}}:=\left(u_{\nu,1}^{*,\text{ext}_{\refE}},\dots u_{\nu,k+1}^{*,\text{ext}_{\refE}}\right), \label{vectornotation1}
\end{equation}
\begin{equation}
{\bf \widetilde{u}}^{*,\text{int}_{\refE}}:=\left(\widetilde{u}_{1}^{*,\text{int}_{\refE}},\dots,\widetilde{u}_{L}^{*,\text{int}_{\refE}}\right), \label{vectornotation2}
\end{equation}
\begin{equation}
{\bf u}^{*,\text{int}_{\refE}}:=\left({\bf u}_{1}^{*,\text{int}_{\refE}},{\bf u}_{2}^{*,\text{int}_{\refE}},{\bf u}_{3}^{*,\text{int}_{\refE}}\right), \quad
{\bf u}^{*,\text{ext}_{\refE}}:=\left({\bf u}_{1}^{*,\text{ext}_{\refE}},{\bf u}_{2}^{*,\text{ext}_{\refE}},{\bf u}_{3}^{*,\text{ext}_{\refE}}\right). \label{vectornotation3}
\end{equation}
\end{subequations}
Then, by applying the decomposition \eqref{cellaverage} of the cell average values, the scheme  \eqref{TVDRK2} can be written as
\begin{subequations}\label{E2:TVDRK2}
\begin{align}
\overline{u}_{K^{n,1}}^{n,1}=&\mathcal{L}\left({\bf \widetilde{u}}^{n,*,\text{int}_{\refE}},{\bf u}^{n,*,\text{int}_{\refE}},{\bf u}^{n,*,\text{ext}_{\refE}},\left|K^{n,1}\right|,\left|K^{n}\right|,t_{n}\right), \label{E2A:TVDRK2} \\
\overline{u}_{K^{n+1}}^{n+1}=&\frac{1}{2}\mathcal{H}\left({\bf \widetilde{u}}^{n,*,\text{int}_{\refE}},{\bf u}^{n,*,\text{int}_{\refE}},\left|K^{n+1}\right|,\left|K^{n}\right|\right) \nonumber  \\
&+\frac{1}{2}\mathcal{L}\left({\bf \widetilde{u}}^{n,1,*,\text{int}_{\refE}},{\bf u}^{n,1,*,\text{int}_{\refE}},{\bf u}^{n,1,*,\text{ext}_{\refE}},\left|K^{n+1}\right|,\left|K^{n,1}\right|,t_{n+1}\right), \label{E2B:TVDRK2}
\end{align}
\end{subequations}
where for all ${\bf a}\in\left[m,M\right]^{L}$ and ${\bf b},{\bf c}\in\mathbb{M}$
\begin{subequations}\label{VectorTVDRK2}
\begin{align}
\mathcal{H}\left({\bf a},{\bf b},\left|K_{1}\right|,\left|K_{2}\right|\right)
:=&\sum_{\gamma=1}^{L}\widetilde{\sigma}_{\gamma}\left(1-\frac{1}{\left|K_{1}\right|}\left(\left|K_{1}\right|-\left|K_{2}\right|\right)\right)a_{\gamma} \nonumber  \\
&+\sum_{\beta=1}^{k+1}\sum_{\nu=1}^{3}\sigma_{\beta}\hat{\sigma}\left(1-\frac{1}{\left|K_{1}\right|}\left(\left|K_{1}\right|-\left|K_{2}\right|\right)\right)b_{\nu,\beta}, \label{HVectorTVDRK2} \\
& \nonumber \\
\mathcal{L}\left({\bf a},{\bf b},{\bf c},\left|K_{1}\right|,\left|K_{2}\right|,t\right)
:=& \sum_{\gamma=1}^{L}\widetilde{\sigma}_{\gamma}\left(1-\frac{1}{\left|K_{1}\right|}\left(\left|K_{1}\right|-\left|K_{2}\right|\right)\right)a_{\gamma} \nonumber  \\
& +\sum_{\beta=1}^{k+1}\sigma_{\beta}\hat{\sigma}H_{1}\left(b_{1,\beta},b_{2,\beta},c_{1,\beta},\left|K_{1}\right|,\left|K_{2}\right|,t\right)  \nonumber  \\
&+\sum_{\beta=1}^{k+1}\sigma_{\beta}\hat{\sigma}H_{2}\left(b_{1,\beta},b_{2,\beta},b_{3,\beta},c_{2,\beta},\left|K_{1}\right|,\left|K_{2}\right|,t\right) \nonumber  \\
& +\sum_{\beta=1}^{k+1}\sigma_{\beta}\hat{\sigma}H_{3}\left(b_{2,\beta},b_{3,\beta},c_{3,\beta},\left|K_{1}\right|,\left|K_{2}\right|,t\right) \label{LVectorTVDRK2}
\end{align}
with
\begin{align}
H_{1}\left(b_{1,\beta},b_{2,\beta},c_{1,\beta},\left|K_{1}\right|,\left|K_{2}\right|,t\right)
=& \left(1-\frac{1}{\left|K_{1}\right|}\left(\left|K_{1}\right|-\left|K_{2}\right|\right)\right)b_{1,\beta} \nonumber  \\
& -\frac{\triangle t}{\hat{\sigma}\left|K_{1}\right|}\widehat{g}\left(\boldsymbol{\omega}_{1,\beta},b_{1,\beta},c_{1,\beta},J_{K_{2}}{\bf A}_{K\left(t\right)}^{-T}\n_{F_{\refE}^{1}}\right)\ell_{F_{\refE}^{1}}
\nonumber  \\
& +\frac{\triangle t}{\hat{\sigma}\left|K_{1}\right|}\widehat{g}\left(\boldsymbol{\omega}_{1,\beta},b_{2,\beta},b_{1,\beta},J_{K_{2}}{\bf A}_{K\left(t\right)}^{-T}\n_{F_{\refE}^{1}}\right)\ell_{F_{\refE}^{1}}, \label{H1VectorTVDRK2} \\
& \nonumber \\
H_{2}\left(b_{1,\beta},b_{2,\beta},b_{3,\beta},c_{2,\beta},\left|K_{1}\right|,\left|K_{2}\right|,t\right)
=& \left(1-\frac{1}{\left|K_{1}\right|}\left(\left|K_{1}\right|-\left|K_{2}\right|\right)\right)b_{2,\beta} \nonumber  \\
& -\frac{\triangle t}{\hat{\sigma}\left|K_{1}\right|}\widehat{g}\left(\boldsymbol{\omega}_{1,\beta},b_{2,\beta},b_{1,\beta},J_{K_{2}}{\bf A}_{K\left(t\right)}^{-T}\n_{F_{\refE}^{1}}\right)\ell_{F_{\refE}^{1}}\nonumber  \\
& -\frac{\triangle t}{\hat{\sigma}\left|K_{1}\right|}\widehat{g}\left(\boldsymbol{\omega}_{2,\beta},b_{2,\beta},c_{2,\beta},J_{K_{2}}{\bf A}_{K\left(t\right)}^{-T}\n_{F_{\refE}^{2}}\right)\ell_{F_{\refE}^{2}}\nonumber  \\
&-\frac{\triangle t}{\hat{\sigma}\left|K_{1}\right|}\widehat{g}\left(\boldsymbol{\omega}_{3,\beta},b_{2,\beta},b_{3,\beta},J_{K_{2}}{\bf A}_{K\left(t\right)}^{-T}\n_{F_{\refE}^{3}}\right)\ell_{F_{\refE}^{3}}, \label{H2VectorTVDRK2} \\
& \nonumber \\
H_{3}\left(b_{2,\beta},b_{3,\beta},c_{3,\beta},\left|K_{1}\right|,\left|K_{2}\right|,t\right)
=& \left(1-\frac{1}{\left|K_{1}\right|}\left(\left|K_{1}\right|-\left|K_{2}\right|\right)\right)b_{3,\beta} \nonumber  \\
&-\frac{\triangle t}{\hat{\sigma}\left|K_{1}\right|}\widehat{g}\left(\boldsymbol{\omega}_{3,\beta},b_{3,\beta},c_{3,\beta},J_{K_{2}}{\bf A}_{K\left(t\right)}^{-T}\n_{F_{\refE}^{3}}\right)\ell_{F_{\refE}^{3}}\nonumber  \\
&+\frac{\triangle t}{\hat{\sigma}\left|K_{1}\right|}\widehat{g}\left(\boldsymbol{\omega}_{3,\beta},b_{2,\beta},b_{3,\beta},J_{K_{2}}{\bf A}_{K\left(t\right)}^{-T}\n_{F_{\refE}^{3}}\right)\ell_{F_{\refE}^{3}}. \label{H3VectorTVDRK2}
\end{align}
\end{subequations}

In the Section \ref{QuadratureRule}, we mentioned that the result in \Cref{GCL} also holds when the edge integrals in \eqref{FancyRHS} are approximated by a $2k+1$ accurate Gauss quadrature formula. Therefore for a vector $\left({\bf c},{\bf c},{\bf c}\right)\in\left[m,M\right]^{L}\cup\mathbb{M}$ with vector components given by the constant $c\in \left[m,M\right]$, we obtain by the \Cref{GCL}
\begin{equation}\label{LVectorConstant}
\mathcal{L}\left({\bf c},{\bf c},{\bf c},\left|K^{n,1}\right|,\left|K^{n}\right|,t_{n}\right)=c,
\end{equation}
\begin{equation}\label{HandLVectorConstant}
\frac{1}{2}\mathcal{H}\left({\bf c},{\bf c},\left|K^{n+1}\right|,\left|K^{n}\right|\right)+\frac{1}{2}\mathcal{L}\left({\bf c},{\bf c},{\bf c},\left|K^{n+1}\right|,\left|K^{n,1}\right|,t_{n+1}\right)=c.
\end{equation}
Henceforth, for the sake of simplicity, we will use the Lax-Friedrichs flux \eqref{Lax-Friedrichs} for the analysis. However, it should be noted that the techniques which are presented in this section can be also applied to any other monotone flux. The use of the Lax-Friedrichs flux ensures to prove the following lemmas.

\begin{lemma}\label{MonotoneTVDRK2b}
Let $\left|K^{n}\right|$, $\left|K^{n,1}\right|$, $\left|K^{n+1}\right|$ be given by \eqref{2O:discreteODE1} and ${\bf a}\in\left[m,M\right]^{L}$, ${\bf b},{\bf c}\in\mathbb{M}$. Then, under the CFL constraint
\begin{equation}\label{CFL1}
\max\left\{ \hat{\sigma}\left|\left(\nabla_{\boldsymbol{\xi}}\cdot\left[\mathrm{{\bf A}}_{K(t)}^{-1}\big(\boldsymbol{\omega}\left(t\right)\big)\right]\right)\right|\left|K\left(t\right)\right|+\sum_{\nu=1}^{3}\lambda^{n}\ell_{F_{\refE}^{\nu}}:\ t\in\left[t_{n},t_{n+1}\right]\right\} \frac{\triangle t}{\left|K^{n,1}\right|}\leq\hat{\sigma}
\end{equation}
with $\lambda^{n}$ given by \eqref{lambda}, it holds
\begin{equation}\label{LVectorBounds}
m\leq \mathcal{L}\left({\bf a},{\bf b},{\bf c},\left|K^{n,1}\right|,\left|K^{n}\right|,t_{n}\right)\leq M.
\end{equation}
\end{lemma}

\begin{proof}
First of all, we apply the Lax-Friedrichs flux \eqref{Lax-Friedrichs} and rewrite the functions \eqref{H1VectorTVDRK2}, \eqref{H2VectorTVDRK2} and \eqref{H3VectorTVDRK2} as follows
\begin{align*}
& H_{1}\left(b_{1,\beta},b_{2,\beta},c_{1,\beta},\left|K_{1}\right|,\left|K_{2}\right|,t\right) \nonumber \\
=& \left(1-\frac{1}{\left|K_{1}\right|}\left(\left|K_{1}\right|-\left|K_{2}\right|\right)-\frac{\triangle t}{\hat{\sigma}\left|K_{1}\right|}\lambda^{n}\ell_{F_{\refE}^{1}}\right)b_{1,\beta} \nonumber \\
&+\frac{\triangle t}{\hat{\sigma}\left|K_{1}\right|}\left(\widehat{g}_{+}\left(\boldsymbol{\omega}_{1,\beta},b_{2,\beta},J_{K_{2}}\tilde{\n}_{F_{\refE}^{1}}\left(t\right)\right)+\widehat{g}_{-}\left(\boldsymbol{\omega}_{1,\beta},c_{1,\beta},J_{K_{2}}\tilde{\n}_{F_{\refE}^{1}}\left(t\right)\right)\right)\ell_{F_{\refE}^{1}}, \nonumber \\
& \nonumber \\
& H_{2}\left(b_{1,\beta},b_{2,\beta},b_{3,\beta},c_{2,\beta},\left|K_{1}\right|,\left|K_{2}\right|,t\right)\nonumber \\
=&\left(1-\frac{1}{\left|K_{1}\right|}\left(\left|K_{1}\right|-\left|K_{2}\right|\right)\right)b_{2,\beta}-\frac{\triangle t}{\hat{\sigma}\left|K_{1}\right|}\sum_{\nu=1}^{3}\widehat{g}_{+}\left(\boldsymbol{\omega}_{\nu,\beta},b_{2,\beta},J_{K_{2}}\tilde{\n}_{F_{\refE}^{\nu}}\left(t\right)\right)\ell_{F_{\refE}^{\nu}} \nonumber \\
&+\frac{\triangle t}{\hat{\sigma}\left|K_{1}\right|}\widehat{g}_{-}\left(\boldsymbol{\omega}_{1,\beta},b_{1,\beta},J_{K_{2}}\tilde{\n}_{F_{\refE}^{1}}\left(t\right)\right)\ell_{F_{\refE}^{1}}\nonumber %\\&
+\frac{\triangle t}{\hat{\sigma}\left|K_{1}\right|}\widehat{g}_{-}\left(\boldsymbol{\omega}_{2,\beta},c_{2,\beta},J_{K_{2}}\tilde{\n}_{F_{\refE}^{2}}\left(t\right)\right)\ell_{F_{\refE}^{2}} \nonumber \\
&+\frac{\triangle t}{\hat{\sigma}\left|K_{1}\right|}\widehat{g}_{-}\left(\boldsymbol{\omega}_{3,\beta},b_{3,\beta},J_{K_{2}}\tilde{\n}_{F_{\refE}^{3}}\left(t\right)\right)\ell_{F_{\refE}^{3}}, \nonumber \\
& \nonumber \\
& H_{3}\left(b_{2,\beta},b_{3,\beta},c_{3,\beta},\left|K_{1}\right|,\left|K_{2}\right|,t\right) \nonumber \\
=& \left(1-\frac{1}{\left|K_{1}\right|}\left(\left|K_{1}\right|-\left|K_{2}\right|\right)-\frac{\triangle t}{\hat{\sigma}\left|K_{1}\right|}\lambda^{n}\ell_{F_{\refE}^{3}}\right)b_{3,\beta} \nonumber \\
&+\frac{\triangle t}{\hat{\sigma}\left|K_{1}\right|}\left(\widehat{g}_{+}\left(\boldsymbol{\omega}_{3,\beta},b_{2,\beta},J_{K_{2}}\tilde{\n}_{F_{\refE}^{3}}\left(t\right)\right)+\widehat{g}_{-}\left(\boldsymbol{\omega}_{3,\beta},c_{3,\beta},J_{K_{2}}\tilde{\n}_{F_{\refE}^{3}}\left(t\right)\right)\right)\ell_{F_{\refE}^{3}},
\end{align*}
where $\tilde{\n}_{F_{\refE}^{\nu}}\left(t\right):={\bf A}_{K\left(t\right)}^{-T}\n_{F_{\refE}^{\nu}}$, $\nu=1,2,3$.

Next, we observe that for all $t\in\left[0,T\right]$, $\nu=1,2,3$ and $\beta=1,...,k+1$
\begin{equation*}
\left|\partial_{b_{2,\beta}}\widehat{g}_{+}\left(\boldsymbol{\omega}_{\nu,\beta},b_{2,\beta},J_{K\left(t\right)}\tilde{\n}_{F_{\refE}^{\nu}}\left(t\right)\right)\right|\leq\lambda^{n},
\end{equation*}
where $\lambda^{n}$ is given by \eqref{lambda}. Hence, we obtain by \eqref{2O:discreteODE2} as well as the CFL constraint \eqref{CFL1}
for  $\beta=1,...,k+1$ %all $\nu=1,2,3$ and
\begin{align}\label{H2VectorIncreasing}
\begin{split}
& \partial_{b_{2,\beta}}H_{2}\left(b_{1,\beta},b_{2,\beta},b_{3,\beta},c_{2,\beta},\left|K^{n,1}\right|,\left|K^{n}\right|,t_{n}\right) \\
\geq&  1-\frac{\triangle t}{\hat{\sigma}\left|K^{n,1}\right|}\left(\hat{\sigma}\left|\nabla_{\boldsymbol{\xi}}\cdot\left[\mathrm{{\bf A}}_{K^{n}}^{-1}\Big(\boldsymbol{\omega}^{n}\Big)\right]\right|\left|K^{n}\right|+\sum_{\nu=1}^{3}\lambda^{n}\ell_{F_{\refE}^{\nu}}\right)\geq0,
\end{split}
\end{align}
and similar we get
\begin{equation}\label{H1VectorIncreasing}
\partial_{b_{1,\beta}}H_{1}\left(b_{1,\beta},b_{2,\beta},c_{1,\beta},\left|K^{n,1}\right|,\left|K^{n}\right|,t_{n}\right)\geq0,
\end{equation}
\begin{equation}\label{H3VectorIncreasing}
\partial_{b_{3,\beta}}H_{3}\left(b_{2,\beta},b_{3,\beta},c_{3,\beta},\left|K^{n,1}\right|,\left|K^{n}\right|,t_{n}\right)\geq0.
\end{equation}

In the following, we highlight by the symbolic notation $\uparrow$ that a function is increasing in the marked arguments. Likewise, we apply the notation $\boldsymbol{\uparrow}$ to highlight that a function with vector arguments increases in each vector component. Then, it follows by \eqref{H2VectorIncreasing}, \eqref{H1VectorIncreasing} and \eqref{H3VectorIncreasing}
\begin{equation*}
H_{1}\left(\uparrow,\uparrow,\uparrow,\left|K^{n,1}\right|,\left|K^{n}\right|,t_{n}\right), \
H_{2}\left(\uparrow,\uparrow,\uparrow,\uparrow,\left|K^{n,1}\right|,\left|K^{n}\right|,t_{n}\right), \
H_{3}\left(\uparrow,\uparrow,\uparrow,\left|K^{n,1}\right|,\left|K^{n}\right|,t_{n}\right),
\end{equation*}
since $\widehat{g}_{\pm}\left(\boldsymbol{\omega},\uparrow,J_{K\left(t\right)}\tilde{\n}\left(t\right)\right)$ for all $t\in\left[t_{n},t_{n+1}\right]$ and all cells $K\left(t\right)\in\mathcal{T}_{\left(t\right)}$.
Therefore, we obtain
\begin{equation}\label{LVectorIncreasing}
\mathcal{L}\left(\boldsymbol{\uparrow},\boldsymbol{\uparrow},\boldsymbol{\uparrow},\left|K^{n,1}\right|,\left|K^{n}\right|,t_{n}\right),
\end{equation}
since for all $\gamma =1,...,L$, it follows by \eqref{2O:discreteODE2} and the CFL constraint \eqref{CFL1}
\begin{equation*}
\partial_{a_{\gamma}}\mathcal{L}\left({\bf a},{\bf b},{\bf c},\left|K^{n,1}\right|,\left|K^{n}\right|,t_{n}\right)
\geq\widetilde{\sigma}_{\gamma}\left(1-\frac{\triangle t}{\left|K^{n,1}\right|}\left|\nabla_{\boldsymbol{\xi}}\cdot\left[\mathrm{{\bf A}}_{K^{n}}^{-1}\Big(\boldsymbol{\omega}^{n}\Big)\right]\right|\left|K^{n}\right|\right)\geq0.
\end{equation*}
Finally, the equation \eqref{LVectorConstant} and \eqref{LVectorIncreasing} provide the inequality \eqref{LVectorBounds}.
\end{proof}

\begin{lemma}\label{MonotoneTVDRK2c}
Let $\left|K^{n}\right|$, $\left|K^{n,1}\right|$, $\left|K^{n+1}\right|$ be given by \eqref{2O:discreteODE1} and ${\bf a},\widetilde{{\bf a}}\in\left[m,M\right]^{L}$, ${\bf b}, \widetilde{{\bf b}},\widetilde{{\bf c}}\in\mathbb{M}$. Then, under the CFL constraint
\begin{equation}\label{CFL2}
\max\left\{ \hat{\sigma}\left|\nabla_{\boldsymbol{\xi}}\cdot\left[\mathrm{{\bf A}}_{K(t)}^{-1}\big(\boldsymbol{\omega}\left(t\right)\big)\right]\right|\left|K\left(t\right)\right|+\sum_{\nu=1}^{3}\lambda^{n}\ell_{F_{\refE}^{\nu}}:\ t\in\left[t_{n},t_{n+1}\right]\right\} \frac{\triangle t}{\left|K^{n+1}\right|}\leq\hat{\sigma}
\end{equation}
with $\lambda^{n}$ given by \eqref{lambda}, we have
\begin{equation}\label{HandLVectorBounds}
m\leq\frac{1}{2}\mathcal{H}\left({\bf a},{\bf b},\left|K^{n+1}\right|,\left|K^{n}\right|\right)+\frac{1}{2}\mathcal{L}\left(\widetilde{{\bf a}},\widetilde{{\bf b}},\widetilde{{\bf c}},\left|K^{n+1}\right|,\left|K^{n,1}\right|,t_{n+1}\right)\leq M.
\end{equation}
\end{lemma}
\begin{proof}
The equations \eqref{2O:discreteODE2} and \eqref{2O:discreteODE3} supply
\begin{equation}\label{2O:discreteODE4}
\left|K^{n+1}\right|-\left|K^{n,1}\right|=\frac{\triangle t}{2}\left(\nabla_{\boldsymbol{\xi}}\cdot\left[\mathrm{{\bf A}}_{K^{n+1}}^{-1}\Big(\boldsymbol{\omega}^{n+1}\Big)\right]\left|K^{n,1}\right|-\nabla_{\boldsymbol{\xi}}\cdot\left[\mathrm{{\bf A}}_{K^{n}}^{-1}\Big(\boldsymbol{\omega}^{n}\Big)\right]\left|K^{n}\right|\right).
\end{equation}
Thus, the same procedure as in the proof of \cref{MonotoneTVDRK2c} provides
\begin{equation}\label{HandLVectorIncreasing1}
\mathcal{L}\left(\boldsymbol{\uparrow},\boldsymbol{\uparrow},\boldsymbol{\uparrow},\left|K^{n+1}\right|,\left|K^{n,1}\right|,t_{n+1}\right)
\end{equation}
by applying the identity \eqref{2O:discreteODE4} and the CFL constraint \eqref{CFL2}. Furthermore, by \eqref{2O:discreteODE2} and \eqref{2O:discreteODE3} follows
\begin{equation}\label{2O:discreteODE5}
\left|K^{n+1}\right|-\left|K^{n}\right|
=
\frac{\triangle t}{2}\left(\nabla_{\boldsymbol{\xi}}\cdot\left[\mathrm{{\bf A}}_{K^{n}}^{-1}\Big(\boldsymbol{\omega}^{n}\Big)\right]\left|K^{n}\right|+\nabla_{\boldsymbol{\xi}}\cdot\left[\mathrm{{\bf A}}_{K^{n+1}}^{-1}\Big(\boldsymbol{\omega}^{n+1}\Big)\right]\left|K^{n,1}\right|\right).
\end{equation}
Hence, we obtain for all $\gamma =1,...,L$
\begin{align*}
& \partial_{a_{\gamma}}\mathcal{H}\left({\bf a},{\bf b}, \left|K^{n+1}\right|,\left|K^{n}\right|,t_{n+1}\right) \nonumber \\
\geq & \widetilde{\sigma}_{\gamma}\left(1-\frac{\triangle t}{2\left|K^{n+1}\right|}\left(\nabla_{\boldsymbol{\xi}}\cdot\left[\mathrm{{\bf A}}_{K^{n}}^{-1}\Big(\boldsymbol{\omega}^{n}\Big)\right]\left|K^{n}\right|+\left|\nabla_{\boldsymbol{\xi}}\cdot\left[\mathrm{{\bf A}}_{K^{n+1}}^{-1}\Big(\boldsymbol{\omega}^{n+1}\Big)\right]\right|\left|K^{n,1}\right|\right)\right)\geq0,
\end{align*}
by the identity \eqref{2O:discreteODE5} and the CFL constraint \eqref{CFL2}. In a similar way, it follows for all $\nu=1,2,3$ as well as $\beta =1,...,k+1$
\begin{equation*}
\partial_{b_{\nu,\beta}}\mathcal{H}\left({\bf a},\boldsymbol{b}, \left|K^{n+1}\right|,\left|K^{n}\right|,t_{n+1}\right)\geq0.
\end{equation*}
This ensures
\begin{equation}\label{HandLVectorIncreasing2}
\mathcal{H}\left(\boldsymbol{\uparrow},\boldsymbol{\uparrow},\left|K^{n+1}\right|,\left|K^{n}\right|\right).
\end{equation}
Therefore, the equation \eqref{HandLVectorConstant}, \eqref{HandLVectorIncreasing1} and \eqref{HandLVectorIncreasing2} supply the inequality \eqref{HandLVectorBounds}.
\end{proof}

We note that the assumption ($A3$) for the mesh parameter provides for all $t\in \left[0,T\right]$ and all cells $K\left(t\right)\in\mathcal{T}_{\left(t\right)}$
\begin{equation*}
\left|K\left(t\right)\right|\geq\pi\rho_{K\left(t\right)}^{2}\geq\frac{\pi}{\kappa^{2}}h_{K\left(t\right)}^{2}\geq\frac{\pi}{\tau^{2}\kappa^{2}}h^{2},
\end{equation*}
since $\rho_{K\left(t\right)}$ denotes the radius of the largest ball contained in the cell $K\left(t\right)$. Therefore, the CFL constraints \eqref{CFL1} and \eqref{CFL2} can be generalized as
\begin{equation}\label{globalCFL}
\max\left\{ \hat{\sigma}\left|\nabla_{\boldsymbol{\xi}}\cdot\left[\mathrm{{\bf A}}_{K(t)}^{-1}\big(\boldsymbol{\omega}\left(t\right)\big)\right]\right|\left|K\left(t\right)\right|+\sum_{\nu=1}^{3}\lambda^{n}\ell_{F_{\refE}^{\nu}}:\ t\in\left[t_{n},t_{n+1}\right]\right\} \frac{\triangle t}{h^{2}}\leq\frac{\pi\hat{\sigma}}{\tau^{2}\kappa^{2}}.
\end{equation}
Now, we apply the generalized CFL constraint and the previous lemmas to prove the maximum principle for the second order fully-discrete ALE-DG method.
\begin{theorem}\label{MPTVDRK2}
%Let $K^{n}$ and $K^{n+1}$ be arbitrary cells, which are connected by the time-dependent straight lines \eqref{vertices},
Suppose ${\bf \widetilde{u}}^{n,*,\text{int}_{\refE}}\in\left[m,M\right]^{L}\ensuremath{,}{\bf u}^{n,*,\text{int}_{\refE}},{\bf u}^{n,*,\text{ext}_{\refE}}\in\mathbb{M}$. Furthermore, ($A1$) - ($A3$) and the CFL constraint \eqref{globalCFL} are satisfied. Then, the solution $u_{h}^{n+1}$ of the second order fully-discrete ALE-DG method revised by Zhang, Xia and Shu's bound preserving limiter \eqref{Limiter} belongs to the interval $\left[m,M\right]$.
\end{theorem}

\begin{proof}
Let $K^{n}\in\mathcal{T}_{\left(t_{n}\right)}$ be an arbitrary cell. The main part of the proof follows in two steps.

\textit{Step 1.}
Since, ${\bf \widetilde{u}}^{n,*,\text{int}_{\refE}}\in\left[m,M\right]^{L}\ensuremath{,}{\bf u}^{n,*,\text{int}_{\refE}},{\bf u}^{n,*,\text{ext}_{\refE}}\in\mathbb{M}$ and the CFL constraint \eqref{globalCFL} is satisfied, it follows $\overline{u}_{K^{n,1}}^{n,1}\in\left[m,M\right]$ by the inequality \eqref{LVectorBounds} in \cref{MonotoneTVDRK2b} and the equation \eqref{E2A:TVDRK2}. Hence, the bound preserving limiter \eqref{Limiter} ensures $u_{h}^{n,1}|_{K^{n,1}}\in\left[m,M\right]$.

\textit{Step 2.}
In the first step, it has been shown that $u_{h}^{n,1}|_{K^{n,1}}\in\left[m,M\right]$ when the bound preserving limiter \eqref{Limiter} was applied. Hence, ${\bf \widetilde{u}}^{n,1*,\text{int}_{\refE}}\in\left[m,M\right]^{L}$, ${\bf u}^{n,1,*,\text{int}_{\refE}},{\bf u}^{n,1,*,\text{ext}_{\refE}}\in\mathbb{M}$.
%since the step 1 of this proof provides that the solution $u_{h}^{n,1}$ revised by the bound preserving limiter \eqref{Limiter} and shrunk to the cell $K^{n,1}$ belongs to the interval $\left[m,M\right]$.
Thus, the inequality \eqref{HandLVectorBounds} in \cref{MonotoneTVDRK2c} and the equation \eqref{E2B:TVDRK2} supply $\overline{u}_{K^{n+1}}^{n+1}\in\left[m,M\right]$, since the CFL constraint \eqref{globalCFL} is satisfied. Finally, the bound preserving limiter \eqref{Limiter} ensures that $u_{h}^{n+1}|_{K^{n+1}}\in\left[m,M\right]$.

In a similar way, we proceed for any other cell in $\mathcal{T}_{\left(t_{n+1}\right)}$. Thus, it follows $u_{h}^{n+1}\in\left[m,M\right]$.
\end{proof}

% The maximum principle for the third order fully-discrete ALE-DG method-------------------------------------------------------------
\subsubsection{The maximum principle for the third order fully-discrete ALE-DG method}\label{MPthirdorder}
In this section, we show briefly how the result in \cref{MPTVDRK2} can be extended to the third order TVD-RK3 ALE-DG method. We apply the decomposition \eqref{cellaverage} of the cell average values and the vector notations \eqref{vectornotation} to rewrite the scheme \eqref{TVDRK3} as
\begin{align*}
\overline{u}_{K^{n,1}}^{n,1}=&\mathcal{L}\left({\bf \widetilde{u}}^{n,*,\text{int}_{\refE}},{\bf u}^{n,*,\text{int}_{\refE}},{\bf u}^{n,*,\text{ext}_{\refE}},\left|K^{n,1}\right|,\left|K^{n}\right|,t_{n}\right), \\ %\label{E2A:TVDRK3} \\
\overline{u}_{K^{n,2}}^{n,2}=& \frac{3}{4}\mathcal{H}\left({\bf \widetilde{u}}^{n,*,\text{int}_{\refE}},{\bf u}^{n,*,\text{int}_{\refE}},\left|K^{n,2}\right|,\left|K^{n}\right|\right) \nonumber  \\
&+\frac{1}{4}\mathcal{L}\left({\bf \widetilde{u}}^{n,1,*,\text{int}_{\refE}},{\bf u}^{n,1,*,\text{int}_{\refE}},{\bf u}^{n,1,*,\text{ext}_{\refE}},\left|K^{n,2}\right|,\left|K^{n,1}\right|,t_{n+1}\right),\\ %\label{E2B:TVDRK3} \\
\overline{u}_{K^{n+1}}^{n+1}=& \frac{1}{3}\mathcal{H}\left({\bf \widetilde{u}}^{n,*,\text{int}_{\refE}},{\bf u}^{n,*,\text{int}_{\refE}},\left|K^{n+1}\right|,\left|K^{n}\right|\right)\nonumber  \\
&  +\frac{2}{3}\mathcal{L}\left({\bf \widetilde{u}}^{n,2,*,\text{int}_{\refE}},{\bf u}^{n,2,*,\text{int}_{\refE}},{\bf u}^{n,2,*,\text{ext}_{\refE}},\left|K^{n+1}\right|,\left|K^{n,2}\right|,t_{n+\frac{1}{2}}\right). %\label{E2C:TVDRK3}
\end{align*}

By the \Cref{GCL} we obtain the identities
\begin{equation*}
\mathcal{L}\left({\bf c},{\bf c},{\bf c},\left|K^{n,1}\right|,\left|K^{n}\right|,t_{n}\right)=c,
\end{equation*}
\begin{equation*}
\frac{3}{4}\mathcal{H}\left({\bf c},{\bf c},\left|K^{n,2}\right|,\left|K^{n}\right|\right)+\frac{1}{4}\mathcal{L}\left({\bf c},{\bf c},{\bf c},\left|K^{n,2}\right|,\left|K^{n,1}\right|,t_{n+1}\right)=c,
\end{equation*}
\begin{equation*}
\frac{1}{3}\mathcal{H}\left({\bf c},{\bf c},\left|K^{n+1}\right|,\left|K^{n}\right|\right)+\frac{2}{3}\mathcal{L}\left({\bf c},{\bf c},{\bf c},\left|K^{n+1}\right|,\left|K^{n,2}\right|,t_{n+\frac{1}{2}}\right)=c,
\end{equation*}
where $\left({\bf c},{\bf c},{\bf c}\right)\in\left[m,M\right]^{L}\cup\mathbb{M}$ with vector components given by the constant $c\in \left[m,M\right]$. Moreover, like in the previous section, we apply the Lax-Friedrichs flux \eqref{Lax-Friedrichs}. Then, by the same argumentation as in the proof of the  \cref{MonotoneTVDRK2b} and \cref{MonotoneTVDRK2c}, we obtain the following lemma.
\begin{lemma}
Let $\left|K^{n}\right|$, $\left|K^{n,1}\right|$, $\left|K^{n,2}\right|$, $\left|K^{n+1}\right|$ be given by \eqref{3O:discreteODE1} and ${\bf a}, {\bf \widetilde{a}}\in\left[m,M\right]^{L}$, ${\bf b},{\bf c},{\bf \widetilde{b}},{\bf \widetilde{c}}\in\mathbb{M}$. Then, under the CFL constraint \eqref{globalCFL}, it holds
\begin{equation*}
m\leq\mathcal{L}\left({\bf a},{\bf b},{\bf c},\left|K^{n,1}\right|,\left|K^{n}\right|,t_{n}\right)\leq M,
\end{equation*}
\begin{equation*}
m\leq\frac{3}{4}\mathcal{H}\left({\bf a},{\bf b},\left|K^{n,2}\right|,\left|K^{n}\right|\right)+\frac{1}{4}\mathcal{L}\left(\widetilde{{\bf a}},\widetilde{{\bf b}},\widetilde{{\bf c}},\left|K^{n,2}\right|,\left|K^{n,1}\right|,t_{n+1}\right)\leq M,
\end{equation*}
\begin{equation*}
m\leq \frac{1}{3}\mathcal{H}\left({\bf a},{\bf b},\left|K^{n+1}\right|,\left|K^{n}\right|\right)+\frac{2}{3}\mathcal{L}\left(\widetilde{{\bf a}},\widetilde{{\bf b}},\widetilde{{\bf c}},\left|K^{n+1}\right|,\left|K^{n,2}\right|,t_{n+\frac{1}{2}}\right)\leq M.
\end{equation*}
\end{lemma}

This lemma provide the following analogue of the Theorem \ref{MPTVDRK2} for the second order fully-discrete ALE-DG method. The proof follows similar to the proof of Theorem \ref{MPTVDRK2}. Hence, it is skipped.
\begin{theorem}
%Let $K^{n}$ and $K^{n+1}$ be arbitrary cells, which are connected by the time-dependent straight lines \eqref{vertices},
Suppose ${\bf \widetilde{u}}^{n,*,\text{int}_{\refE}}\in\left[m,M\right]^{L}$, ${\bf u}^{n,*,\text{int}_{\refE}},{\bf u}^{n,*,\text{ext}_{\refE}}\in\mathbb{M}$. Furthermore, ($A1$) - ($A3$) and the CFL constraint \eqref{globalCFL} are satisfied. Then, the solution $u_{h}^{n+1}$ of the third order fully-discrete ALE-DG method revised by Zhang, Xia and Shu's bound preserving limiter \eqref{Limiter} belongs to the interval $\left[m,M\right]$.
\end{theorem}
\begin{remark}
It is also possible to apply other TVD-RK methods like the five stage fourth order method of Spiteri and Ruuth \cite{Spiteri} as time integrator in a fully-discrete ALE-DG method. Then, it can be proven by the techniques presented in this section that these fully-discrete ALE-DG methods satisfy the maximum principle.
\end{remark}

% Numerical experiments --------------------------------------------------------------------------------------------------------
\section{Numerical experiments}
\label{sec:Numerical}
In this section, we demonstrate the performance of the ALE-DG method for conservation laws in two dimensions. In our simulation, the criss-triangular meshes are used. Furthermore, the third order TVD Runge-Kutta method is used in the first example (Example 5.1). In the other examples the five stage fourth order TVD Runge-Kutta method of Spiteri and Ruuth \cite{Spiteri} is used for the time discretization. We observe that for this high order approximation in time our theoretical results hold numerically, too. In order to avoid complications with the stability of the explicit time integrator, we apply the suitable CFL condition dependent on equation \eqref{globalCFL}.

To verify our theoretical results, we present numerical simulations for a linear advection equation and Burgers' equation. Moreover, to highlight that the ALE-DG method can be also used for systems of conservation laws, we present a plain wave problem and a smooth vortex problem for the compressible Euler equations with a polytropic gas.

In all the numerical simulations, we consider two moving mesh scenarios. First a static uniform criss-triangular mesh with cell size $h_0$ is used. Next, a moving mesh with the grid point distribution
\begin{align}\label{GridpointDistribution}
\begin{split}
&x_j(t_n)= x_j(0) + 0.3\sin\left(\frac{2\pi x_j(0)}{x_r-x_l}\right)\sin\left(\frac{2\pi y_j(0)}{y_r-y_l}\right)\sin(2\pi(t_n)/t_0),\\
&y_j(t_n)= y_j(0) + 0.2\sin\left(\frac{2\pi x_j(0)}{x_r-x_l}\right)\sin\left(\frac{2\pi y_j(0)}{y_r-y_l}\right)\sin(4\pi(t_n)/t_0)
\end{split}
\end{align}
is used. In equation \eqref{GridpointDistribution} the points $(x_j,y_j)$ are the vertices of the triangular mesh and $t_0=\sqrt{10^2+5^2}$. The vertices $(x_j(0),y_j(0))$ at initial time are given by the same mesh as in the first moving mesh scenario. The grid point distribution \eqref{GridpointDistribution} has also been used by  Klingenberg \cite{KlingenbergHJ2016HJ} et al. and Persson et al. \cite{Persson2009}. As an example, in Fig. \ref{figuremesh} we draw a typical mesh at $t=0$ and the deformed one at $t=1$.

\begin{figure}
\begin{center}
\includegraphics[width=2.5in]{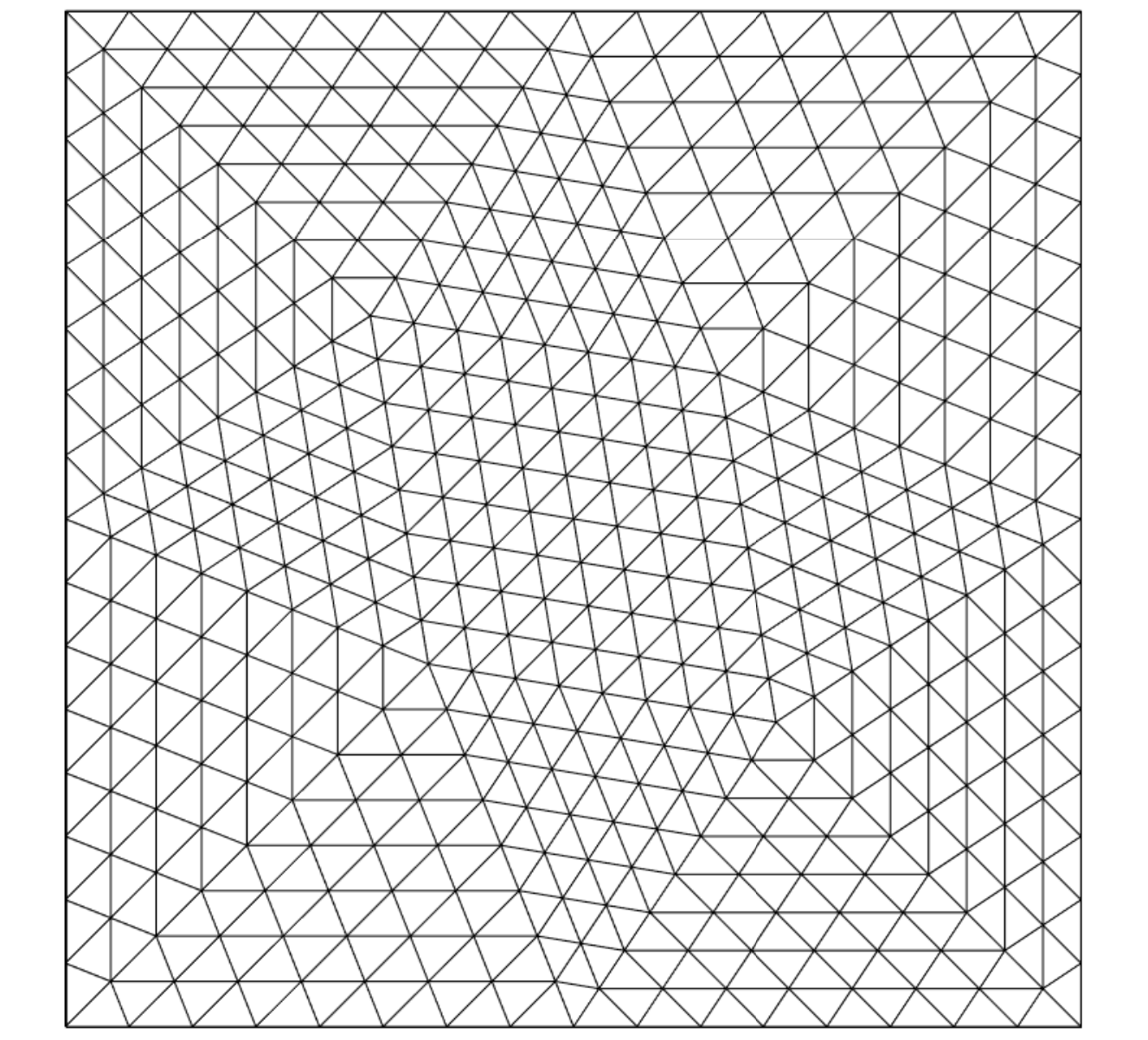}
\includegraphics[width=2.7in]{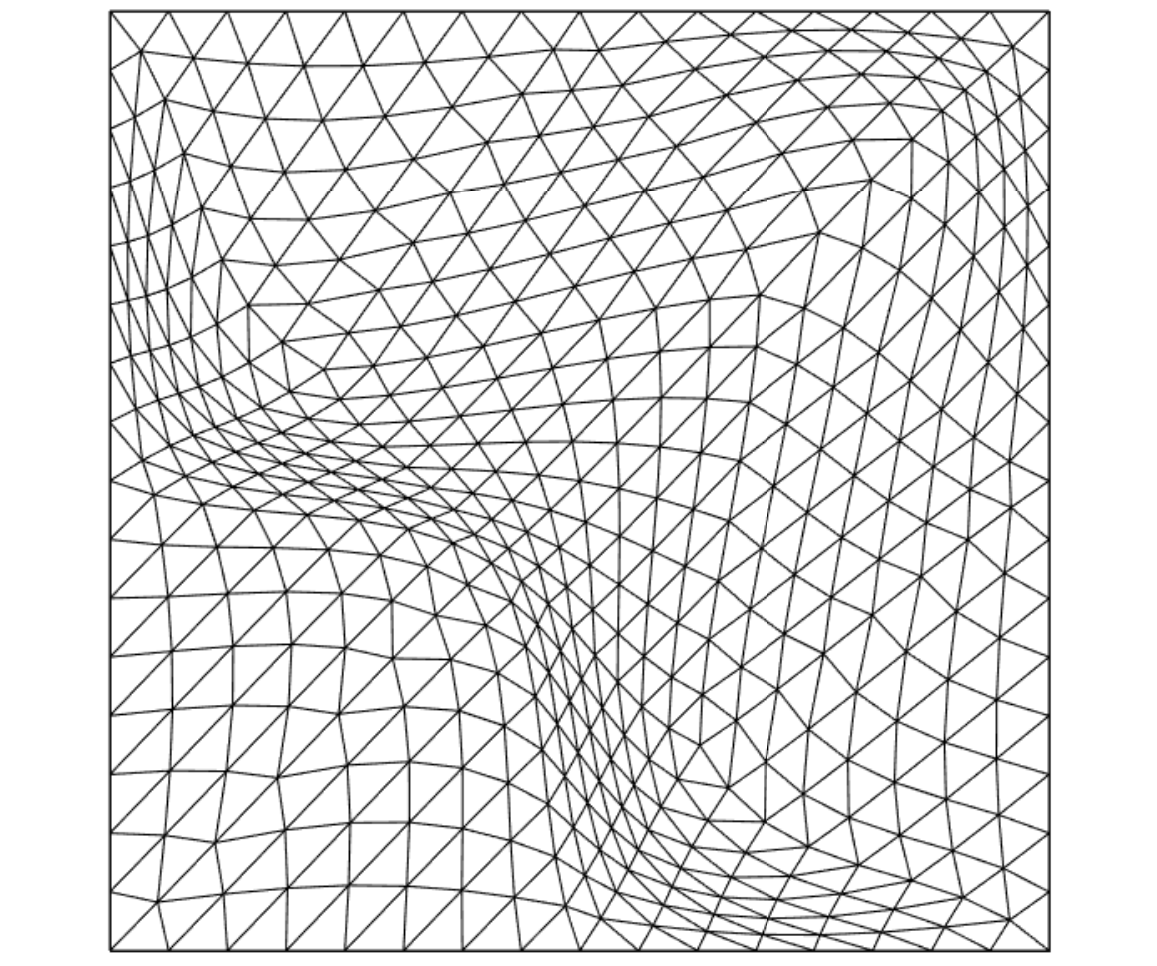}
\caption{The typical mesh at $t=0$ (left) and deformation mesh at $t=1$(right).}
\label{figuremesh}
\end{center}
\end{figure}

% Example 5.1: Linear advection equation --------------------------------------------------------------------------------------------------------
\paragraph{Example 5.1}\label{Linear}  (Linear advection equation)
Here, we test the linear equation
\begin{align}\label{LinearProblem}
&\partial_t u + \partial_x u+\partial_y u=0, \quad (x,y)\in [0,2]\times[0,2],
\end{align}
with the periodic boundary condition and the initial condition is taken to be $u_0(x,y)=1+0.5\sin(\pi(x+y))$. The exact solution is $u(x,y,t)=u_0(x-t,y-t)$ at time $t$.

The numerical solution on the static uniform grid is $u_h^S$ and the moving mesh solution is $u_h^M$. In Table \ref{LinearTable1}, we show the $\mathrm{L}^{2}$-errors and the rates of convergence of the numerical solutions $u_h^S$ and $u_h^M$  for the advection equation at time $t=1$. In the computation, we used piecewise $P^k$ polynomial spaces with $k=1,2,3$ on static and moving triangular meshes with cell size $h_0$. We observe that both $u_h^S$ and $u_h^M$  have the optimal accuracy when $P^k$ polynomial spaces with $k=1,2,3$ are applied, but they do not satisfy the discrete maximum principle. This was expected, since the bound preserving limiter in Section \ref{pplimiter} was not applied in the test case.

To verify the maximum principle, we run the same test case again, but this time with the bound preserving limiter. The results are given in Table \ref{Linearpptable}, where $\widetilde{u}_h^M$ is the ALE-DG solution on the moving mesh with bound preserving limiter. We observe that the $\mathrm{L}^{2}$-errors and the rates of convergence for the moving mesh solution $\widetilde{u}_h^M$ are not affected by the use of the bound preserving limiter and thus the method still has optimal accuracy. Furthermore, we  can see that the numerical solution $\widetilde{u}_h^M$ is limited in the same range $[0.5,1.5]$ as the initial data.

\begin{table}[!htb]
\caption{\label{LinearTable1} $\mathrm{L}^{2}$-errors and the rates of convergence for the linear advection equation \eqref{LinearProblem} at final time $t=1$ on static (right) and moving (left) triangular meshes  with cell size $h_0$. The bound preserving limiter is not applied.} \centering
\begin{scriptsize}
\begin{center}
\begin{tabular}{|c|c|cccc|cccc|}
  \hline
           &       && $u-u_h^M$     &&                  & & $u-u_h^S$                &  &    \\\cline{3-6}\cline{7-10}
           & $h_0$ &$ \mathrm{L}^{2}$ norm  & order      &$\min{(1.5-u_h^M)}$&$\min{(u_h^M-0.5)}$  &$ \mathrm{L}^{2}$ norm   & order&$\min{(1.5-u_h^S)}$&$\min{(u_h^S-0.5)}$\\\hline
$P^1$ &	1/2	&	1.30E-01	&	--	&	2.22E-02	&	6.11E-02	&	1.11E-01	&	--	&	2.39E-02	&	2.49E-02	\\
&	1/4	&	3.09E-02	&	2.07 	&	-7.38E-02	&	-8.71E-03	&	2.09E-02	&	2.41 	&	-3.50E-02	&	-1.89E-02	\\
&	1/8	&	6.77E-03	&	2.19 	&	0.00E+00	&	9.60E-05	&	4.43E-03	&	2.24 	&	-1.29E-02	&	-1.11E-02	\\
&	1/16	&	1.59E-03	&	2.09 	&	-7.27E-03	&	-4.64E-04	&	1.04E-03	&	2.09 	&	-3.47E-03	&	-3.41E-03	\\
&	1/32	&	3.88E-04	&	2.03 	&	-1.85E-03	&	-1.10E-04	&	2.54E-04	&	2.04 	&	-9.40E-04	&	-9.27E-04	\\
\hline
$P^2$ &	1/2	&	2.30E-02	&	--	&	-1.01E-01	&	-2.77E-02	&	1.79E-02	&	--	&	-7.32E-02	&	-4.33E-02	\\
&	1/4	&	4.88E-03	&	2.24 	&	-3.11E-02	&	-4.04E-03	&	3.00E-03	&	2.57 	&	-9.71E-03	&	-5.66E-03	\\
&	1/8	&	7.64E-04	&	2.68 	&	-4.24E-03	&	-1.10E-03	&	4.09E-04	&	2.88 	&	-8.10E-04	&	-1.06E-03	\\
&	1/16	&	1.03E-04	&	2.88 	&	-4.30E-04	&	-1.14E-04	&	5.12E-05	&	3.00 	&	-1.30E-04	&	-1.87E-04	\\
&	1/32	&	1.31E-05	&	2.98 	&	-6.00E-05	&	-2.80E-05	&	6.28E-06	&	3.03 	&	-2.00E-05	&	-2.50E-05	\\
\hline
$P^3$&	1/2	&	4.05E-03	&	--	&	-2.60E-03	&	-9.28E-04	&	2.01E-03	&	--	&	-5.50E-04	&	-2.20E-03	\\
&	1/4	&	3.12E-04	&	3.70 	&	-1.00E-05	&	-1.39E-04	&	1.30E-04	&	3.95 	&	-6.00E-05	&	-1.29E-04	\\
&	1/8	&	1.93E-05	&	4.02 	&	1.00E-05	&	-6.00E-06	&	7.85E-06	&	4.05 	&	0.00E+00	&	-9.00E-06	\\
&	1/16	&	1.22E-06	&	3.98 	&	0.00E+00	&	-1.00E-06	&	4.79E-07	&	4.03 	&	0.00E+00	&	0.00E+00	\\
&	1/32	&	7.71E-08	&	3.98 	&	0.00E+00	&	0.00E+00	&	2.96E-08	&	4.02 	&	0.00E+00	&	0.00E+00	\\
\hline
\end{tabular}
\end{center}
\end{scriptsize}
\end{table}

\begin{table}[!htb]
\caption{\label{Linearpptable} $\mathrm{L}^{2}$-errors and the rates of convergence for the moving mesh ALE-DG solution $\widetilde{u}_h^M$ with the bound preserving limiter at final time $t=1$ for the linear advection equation \eqref{LinearProblem} on moving triangular meshes with cell size $h_0$.} \centering
\begin{center}
\begin{tabular}{|c|c|cccc|}
  \hline
& $h_0$ &$ \norm{u-\widetilde{u}_h^M}$   & order    &$\min{(1.5-\widetilde{u}_h^M)}$   & $\min{(\widetilde{u}_h^M-0.5)}$\\\hline
$P^1$ &	1/2	&	1.36E-01	&	--	&	4.17E-02	&	7.82E-02	\\
&	1/4	&	3.31E-02	&	2.04 	&	0.00E+00	&	1.99E-03	\\
&	1/8	&	7.94E-03	&	2.06 	&	0.00E+00	&	5.05E-03	\\
&	1/16	&	1.84E-03	&	2.11 	&	0.00E+00	&	1.09E-03	\\
&	1/32	&	4.41E-04	&	2.06 	&	0.00E+00	&	2.72E-04	\\
\hline
$P^2$ &	1/2	&	6.26E-02	&	--	&	0.00E+00	&	2.20E-02	\\
&	1/4	&	1.07E-02	&	2.54 	&	0.00E+00	&	8.60E-04	\\
&	1/8	&	1.18E-03	&	3.19 	&	0.00E+00	&	3.64E-05	\\
&	1/16	&	1.23E-04	&	3.26 	&	0.00E+00	&	4.52E-06	\\
&	1/32	&	1.46E-05	&	3.08 	&	0.00E+00	&	2.13E-08	\\
\hline	
$P^3$ &	1/2	&	5.96E-03	&	--	&	0.00E+00	&	1.63E-03	\\
&	1/4	&	4.69E-04	&	3.67 	&	2.43E-04	&	1.98E-04	\\
&	1/8	&	3.02E-05	&	3.96 	&	4.66E-05	&	9.56E-06	\\
&	1/16	&	1.76E-06	&	4.10 	&	1.47E-06	&	1.62E-06	\\
&	1/32	&	1.01E-07	&	4.13 	&	0.00E+00	&	0.00E+00	\\
\hline
\end{tabular}
\end{center}
\end{table}

% Example 5.2: Burgers' equation --------------------------------------------------------------------------------------------------------
\paragraph{Example 5.2}  (Burgers' equation)
Next, we investigate the Burgers' equation
\begin{align}\label{BurgersProblem}
&\partial_t u + \partial_x \bigg(\frac{u^2}{2}\bigg)+\partial_y \bigg(\frac{u^2}{2}\bigg)=0, \quad (x,y)\in [0,2]\times[0,2].
\end{align}
In our simulation, the periodic boundary condition is used and the initial condition  is also taken to be $u_0(x,y)=1+0.5\sin(\pi(x+y))$. We compute this example up to time $t=0.1$ before the shock front has been developed in the numerical solution.

In Table \ref{BurgesTable1}, the $\mathrm{L}^{2}$-errors and the rates of convergence for the numerical solutions $u_h^S$ and $u_h^M$, $\tilde{u}_h^M$ are presented. These functions are computed by the ALE-DG method with $P^k$, $k=1,2,3$, polynomial spaces. As in the previous example $u^S_h$ and $u^M_h$ are the numerical solutions of the ALE-DG method on the static uniform mesh and on the moving mesh. The solution $\tilde{u}_h^M$ is the moving mesh ALE-DG solution revised by the bound preserving limiter. We observe that the optimal accuracy is obtained for $u_h^M,\widetilde{u}_h^M$ and $u_h^S$ in both moving mesh scenarios. Furthermore, maximum and minimum values of $\widetilde{u}_h^M$ are limited in the same range $[0.5,1.5]$  as the initial data when the bound preserving limiter is applied in the ALE-DG method on the moving grid.
\begin{table}[!htb]
\caption{\label{BurgesTable1} $\mathrm{L}^{2}$-errors and rates of convergence for Burgers' equation \eqref{BurgersProblem} at final time $t=0.1$ on static (left) and moving (center) triangular meshes with cell size $h_0$. On the right the $\mathrm{L}^{2}$-errors, rates of convergence and bounds for the ALE-DG solution revised by the bound preserving limiter.} \centering
\begin{scriptsize}
\begin{center}
\begin{tabular}{|c|c|cc|cc|cccc|}
  \hline
& $h_0$ &$\norm{u-u_h^S}$& order &$ \norm{u-u_h^M}$   & order      &$ \norm{u-\widetilde{u}_h^M}$   & order&$\min{(1.5-\widetilde{u}_h^M)}$&$\min{(\widetilde{u}_h^M-0.5)}$\\\hline
$P^1$ &	 1/2	&	6.15E-02	&	--	&	6.21E-02	&	--	&	6.18E-02	&	--	&	0.00E+00	&	0.00E+00	\\
&	 1/4	&	1.78E-02	&	1.79 	&	1.65E-02	&	1.91 	&	1.58E-02	&	1.97 	&	0.00E+00	&	0.00E+00	\\
&	 1/8	&	4.18E-03	&	2.09 	&	3.89E-03	&	2.09 	&	3.87E-03	&	2.03 	&	0.00E+00	&	0.00E+00	\\
&	  1/16	&	1.02E-03	&	2.04 	&	9.44E-04	&	2.04 	&	9.82E-04	&	1.98 	&	0.00E+00	&	0.00E+00	\\
&	  1/32	&	2.49E-04	&	2.03 	&	2.31E-04	&	2.03 	&	2.40E-04	&	2.03 	&	0.00E+00	&	0.00E+00	\\
\hline
$P^2$&	 1/2	&	2.54E-02	&	--	&	2.54E-02	&	--	&	4.71E-02	&	--	&	0.00E+00	&	0.00E+00	\\
&	 1/4	&	4.16E-03	&	2.61 	&	4.10E-03	&	2.63 	&	1.23E-02	&	1.93 	&	0.00E+00	&	0.00E+00	\\
&	 1/8	&	7.02E-04	&	2.57 	&	6.72E-04	&	2.61 	&	8.18E-04	&	3.91 	&	0.00E+00	&	0.00E+00	\\
&	  1/16	&	1.14E-04	&	2.62 	&	1.08E-04	&	2.64 	&	1.10E-04	&	2.90 	&	0.00E+00	&	0.00E+00	\\
&	  1/32	&	1.66E-05	&	2.78 	&	1.59E-05	&	2.77 	&	1.59E-05	&	2.78 	&	0.00E+00	&	0.00E+00	\\
\hline
$P^3$&	 1/2	&	7.70E-03	&	--	&	7.70E-03	&	--	&	1.22E-02	&	--	&	0.00E+00	&	0.00E+00	\\
&	 1/4	&	8.82E-04	&	3.12 	&	9.17E-04	&	3.07 	&	1.07E-03	&	3.51 	&	0.00E+00	&	0.00E+00	\\
&	 1/8	&	6.44E-05	&	3.78 	&	6.15E-05	&	3.90 	&	6.35E-05	&	4.08 	&	0.00E+00	&	0.00E+00	\\
&	  1/16	&	4.18E-06	&	3.95 	&	3.93E-06	&	3.97 	&	4.02E-06	&	3.98 	&	6.08E-07	&	9.90E-09	\\
&	  1/32	&	2.72E-07	&	3.94 	&	2.55E-07	&	3.95 	&	2.59E-07	&	3.96 	&	6.36E-08	&	1.48E-08	\\
\hline
\end{tabular}
\end{center}
\end{scriptsize}
\end{table}

To show that our proposed ALE-DG methods can handle the problem with shocks, we show the numerical solutions $u_h^S$ and $u_h^M$ of Burgers' equation at time $t=0.45$ with piecewise $P^1$ polynomial approximation in Fig. \ref{Burges_shock}. Here, we use the slope limiter developed by Cockburn et al. in \cite{Cockburn1998}. From the results, it can be seen that the ALE-DG methods can capture the shocks well for the Burgers' equation on both static and moving meshes.

\begin{figure}
\begin{center}
\includegraphics[width=2.5in]{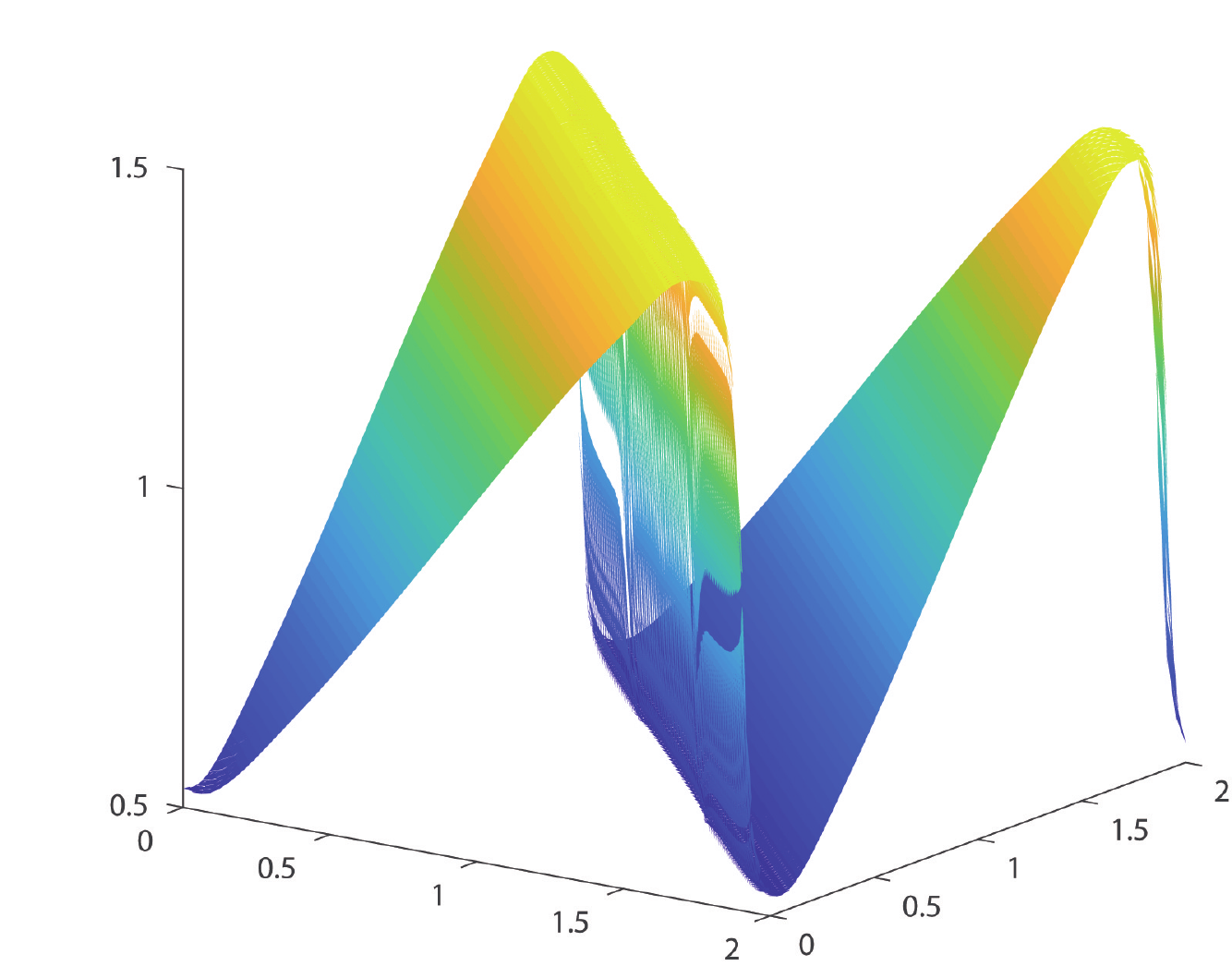}
\includegraphics[width=2.5in]{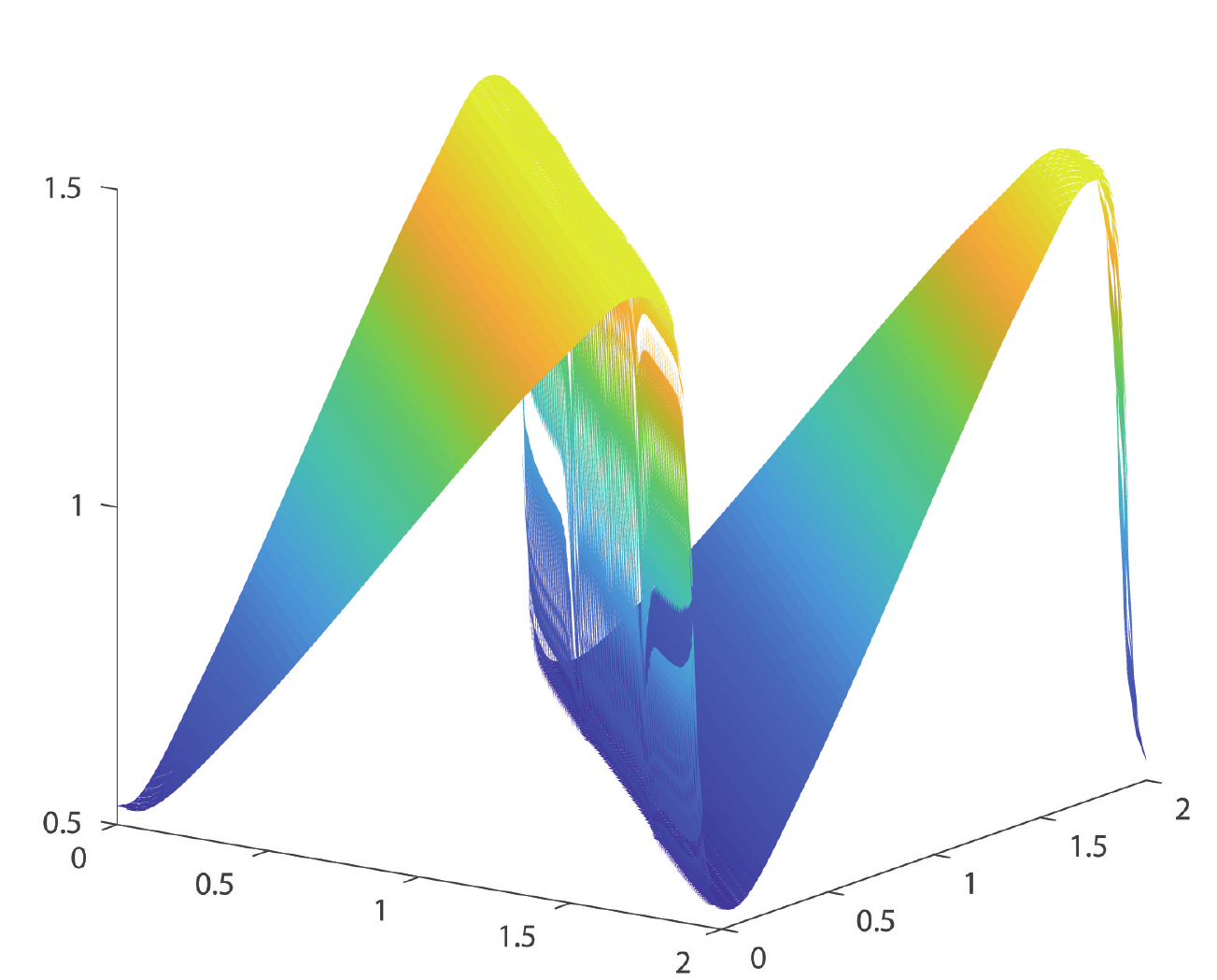}
\caption{The ALE-DG solutions $u_h^S$ (left) and $u_h^M$ (right) at time $t=0.45$ with piecewise $P^1$ polynomial for Burgers' equation.}
\label{Burges_shock}
\end{center}
\end{figure}

% Example 5.3: Compressible Euler equations --------------------------------------------------------------------------------------------------------
\paragraph{Example 5.3}  (Compressible Euler equations)
We consider the two dimensional compressible Euler equations of gas dynamics for a polytropic gas
\begin{align}\label{Euler}
\partial_{t}\mathbf{U}+\nabla\cdot\mathbf{F(U)}=0, \quad (x,y)\in [x_l,x_r]\times[y_l,y_r]\subset\R^2,
\end{align}
with
\begin{align}
\mathbf{U}=(\rho,\rho u,\rho v,E)^{T},\quad
\mathbf{F(U)}=\left[\rho{\bf u},\rho\mathbf{u}\otimes\mathbf{u}+p{\bf I},\left(E+p\right)\mathbf{u}\right]^{T}.
\end{align}
Here, $\rho$ is the density, $\mathbf{u}=\left(u,v\right)^{T}$ is the velocity field and $E$ is the total energy. Moreover, the adiabatic constant of air  $\gamma=1.4$ is used, the pressure is given by $p=(\gamma-1)\left(E-\frac{1}{2}\rho\abs{{\bf u}}^2\right)$ and ${\bf I}$ is the identity matrix.
In our simulation, we test a plain wave problem and a smooth vortex problem.

First, we consider the plain wave problem and choose the domain related parameter in \eqref{Euler} as $x_l=y_l=0$ and $x_r=y_r=2$. The problem has the initial data
\begin{align}
(\rho, u, v, p)^T=(1+0.5\sin(\pi(x+y)),1,1,1)^T
\end{align}
and is investigated with the periodic boundary condition. The results in Table \ref{EulerLinearTable} show the $\mathrm{L}^{2}$-errors and the rates of convergence of the density $\rho_h$ given by the ALE-DG method with $P^k$, $k=1,2,3$, polynomial spaces. The numerical solutions of the ALE-DG method on the static uniform mesh and the moving mesh are $\rho^S_h$ and $\rho^M_h$. The numerical results show that we can obtain the optimal accuracy on both meshes.

\begin{table}[!htb]
\caption{\label{EulerLinearTable}   $\mathrm{L}^{2}$-errors and rates of convergence at final time $t=1$ for the Euler plain wave problem on static (right) and moving (left) triangular meshes with cell size $h_0$.} \centering
\begin{center}
\begin{tabular}{|c|c|cc|cc|}
  \hline
           &       & $\rho-\rho_h^M$              &             & $\rho-\rho_h^S$     &    \\\cline{3-4}\cline{5-6}
           & $h_0$ &$ \mathrm{L}^{2}$-norm   & order    &$ \mathrm{L}^{2}$-norm   & order\\\hline
$P^1$  &	1/2	&	1.35E-01	&	--	&	1.15E-01	&	--	\\
&	1/4	&	3.04E-02	&	2.15 	&	1.81E-02	&	2.67 	\\
&	1/8	&	6.06E-03	&	2.32 	&	3.49E-03	&	2.37 	\\
&	1/16	&	1.40E-03	&	2.11 	&	8.02E-04	&	2.12 	\\
&	1/32	&	3.41E-04	&	2.04 	&	1.93E-04	&	2.05 	\\
\hline
$P^2$ &	1/2	&	2.64E-02	&	--	&	2.23E-02	&	--	\\
&	1/4	&	6.35E-03	&	2.06 	&	4.58E-03	&	2.28 	\\
&	1/8	&	1.08E-03	&	2.56 	&	6.74E-04	&	2.76 	\\
&	1/16	&	1.55E-04	&	2.79 	&	8.62E-05	&	2.97 	\\
&	1/32	&	2.04E-05	&	2.93 	&	1.04E-05	&	3.05 	\\
\hline	
$P^3$ &	1/2	&	4.75E-03	&	--	&	2.37E-03	&	--	\\
&	1/4	&	3.44E-04	&	3.79 	&	1.37E-04	&	4.11 	\\
&	1/8	&	2.02E-05	&	4.09 	&	8.05E-06	&	4.09 	\\
&	1/16	&	1.34E-06	&	3.92 	&	4.92E-07	&	4.03 	\\
&	1/32	&	8.78E-08	&	3.93 	&	3.05E-08	&	4.01 	\\
%\hline
%$P^4$ &	1/2	&	5.52E-04	&	--	&	3.38E-04	&	--	\\
%&	1/4	&	3.50E-05	&	3.98 	&	1.33E-05	&	4.67 	\\
%&	1/8	&	1.46E-06	&	4.58 	&	4.69E-07	&	4.83 	\\
%&	1/16	&	5.44E-08	&	4.75 	&	1.52E-08	&	4.94 	\\
%&	1/32	&	1.86E-09	&	4.87 	&	4.74E-10	&	5.01 	\\
\hline
\end{tabular}
\end{center}
\end{table}

Next, we consider the smooth vortex problem and choose the domain related parameter in \eqref{Euler} as $x_l=y_l=0$, $x_r=20$ and $y_r=15$. This problem was also presented by Persson et al. \cite{Persson2009} and the initial condition is
\begin{align*}
&\rho= \rho_0(1-\alpha e^{r})^{\frac{1}{\gamma-1}}, \quad  \;p= p_0(1-\alpha e^{r})^{\frac{\gamma}{\gamma-1}},\\
&{\bf u}=(u,v)^{T}=\left(u_{0}\cos(\theta),v_{0}\sin(\theta)\right)^{T}+\frac{\epsilon}{2\pi r_{0}}e^{0.5r}\left(-u_{0}(y-y_{0}),v_{0}(x-x_{0})\right)^{T},
\end{align*}
where $(\rho_0,u_0,v_0,p_0)^T=(1,1,1,1)^T$, $\theta=\arctan(0.5)$, $(x_0,y_0)=(5,5)$, $\epsilon=0.3$, $r_0=1.5$, $r=(1-(x-x_0)^2-(y-y_0)^2)/r_0^2$ and $\alpha=\frac{(\gamma-1)\epsilon^2}{8\gamma\pi^2}$. We test the problem up to time $t=\sqrt{10^2+5^2}$ with the Direchlet boundary condition. The ALE-DG method with {$P^k,k=1,2,3$} approxiation is used  to solve the problem  on static uniform triangular meshes with cell size $h_0$ and on moving meshes with the grid point distribution \eqref{GridpointDistribution} as before. The $\mathrm{L}^{2}$-errors and the rates of convergence for the numerical solutions of the density $\rho_h^S,\rho_h^M$ and the pressure $p_h^M$, $p_h^S$ are shown in Table \ref{EulerVortexTable}. We see that the numerical solutions are optimally accurate in both moving mesh scenarios.

\begin{table}[!htb]
\caption{\label{EulerVortexTable} $\mathrm{L}^{2}$-errors and the rates of convergence for the density $\rho$ and the pressure $p$ at final time $t=\sqrt{10^2+5^2}$ for the Euler vortex problem on static (right) and moving (left) triangular meshes with cell size $h_0$.} \centering
\begin{center}
\begin{tabular}{|c|c|cccc|cccc|}
  \hline
           & $h_0$ &$ \norm{\rho-\rho_h^M}$   & order      &$ \norm{p-p_h^M}$ &order&$ \norm{\rho-\rho_h^S}$  & order&$ \norm{p-p_h^S}$&order\\\hline
$P^1$ &	l/2	&	1.35E-03	&	--	&	1.90E-03	&	--	&	1.29E-03	&	--	&	1.81E-03	&	--	\\
&	l/4	&	2.84E-04	&	2.25 	&	3.98E-04	&	2.25 	&	2.64E-04	&	2.29 	&	3.72E-04	&	2.28 	\\
&	l/8	&	5.60E-05	&	2.34 	&	7.83E-05	&	2.35 	&	5.19E-05	&	2.35 	&	7.28E-05	&	2.35 	\\
&	l/16	&	1.22E-05	&	2.19 	&	1.71E-05	&	2.19 	&	1.16E-05	&	2.16 	&	1.63E-05	&	2.16 	\\
&	l/32	&	2.81E-06	&	2.12 	&	3.94E-06	&	2.12 	&	2.71E-06	&	2.10 	&	3.80E-06	&	2.10 	\\
\hline
$P^2$ &	l/2	&	4.75E-04	&	--	&	6.59E-04	&	--	&	4.32E-04	&	--	&	6.00E-04	&	--	\\
&	l/4	&	6.49E-05	&	2.87 	&	9.11E-05	&	2.85 	&	6.34E-05	&	2.77 	&	8.89E-05	&	2.75 	\\
&	l/8	&	1.16E-05	&	2.48 	&	1.63E-05	&	2.48 	&	9.43E-06	&	2.75 	&	1.32E-05	&	2.75 	\\
&	l/16	&	1.85E-06	&	2.65 	&	2.59E-06	&	2.65 	&	1.27E-06	&	2.90 	&	1.78E-06	&	2.90 	\\
&	l/32	&	2.92E-07	&	2.66 	&	4.09E-07	&	2.66 	&	1.77E-07	&	2.84 	&	2.48E-07	&	2.84 	\\
\hline
$P^3$ &	l/2	&	1.26E-04	&	--	&	1.75E-04	&	--	&	1.16E-04	&	--	&	1.61E-04	&	--	\\
&	l/4	&	6.72E-06	&	4.23 	&	9.36E-06	&	4.23 	&	5.50E-06	&	4.40 	&	7.68E-06	&	4.39 	\\
&	l/8	&	3.24E-07	&	4.37 	&	4.52E-07	&	4.37 	&	2.20E-07	&	4.64 	&	3.07E-07	&	4.64 	\\
&	l/16	&	1.54E-08	&	4.40 	&	2.14E-08	&	4.40 	&	1.09E-08	&	4.34 	&	1.51E-08	&	4.35 	\\
&	l/32	&	8.28E-10	&	4.22 	&	1.15E-09	&	4.22 	&	5.63E-10	&	4.27 	&	7.84E-10	&	4.27 	\\
%\hline
%$P^4$ &	l/2	&	2.27E-05	&	--	&	3.19E-05	&	--	&	1.90E-05	&	--	&	2.63E-05	&	--	\\
%&	l/4	&	7.24E-07	&	4.97 	&	1.00E-06	&	4.99 	&	5.65E-07	&	5.07 	&	7.85E-07	&	5.07 	\\
%&	l/8	&	3.05E-08	&	4.57 	&	4.24E-08	&	4.57 	&	2.84E-08	&	4.31 	&	3.97E-08	&	4.31 	\\
%&	l/16	&	1.05E-09	&	4.86 	&	1.46E-09	&	4.86 	&	6.96E-10	&	5.35 	&	9.70E-10	&	5.35 	\\
%&	l/32	&	3.72E-11	&	4.82 	&	5.20E-11	&	4.81 	&	2.38E-11	&	4.87 	&	3.33E-11	&	4.86 	\\
\hline
\end{tabular}
\end{center}
\end{table}

% Example 5.4: Constant state preservation --------------------------------------------------------------------------------------------------------
\paragraph{Example 5.4}  (Constant state preservation)
The previous examples show that the ALE-DG method on moving meshes maintains the high order accuracy as the DG method on static meshes. The ability of the ALE-DG method to preserve constant states needs to be investigated, too. For this reason the linear advection equation \eqref{LinearProblem}  and the Burgers' equation \eqref{BurgersProblem} are considered with the constant initial condition $u_{0}=1$. We solve these initial value problems with the ALE-DG method on moving triangular meshes with the grid point distribution \eqref{GridpointDistribution}. In Table \ref{GCLtable} the results of the computations are listed and it can be seen that the ALE-DG method numerically satisfies the GCL. This result was expected, since we used a time discretization with an order greater than two and in the Section \ref{Section:GCL} it has been proven that a time discretization of this type is enough to ensure that the method preserves constant states.

\begin{table}[!htb]
\caption{\label{GCLtable} $\mathrm{L}^{2}$-errors for the advection equation and Burgers' equation at time $t=1$ with constant initial condition $u_0=1$ on moving triangular meshes with the grid point distribution \eqref{GridpointDistribution} and cell size $h_0$.} \centering
\begin{center}
\begin{tabular}{|c|ccc|ccc|}
  \hline
        & Advection equation &$u-u_h^M$&                  &  Burgers' equation &$u-u_h^M$                &    \\\cline{2-4}\cline{5-7}
  $h_0$ &$ P^1$   & $ P^2$      &$ P^3$    &$ P^1$     & $ P^2$       &$ P^3$\\\hline
1/2	&	5.71E-16	&	3.72E-15	&	8.65E-15	&	3.03E-16	&	2.57E-15	&	7.35E-15	\\
1/4	&	7.89E-16	&	7.42E-15	&	1.99E-14	&	5.20E-16	&	5.93E-15	&	1.56E-14	\\
1/8	&	2.27E-15	&	1.24E-14	&	3.86E-14	&	1.13E-15	&	8.86E-15	&	2.89E-14	\\
1/16	&	4.21E-15	&	2.47E-14	&	7.88E-14	&	2.44E-15	&	1.75E-14	&	5.86E-14	\\
1/32	&	9.11E-15	&	5.39E-14	&	1.67E-13	&	5.06E-15	&	3.56E-14	&	1.19E-13	\\
\hline
\end{tabular}
\end{center}
\end{table}

To further show the D-GCL of ALE-DG methods, we adopted the meshes $\mathcal{T}_1$ and $\mathcal{T}_2$ of $32$ cells in  Fig. \ref{Mesh2} with recursive refinement as the initial and final meshes. For the linear advection equation with the constant initial condition $u_0(x) = 1$, we show the D-GCL errors at time $T=1.0$ in Table \ref{GCL3} by forward Euler, TVD-RK2 and TVD-RK3 methods respectively. Here, $P^1$ piecewise polynomial space is used in the ALE-DG method. We take time step size $\triangle t=\frac{h_0}{\max(|\boldsymbol{\omega}|)}$ with $\boldsymbol{\omega}=(\frac{(\mathbf{x_1}-\mathbf{x_2})}{T},\frac{(\mathbf{y_1}-\mathbf{y_2})}{T})$ and $(\mathbf{x_1},\mathbf{y_1}),(\mathbf{x_2},\mathbf{y_2})$ are vertices of meshes $\mathcal{T}_1$ and $\mathcal{T}_2$. The numerical results are consistent with the analysis on the D-GCL of ALE-DG methods.

%\begin{table}[!htb]
%\caption{\label{GCL3} $\mathrm{L}^{\infty}$-errors and $\mathrm{L}^{2}$-errors for the advection equation at time $\triangle t$ (one time step) with constant initial condition $u_0=1$ on moving triangular meshes with the grid point distribution in Fig. \ref{Mesh2}.} \centering
%\begin{center}
%\begin{tabular}{|c|cc|cc|cc|}
%  \hline
%        &Forward Euler&               &  TVD-RK2&  &  TVD-RK3       &    \\\cline{2-3}\cline{4-5}\cline{6-7}
%  $N$ &$\mathrm{L}^{\infty}$ & $\mathrm{L}^{2}$  &$\mathrm{L}^{\infty}$  &$\mathrm{L}^{2}$   & $\mathrm{L}^{\infty}$ &$\mathrm{L}^{2}$\\\hline
%32	&	4.00E+00	&	2.22E+00	&	2.66E-15	&	8.12E-16	&	2.44E-15	&	8.30E-16	\\
%128	&	4.00E+00	&	2.22E+00	&	2.89E-15	&	9.29E-16	&	4.00E-15	&	1.10E-15	\\
%512	&	4.00E+00	&	2.22E+00	&	9.33E-15	&	1.63E-15	&	1.07E-14	&	1.98E-15	\\
%2048&	4.00E+00	&	2.22E+00	&	1.91E-14	&	3.06E-15	&	2.49E-14	&	3.61E-15	\\
%8192&	4.00E+00	&	2.22E+00	&	3.80E-14	&	5.74E-15	&	4.66E-14	&	6.87E-15	\\
%
%\hline
%\end{tabular}
%\end{center}
%\end{table}

\begin{table}[!htb]
\caption{\label{GCL3} $\mathrm{L}^{\infty}$-errors and $\mathrm{L}^{2}$-errors for the advection equation at $t=1.0$ with constant initial condition $u_0=1$ on moving triangular meshes with the grid point distribution in Fig. \ref{Mesh2}.} \centering
\begin{center}
\begin{tabular}{|c|cc|cc|cc|}
  \hline
        &Forward Euler&               &  TVD-RK2&  &  TVD-RK3       &    \\\cline{2-3}\cline{4-5}\cline{6-7}
  $N$ &$\mathrm{L}^{\infty}$ & $\mathrm{L}^{2}$  &$\mathrm{L}^{\infty}$  &$\mathrm{L}^{2}$   & $\mathrm{L}^{\infty}$ &$\mathrm{L}^{2}$\\\hline
32	&	1.40E-02	&	8.79E-03	&	3.55E-15	&	1.74E-15	&	1.47E-14	&	8.21E-15	\\
128	&	6.21E-03	&	3.09E-03	&	1.78E-15	&	7.29E-16	&	1.35E-14	&	9.11E-15	\\
512	&	2.72E-03	&	1.30E-03	&	2.33E-15	&	4.65E-17	&	3.57E-12	&	1.26E-12	\\
2048	&	1.34E-03	&	6.22E-04	&	1.03E-11	&	2.21E-12	&	2.39E-11	&	3.39E-12	\\
8192	&	6.67E-04	&	3.07E-04	&	3.06E-11	&	1.00E-11	&	8.45E-11	&	1.77E-11	\\
\hline
\end{tabular}
\end{center}
\end{table}

\begin{figure}
\begin{center}
\includegraphics[width=2.7in]{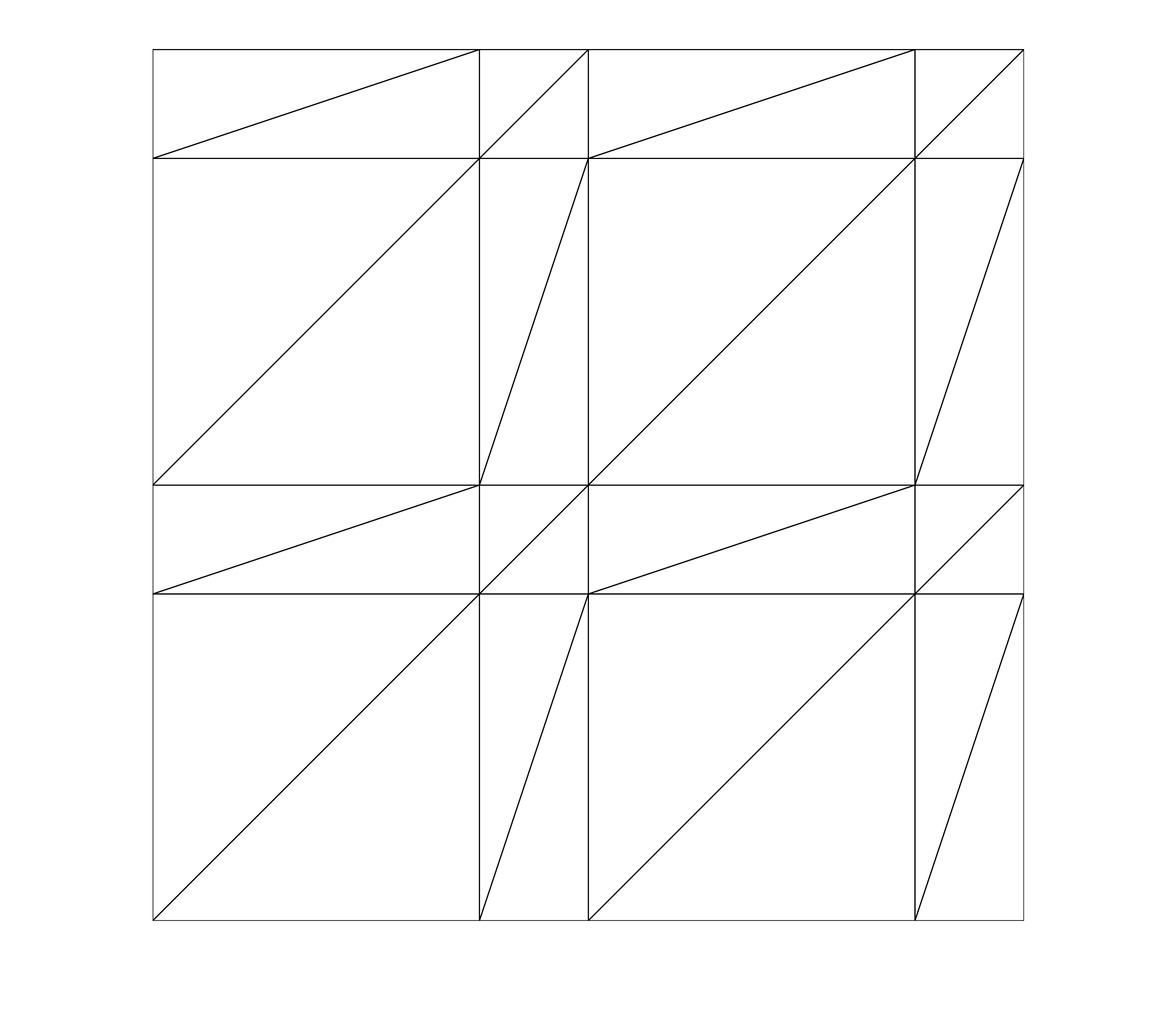}
\includegraphics[width=2.7in]{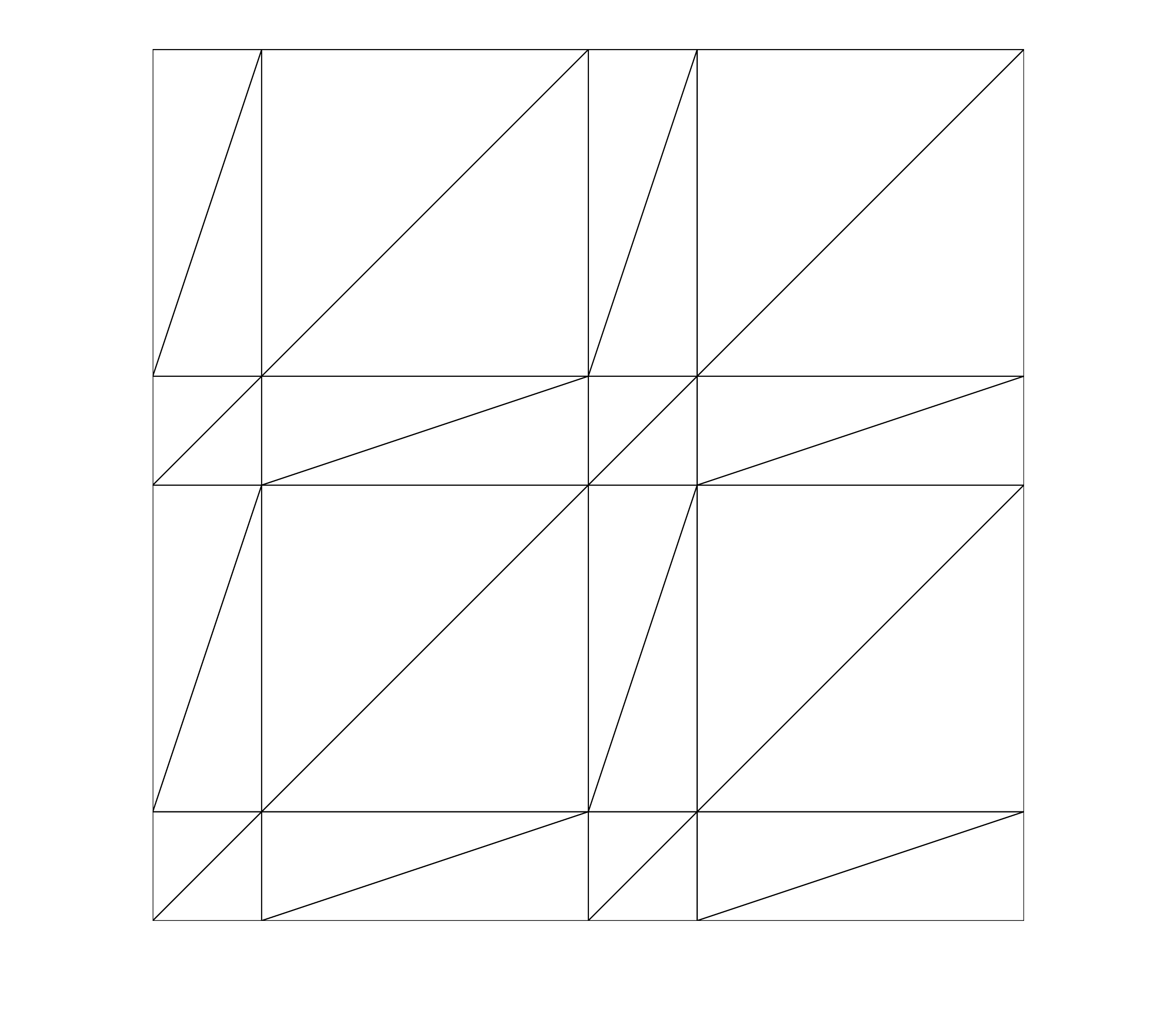}
\caption{The meshes with $32$ cells are used in test of D-GCL for linear advection equation. Left: $\mathcal{T}_1$; Right: $\mathcal{T}_2$. }
\label{Mesh2}
\end{center}
\end{figure}
% Conclusions ------------------------------------------------------------------------------------------------------------------
\section{Conclusions}
In this paper, an ALE-DG method to solve conservation laws in several space dimensions on moving simplex meshes has been developed and analyzed. We began the paper with an analysis of the semi-discrete ALE-DG method and proved the $\mathrm{L}^{2}$-stability. Moreover, we presented a suboptimal a priori error estimate with respect to the $\mathrm{L}^{\infty}\left(0,T;\mathrm{L}^{2}\left(\Omega\right)\right)$-norm, where the suboptimality refers to the approximation properties of the discrete space.

Afterward, the fully-discrete ALE-DG method was investigated. In the context of Total-variation-diminishing Runge-Kutta methods, a relationship between the spatial dimension and the discrete geometric conservation law was elaborated. %More precisely, for the ALE-DG method to solve conservation laws in a $d$-dimensional space, we proved that the geometric conservation law is merely satisfied when a time discretization method of order greater than or equal to $d$ is used.
Furthermore, in two dimensions, second and third order fully-discrete ALE-DG methods were presented. We proved that these methods satisfy the maximum principle when the bound preserving limiter developed by Zhang, Xia and Shu \cite{Zhang2012} is applied. In a future work, it would be worthwhile to investigate if these methods are positive preserving when they are applied to the compressible Euler equations.

Beside our theoretical investigations, several numerical test examples for two moving mesh scenarios have been presented. These examples support our theoretical results and show that the ALE-DG method is numerically stable and uniformly high order accurate. In particular, the two test examples for the compressible Euler equations support the expectation that the ALE-DG method can be also applied to systems of conservation laws even when the development and analysis in this paper has been focused on scalar conservation laws in several space dimensions.

It should be mentioned that we did not use a moving mesh methodology in the numerical examples. The grid point distribution was specified for the calculations. The development of a suitable moving mesh methodology for our ALE-DG method is also a project for a future work.

% Bibliography -----------------------------------------------------------


\begin{thebibliography}{99}

\bibitem{Bellman} R. Bellman. Introduction to matrix analysis, volume 19 of Classics in Applied Mathematics. SIAM, Philadelphia, PA, 2nd Edition, 1987.

\bibitem{Boscheri2017} W. Boscheri and M. Dumbser. Arbitrary-Lagrangian-Eulerian discontinuous Galerkin schemes with a posteriori subcell finite volume limiting on moving unstructured meshes. J. Comput. Phys. 346 (2017), 449-479.

\bibitem{CiarletNF} P. G. Ciarlet. Linear and nonlinear functional analysis with applications, Volume 130, SIAM, Philadelphia, PA, 2013.

\bibitem{CiarletFE} P. G. Ciarlet. The finite element method for elliptic problems, Volume 40 of Classics in Applied Mathematics, SIAM, Philadelphia, PA, 2002.

\bibitem{Cockburn1998} B. Cockburn and C.-W. Shu. The Runge-Kutta discontinuous Galerkin method for conservation laws V, J. Comput. Phys. 141 (1998), 199-224.

\bibitem{Cockburn2001} B. Cockburn and C.-W. Shu. Runge-Kutta discontinuous Galerkin methods for convection-dominated problems, J. Sci. Comput. 16 (2001), 173-261.

\bibitem{Donea2004} J. Donea, A. Huerta, J. P. Ponthot and A. Rodr\'iguez-Ferran. In Encyclopedia of Computational Mechanics, E. Stein, R. De Borst and T. J.R. Hughes (Eds.), Volume. 1: Fundamentals., Chapter 14: Arbitrary Lagrangian-Eulerian Methods, Wiley, 2004.

\bibitem{Etienne2009} S. \'{E}tienne, A. Garon and D. Pelletier. Perspective on the geometric conservation law and finite element methods for ALE simulations of incompressible flow, J. Comput. Phys. 228 (2009), 2313-2333.

\bibitem{Farhat2001} C. Farhat, P. Geuzaine and C. Grandmont. The discrete geometric conservation law and the nonlinear stability of ALE schemes for the solution of flow problems on moving grids, J. Comput. Phys. 174 (2001), 669-694.

\bibitem{Gottlieb1998} S. Gottlieb and C.-W. Shu. Total variation diminishing Runge-Kutta schemes, Math. Comp. 67 (1998), 73-85.

\bibitem{Gottlieb2001} S. Gottlieb, C.-W. Shu, and E. Tadmor, Strong stability-preserving high-order time discretization methods. SIAM review, 43 (2001), 89-112.

\bibitem{Farhat2000} H. Guillard and C. Farhat. On the significance of the geometric conservation law for flow computations on moving meshes, Comput. Method. Appl. M. 190 (2000), 1467-1482.

%\bibitem{Harten} A. Harten. High resolution schemes for hyperbolic conservation laws, J. Comput. Phys. 49 (1983), 357-393.

\bibitem{Jiang1994MC} G.-S. Jiang and C.-W. Shu. On a cell entropy inequality for discontinuous Galerkin methods, Math. Comp. 62 (1994), 531-538.

\bibitem{Klingenberg2016} C. Klingenberg, G. Schn\"{u}cke and Y. Xia. Arbitrary Lagrangian-Eulerian discontinuous Galerkin method for conservation laws: analysis and application in one dimension, Math. Comp. 86 (2017), 1203-1232.

\bibitem{KlingenbergHJ2016HJ} C. Klingenberg, G. Schn\"{u}cke and Y. Xia. An arbitrary Lagrangian-Eulerian local discontinuous Galerkin method for Hamilton-Jacobi equations, J. Sci. Comput., 73 (2017), 906-942.

\bibitem{Kopriva2006} D. A. Kopriva. Metric identities and the discontinuous spectral element method on curvilinear meshes. J. Sci. Comput., 26 (2006), 301-327.

\bibitem{Kopriva2016} D. A. Kopriva, A. R. Winters, M. Bohm and G. J. Gassner. A provably stable discontinuous Galerkin spectral element approximation for moving hexahedral meshes. Comput. Fluids 139 (2016), 148-160.

\bibitem{Kruzkov1970} S. N. Kru\v{z}kov. First order quasilinear equations in several independent variables, Math. USSR-Sbornik, 10 (1970), 217-243.

\bibitem{Kuzmin2010} D. Kuzmin. A guide to numerical methods for transport equations, Lecture, Friedrich-Alexander-Universit\"{a}t Erlangen-N\"{u}rnberg (2010).

\bibitem{Kuzmin2001} D. Kuzmin and S. Turek. Flux correction tools for finite elements, J. Comput. Phys. 175 (2002), 525-558.

\bibitem{Lesoinne1996}  M. Lesoinne and C. Farhat. Geometric conservation laws for flow problems with moving boundaries and deformable meshes, and their impact on aeroelastic computations, Comput. Method. Appl. M. 134 (1996), 71-90.

\bibitem{Lombard1979} C. K. Lombard and P. D. Thomas. Geometric conservation law and its application to flow computations on moving grids, AIAA J., 17 (1979), 1030-1037.

\bibitem{Lomtev1999JCP} I. Lomtev, R. M. Kirby and G. E. Karniadakis. A discontinuous Galerkin ALE method for compressible viscous flows in moving domains, J. Comput. Phys. 155 (1999), 128-159.

\bibitem{Luo2015} J. Luo, C.-W. Shu, and Q. Zhang. A priori error estimates to smooth solutions of the third order Runge –Kutta discontinuous Galerkin method for symmetrizable systems of conservation laws, ESAIM: M2AN 49 (2015), 991-1018.


\bibitem{Mavriplis2006} D. J. Mavriplis and Z. Yang. Construction of the discrete geometric conservation law for high-order time-accurate simulations on dynamic meshes, J. Comput. Phys. 213 (2006), 557-573.

\bibitem{Minoli2011}
C. A. Acosta Minoli and D. A. Kopriva. Discontinuous Galerkin spectral element approximations on moving meshes,  J. Comput. Phys. 230 (2011), 1876-1902.

\bibitem{Nguyen2010JFS} V. T. Nguyen. An arbitrary Lagrangian-Eulerian discontinuous Galerkin method for simulations of flows over variable geometries, J. Fluid. Struct. 26 (2010), 312-329.

\bibitem{Osher1985} S. Osher. Convergence of generalized MUSCL schemes, SIAM J. Numer. Anal. 22 (1985), 947-961.

\bibitem{Persson2009} P. O. Persson, J. Bonet, J. Peraire. Discontinuous Galerkin solution of the Navier-Stokes equations on deformable domains, Comput. Method. Appl. M. 198 (2009), 1585-1595.

\bibitem{Reed1973} W.H. Reed, T.R. Hill. Triangular mesh method for the neutron transport equation, Technical report LA-UR-73-479, Los Alamos Scientific Laboratory, Los Alamos, NM, 1973.

\bibitem{Robinson1991} B. A. Robinson, H. T. Yang and J. T. Batina. Aeroelastic analysis of wings using the Euler equations with a deforming mesh. J. Aircraft 28 (1991), 781-788.

\bibitem{Shu} C.-W. Shu. Total-variation-diminishing time discretizations, SIAM J. Sci. Stat. Comput. 9 (1988), 1073-1084.

\bibitem{Spiteri} R. J. Spiteri and S. J. Ruuth. A new class of optimal high-order strong-stability-preserving time discretization methods, SIAM J. Numer. Anal. 40 (2002), 469-491.

\bibitem{Persson2015}  L. Wang and P. O. Persson. High-order discontinuous Galerkin simulations on moving Domains using ALE formulations and local remeshing and projections, 53rd AIAA Aerospace Sciences Meeting, AIAA SciTech Forum (AIAA 2015-0820).
Available via \url{http://dx.doi.org/10.2514/6.2015-0820 }

\bibitem{Xu2007}
Y. Xu and C.-W. Shu. Error estimates of the semi-discrete local discontinuous Galerkin method for nonlinear convection–diffusion and KdV equations, Comput. Methods Appl. Mech. Eng., 196 (2007), 3805-3822.

\bibitem{Xu2014}
Z. Xu. Parametrized maximum principle preserving flux limiters for high order schemes solving hyperbolic conservation laws: one-dimensional scalar problem, Math. Comp. 83 (2014), 2213-2238.

\bibitem{QZhang2004}
Q. Zhang and C.-W. Shu. Error estimates to smooth solutions of Runge-Kutta discontinuous Galerkin methods for scalar conservation laws, SIAM J. Numer. Anal. 42 (2004), 641-666.

\bibitem{QZhang2006}
Q. Zhang and C.-W. Shu. Error estimates to smooth solutions of Runge-Kutta discontinuous Galerkin method for symmetrizable systems of conservation law, SIAM J. Numer. Anal. 44 (2006), 1703-1720.

\bibitem{QZhang2010}
Q. Zhang and C.-W. Shu. Stability analysis and a priori error estimates of the third order explicit Runge–Kutta discontinuous Galerkin method for scalar conservation laws, SIAM J. Numer. Anal., 48 (2010), 1038-1063.

\bibitem{QZhang2014}
Q. Zhang and Chi-Wang Shu. Error estimates for the third order explicit Runge-Kutta discontinuous Galerkin method for a linear hyperbolic equation in one-dimension with discontinuous initial data, Numer. Math. 126 (2014), 703-740.

\bibitem{Zhang2010} X. Zhang and C.-W. Shu. On maximum principle-satisfying high order schemes for scalar conservation laws, J. Comput. Phys. 229 (2010), 3091-3120.

\bibitem{Zhang2012} X. Zhang, Y. Xia and C.-W. Shu. Maximum-principle-satisfying and positivity-preserving high order discontinuous Galerkin schemes for conservation laws on triangular meshes, J. Sci. Comput. 50 (2012), 29-62.

\bibitem{Zhou2018} L. Zhou, Y. Xia and C-W. Shu. Stability analysis and error estimates of arbitrary Lagrangian-Eulerian discontinuous Galerkin method coupled with Runge-Kutta time-marching for linear conservation laws, ESAIM: Mathematical Modelling and Numerical Analysis, to appear.
%-Math. Model Num.)

\end{thebibliography}
\end{document}